\documentclass[reqno]{amsart}
\usepackage{amssymb,amsmath,amsthm,amstext,amsfonts}
\usepackage[all]{xy}
\usepackage{graphicx,color}
\usepackage[ansinew]{inputenc}
\usepackage{hyperref}
\usepackage{here}
\usepackage{enumitem}
\usepackage[titletoc]{appendix}

\usepackage{caption}

\DeclareCaptionType{mycapequ}[][List of equations]
\hypersetup{ colorlinks   = true, 
urlcolor  = blue, 
linkcolor    = blue, 
citecolor   = blue 
}

\makeatletter \@addtoreset{equation}{section} \makeatother

\renewcommand\thetable{\thesection.\@arabic\c@table}

\theoremstyle{plain}
\newtheorem{maintheorem}{Theorem}

\newtheorem{maincorollary}{Corollary}

\newtheorem{theorem}{Theorem }[section]
\newtheorem{proposition}[theorem]{Proposition}
\newtheorem{lemma}[theorem]{Lemma}
\newtheorem{corollary}[theorem]{Corollary}
\newtheorem{claim}{Claim}
\theoremstyle{definition} \theoremstyle{remark}
\newtheorem{remark}[theorem]{Remark}
\newtheorem{example}[theorem]{Example}
\newtheorem{definition}[theorem]{Definition}

\newtheorem{question}{Question}

\newtheorem{questions}{Question}

\newcommand{\eqdef}{\stackrel{\scriptscriptstyle\rm def.}{=}}

\newcommand{\eps}{\varepsilon}

\newcommand{\orb}{\mathrm{orb}^X}

\newcommand{\Z}{\mathbb{Z}}
\newcommand{\N}{\mathbb{N}}
\newcommand{\R}{\mathbb{R}}

\newcommand{\zero}{\operatorname{Zero}}
\newcommand{\mundo}{\mathfrak{X}^1(M)}
\newcommand{\cent}{\mathfrak{C}}

\newcommand{\information}{{
  \bigskip
  \footnotesize
  	\textbf{Martin Leguil}: \textsc{Department of Mathematics, University of Toronto, 40 St George
St. Toronto, M5S 2E4, ON, Canada} \par\nopagebreak
\textsc{CNRS-Laboratoire de Math\'ematiques d'Orsay, UMR 8628, Universit\'e Paris-Sud 11, Orsay Cedex 91405, France } \par\nopagebreak
  	\textit{E-mail:} \texttt{martin.leguil@utoronto.ca / martin.leguil@math.u-psud.fr}
  	
  	\medskip
	\textbf{Davi Obata}: \textsc{CNRS-Laboratoire de Math\'ematiques d'Orsay, UMR 8628, Universit\'e Paris-Sud 11, Orsay Cedex 91405, France } \par\nopagebreak
  \textsc{Instituto de Matem\'atica, Universidade Federal do Rio de Janeiro, P.O. Box 68530, 21945-970, Rio de Janeiro Brazil}\par\nopagebreak
  \textit{E-mail:} \texttt{davi.obata@math.u-psud.fr}

  \medskip
  \textbf{Bruno Santiago}: \textsc{Instituto de matem\'atica e estat\'istica, Universidade Federal Fluminense, Rua Professor Marcos Waldemar de Freitas Reis, s/n, Bloco G Gabinete 72, 4º andar, Campus do Gragoat\'a, 
  	S\~ao Domingos - Niter\'oi - RJ - CEP: 24.210-201} \par\nopagebreak
  \textit{E-mail:} \texttt{brunosantiago@id.uff.br}
}}

\begin{document}

\title{On the centralizer of vector fields: criteria of triviality and genericity results}

\author{Martin Leguil$^*$}
\thanks{$^*$M.L. was supported by the NSERC Discovery grant 502617-2017 of  Jacopo De Simoi.}

\author{Davi Obata$^\dagger$}
\thanks{$^\dagger$D.O. was supported by the ERC project 692925 NUHGD}

\author{Bruno Santiago$^\ddagger$}
\thanks{$^\ddagger$B.S. was supported by Fondation Louis D-Institut de France (project coordinated by M. Viana)}

\begin{abstract}
	In this paper we study the problem of knowing when the centralizer of a vector field is ``small''. We obtain several criteria that imply different types of ``small'' centralizers, namely \textit{collinear, quasi-trivial} and \textit{trivial}. There are two types of results in the paper: general dynamical criteria that imply one of the ``small'' centralizers above; and genericity results about the centralizer.
	
	Some of our general criteria imply that the centralizer is trivial in the following settings: non-uniformly hyperbolic conservative $C^2$ flows; transitive separating $C^1$ flows; Kinematic expansive $C^3$ flows on $3$ manifolds whose singularities are all hyperbolic.
	
	For genericity results, we obtain that $C^1$-generically the centralizer is quasi-trivial, and in many situations we can show that it is actually trivial.  
\end{abstract}

\date{\today}

\maketitle
\tableofcontents

\section{Introduction}

Given a dynamical system, it is natural to try to understand the symmetries that it may have. Oftentimes, they may give extra information which can be used to understand the dynamical behaviour. For example, towards the end of the 19th century, Lie was able to use the symmetries of some differential equations to derive their solutions, ans it was actually during this work that he introduced the notion of Lie groups. There are several different notions of symmetries that one may consider for a dynamical system. One of them is the so-called \textit{centralizer}, which is the main object of interest in this paper. Let us give a brief account of the study of centralizers in dynamics.

\subsection*{Centralizers of diffeomorphisms}

Let $M$ be a compact riemannian manifold and for each $r\geq 1$ we consider $\mathrm{Diff}^r(M)$ to be the set of $C^r$-diffeomorphisms of $M$. For a given $f\in \mathrm{Diff}^r(M)$ and $s\in [1,r]$ we define its $C^s$-\textit{centralizer} as
\[
\cent^s(f):= \{ g\in \mathrm{Diff}^s(M): f\circ g = g\circ f\}.
\]
In other words, it is the set of diffeomorphisms that commutes with $f$. Observe that some trivial solutions of the equation $f\circ g = g\circ f$ are given by $g= f^n$, for any $n\in \Z$. A natural question is to know when these are the only solutions of such equation. Whenever the centralizer is generated by $f$, we say that $f$ has \textit{trivial centralizer}. We remark that, whenever the centralizer of $f$ is non-trivial, then $f$ embeds into a non-trivial $\Z^2$-action. 

Kopell in her Ph.D. thesis in $1970$ proved that for $r\geq 2$ and when $M$ is the circle $S^1$, there is an open and dense subset of $C^r$-diffeomorphisms with trivial centralizer (see \cite{Kopell}). Motivated by Kopell's result, Smale asked the following question:

\begin{questions}[\cite{Smale0}, \cite{Smale}]
Is the set of $C^r$-diffeomorphisms with trivial centralizer a residual (or generic) subset? That is, does it contain a dense $G_{\delta}$-subset of the space of $C^r$-diffeomorphisms? Is it open and dense?
\end{questions}

This question remains open in this generality, but there are several partial answers. We refer the reader to \cite{BakkerFisher, Burslem, Fisher2008, Fisher2009, PalisYoccoz01, PalisYoccoz02, Plykin, Rocha, Rocha2, RochaVarandas} for some related results in the hyperbolic and partially hyperbolic setting. 

Bonatti-Crovisier-Wilkinson gave a positive answer to Smale's question for the $C^1$-topology. They proved that a $C^1$-generic diffeomorphism has trivial $C^1$-centralizer (see \cite{BonattiCrovisierWilkinson}). From their result, a natural question is to know if for the $C^1$-topology the property of having trivial $C^1$-centralizer is open and dense. It turns out that the answer is no. This is given by Bonatti-Crovisier-Vago-Wilkinson (\cite{BonattiCrovisierVagoWilkinson}), where they proved that any manifold admits a $C^1$-open set $\mathcal{U} \subset \mathrm{Diff}^1(M)$ such that there exists a subset $\mathcal{D} \subset \mathcal{U}$, which is $C^1$-dense in $\mathcal{U}$, with the property that any diffeomorphism $f\in \mathcal{D}$ has non-trivial centralizer.

A great portion of this paper is dedicated to extend the result of Bonatti-Crovisier-Wilkinson in \cite{BonattiCrovisierWilkinson} for flows. As we will see, there are some difficulties that arises when one studies the centralizer of flows, which do not appear for diffeomorphisms. 

\subsection*{Centralizers of flows}

Let us now turn our attention to the symmetries of a continuous dynamical system. Since we will restrict our study to the $C^1$-category, we can represent a flow by the vector field that generates it. Let $\mathfrak{X}^r(M)$ be the set of $C^r$ vector fields of $M$. Recall that a smooth manifold carries a Lie bracket operator $[\cdot,\cdot]$ that acts on $\mathfrak{X}^r(M) \times \mathfrak{X}^r(M)$. For $X,Y \in \mathfrak{X}^r(M)$, it is defined by $[X, Y] := XY-YX$. If $[X,Y] = 0$, we say that $X$ and $Y$ commute. 

Let $X \in \mathfrak{X}^r(M)$ and $1\leq s \leq r$, we define the \textit{$C^s$-centralizer} of $X$ by
\[
\cent^s(X) := \{ Y\in \mathfrak{X}^s(M): [X,Y] = 0\}.
\]
This is the set of vector fields that commute with $X$.

Given $X$, the equation $[X,Y] = 0$ has some trivial solutions. Indeed, for any $c\in \R$, the vector field $Y= cX$ commutes with $X$. More generally, for any function $f\colon M \to \R$ such that $Xf = 0$, then the vector field $Y = fX$ also commutes with $X$. In what follows we will define different types of ``triviality'' for the centralizer of flows.

A $C^r$ vector field $X$ has \textit{$C^s$-collinear centralizer} if for any $Y\in \cent^s(X)$, for any point $x\in M$ the space generated by the vectors $X(x)$ and $Y(x)$ has dimension at most $1$. This definition says that if $Y$ commutes with $X$, then $Y$ has the ``same direction'' of $X$.

Recall that two vector fields commute if and only if their flows commute, that is, for any $t_X, t_Y\in \R$ we have that $X_{t_X} \circ Y_{t_Y}(.) = Y_{t_Y} \circ X_{t_X} (.)$. Two commuting flows induce an $\R^2$-action. If $X$ has collinear centralizer, then the flow generated by $X$ does not embed into a non-trivial $\R^2$-action, that is, there are no orbits of the action with dimension $2$.

A slightly stronger notion of triviality is the following: we say that $X$ has $C^s$-\textit{quasi-trivial centralizer} if for any $Y\in \cent^s(X)$ there is a continuous function $f\colon M\to \R$, which is differentiable along $X$-orbits, such that $Y=fX$.

At last, we say that $X$ has \textit{$C^s$-trivial centralizer} if the centralizer is given by the set $\{cX: c\in \R\}$. Observe that this is the smallest possible centralizer that a vector field may have. It is natural to ask in this context the following version of Smale's questions for the vector field centralizer.

\begin{questions}
\label{q.smale}
Is the set of $C^r$-vector fields with trivial (quasi-trivial or collinear) centralizer a residual (or generic) subset? Is it open and dense?
\end{questions}

There are several works that study the different types of triviality of the vector field centralizer. In 1973, Kato-Morimoto proved that the centralizer of an Anosov flow is quasi-trivial (see \cite{KatoMorimoto}). The main feature used in their proof is a topological property called (Bowen-Walters) expansivity. We remark that there are several different notions of expansivity for flows.

A few years later in \cite{Oka}, Oka extended Kato-Morimoto's result for (Bowen-Walters) expansive flows. This type of expansivity is somehow restrictive, since it implies that every singularity is an isolated point of the manifold.

In \cite{Sad}, Sad in his Ph.D. thesis adapted the remarkable work of Palis-Yoccoz \cite{PalisYoccoz02} for flows. He proved that the triviality of the vector field centralizer holds for an $C^{\infty}$-open and dense subset of $C^{\infty}$ Axiom A vector fields that verify the strong transversality condition. The singularities of an Axiom A flow are dynamically isolated, meaning that they are not contained in a non-trivial transitive set. For these type of flows the singularities do not give any trouble in the proofs of triviality of the centralizer.

Much more recently, in 2018, Bonomo-Rocha-Varandas (\cite{BonomoRochaVarandas}) studied the centralizer for Komuro expansive flows. We remark that Komuro expansivity allows the presence of singularities, which includes for instance Lorenz attractors. They prove the triviality of the centralizer of transitive $C^{\infty}$-Komuro expansive transitive flows whose singularities are hyperbolic and verify a non-resonance condition. 

Bonomo-Varandas proved in \cite{BonomoVarandas} that a $C^1$-generic divergence free vector field has trivial vector field centralizer (they also obtain a generic result for Hamiltonian flows in the same paper). In a different paper, \cite{BonomoVarandasaxioma}, Bonomo-Varandas obtain that $C^1$-generic sectional Axiom A vector fields have trivial vector field centralizer (see the introduction of \cite{BonomoVarandasaxioma} for the definition of sectional Axiom A).

In this paper, there are two types of results: general results which study dynamical conditions on $X$ that imply ``triviality'' of its centralizer, and genericity results. In what follows we will state our main results. 

\subsection*{Quasi-trivial centralizers}

We obtain some easy criteria that imply collinearity of the centralizer. A natural problem is to know when collinearity can be promoted to quasi-triviality. If $Y$ commutes with $X$ and $Y$ is collinear to $X$, it is easy to see that there is a continuous function $f$, defined on regular (or non-singular) points such that $Y= fX$. The problem of going from collinearity to quasi-triviality is a problem of extending continuously the function $f$ to the entire manifold. This is not always the case; indeed, in Section \ref{sec.quasitrivial} we construct an example of a vector field with collinear centralizer which is not quasi-trivial.

Nevertheless, when all the singularities of a $C^1$ vector field are hyperbolic, collinearity can actually be promoted to quasi-triviality:

\begin{maintheorem}\label{res qu tri}
Let $M$ be a compact manifold. If $X\in \mathfrak{X}^1(M)$ has collinear $C^1$-centralizer and all the singularities of $X$ are hyperbolic, then $X$ has quasi-trivial $C^1$-centralizer.
\end{maintheorem}

We stress that we do not require any regularity nor absence of resonance like conditions on the singularity. This is an important improvement compared with previous results \cite{BonomoRochaVarandas,BonomoVarandas,BonomoVarandasaxioma} and \cite{Sad}.  

A significant part of this paper is dedicated to the proof of the $C^1$-genericity of quasi-trivial centralizer. This is given in the following theorem:

\begin{maintheorem}\label{theo d}
Let $M$ be  a compact manifold.
There exists a residual subset $ \mathcal{R} \subset \mathfrak{X}^1(M)$ such that any $X\in \mathcal{R}$ has quasi-trivial $C^1$-centralizer. 
\end{maintheorem}

This result is a version for the vector field centralizer of Bonatti-Crovisier-Wilkinson's result \cite{BonattiCrovisierWilkinson}.

\subsection*{Trivial centralizers}
Next we see in which situations we can conclude the triviality of the centralizer. It is easy to construct examples of vector fields whose centralizer is quasi-trivial  but not trivial. In Section \ref{sec.trivial} we explain how Example \ref{example.separating} has quasi-trivial centralizer, but not trivial. 

The problem of knowing if a vector field with quasi-trivial centralizer has trivial centralizer is reduced to the problem of knowing when an $X$-invariant function\footnote{A function is $f$ is $X$-invariant if $Xf=0$, which amounts to saying that $f\circ X_t = f$, $\forall\, t\in \R$.} is constant. This problem will be studied in Section \ref{sec.trivial}.

Our first criterion to obtain triviality is based on the notion of \textit{spectral decomposition}. We say that $X$ admits a \emph{countable spectral decomposition} if the non-wandering set, $\Omega(X)$, satisfies $\Omega(X) = \displaystyle \sqcup_{i\in \N} \Lambda_i$, where the sets $\Lambda_i$ are pairwise disjoint, each of which is compact, $X$-invariant and transitive, i.e., contains a dense orbit. 

\begin{maintheorem}\label{theo e}
Let $M$ be a compact connected manifold and let $X\in \mathfrak{X}^1(M)$. Assume that all the singularities of $X$ are hyperbolic, that  $X$ admits a countable spectral decomposition and that the $C^1$-centralizer of $X$ is collinear. Then $\mathfrak{C}^1(X)$ is trivial. 
\end{maintheorem}

With the assumption of a very weak type of expansivity, called \textit{separating} (see Definition \ref{def.separating}), we can obtain the following result:
\begin{maintheorem}
If $X$ is a transitive, separating $C^1$ vector field, then $X$ has trivial $C^1$-centralizer. 
\end{maintheorem}

We remark that the separating property is much weaker than Komuro expansiveness. In particular, our results generalize to a much larger class of vector fields the results about centralizers of flows from \cite{KatoMorimoto, Oka, BonomoRochaVarandas}. After this work was completed, Bakker-Fisher-Hasselblatt in \cite{BakkerFisherHasselblatt} were able to prove a similar result in the $C^0$-category. However, their result uses a type of expansiveness stronger than separating, called kinematic expansiveness.

In higher regularity, Pesin's theory in the non-uniformly hyperbolic case and Sard's theorem give us two useful tools to verify triviality of the centralizer. Using Pesin's theory as a tool, we obtain the following result:

\begin{maintheorem}
\label{thm.NUH}
Let $M$ be a compact manifold of dimension $d \geq 2$. 
Let $X\in \mathfrak{X}^2(M)$ be a vector field with finitely many singularities and let $\mu$ be a $X$-invariant probability measure such that $\mathrm{supp}\mu = M$. If $\mu$ is non-uniformly hyperbolic for $X$, then $X$ has trivial $C^1$-centralizer.
\end{maintheorem}

Theorem E can be applied for non-uniformly hyperbolic geodesic flows, like the ones constructed by Donnay \cite{Donnay} and Burns-Gerber \cite{BurnsGerber}. In particular, we obtain that non-uniformly hyperbolic geodesic flows have trivial centralizer.

In dimension three, under higher regularity assumptions, we are also able to obtain triviality, for a slightly stronger notion of expansiveness called \textit{kinematic expansive}, which is stronger than separating (see Definition \ref{Def.kinematic}).

\begin{maintheorem}\label{thm.dim3triviality}
Let $M$ be a compact $3$-manifold and consider $X\in \mathfrak{X}^3(M)$. If $X$ is Kinematic expansive and all its singularities are hyperbolic, then its $C^3$-centralizer is trivial.
\end{maintheorem}

The technique we use in the above theorem, which relies on Sard's Theorem, also leads to a criterion  to obtain triviality from a collinear centralizer of high regularity.   

\begin{maintheorem}\label{thm.CDtrivial}
Let $M$ be a compact, connected Riemannian manifold of dimension $d \geq 1$, and let $X\in \mathfrak{X}^d(M)$.  Assume that every singularity and periodic orbit of $X$ is hyperbolic, that $\Omega(X) = \overline{\mathrm{Per}(X)}$ and that the $C^d$-centralizer of $X$ is collinear. 
Then $X$ has trivial $C^d$-centralizer. 
\end{maintheorem}

This criterion is not sufficient if we want to obtain a generic result, due to the lack of a $C^d$-closing lemma. However, following the arguments of \cite{Mane,Hurley}, we can show that $C^d$-generically the triviality of the $C^d$-centralizer is equivalent to the collinearity of the $C^d$-centralizer. 
\begin{maintheorem}\label{thm.equivanlentcentralizer}
Let  $M$ be a compact, connected Riemannian manifold of dimension $d \geq 1$. There exists a  residual set $\mathcal{R}_T\subset \mathfrak{X}^d(M)$ such that for any $X \in \mathcal{R}_T$, the $C^d$-centralizer $\mathfrak{C}^d(X)$ of $X$ is collinear if and only if it is trivial. 
\end{maintheorem}

In the next section we give all the precise definitions used in these theorems. Let us make a few remarks. In our $C^1$-generic result, we can actually obtain triviality in several scenarios, see Theorem \ref{thm.headinggenericsection}. What is missing to obtain the triviality of the centralizer for a $C^1$-generic vector field is to prove that a $C^1$-generic vector field does not admit any non-trivial $C^1$ first integral (see Section \ref{sec.trivial}). This is a conjecture made by Thom \cite{Thom}. With our result, a complete answer for Question \ref{q.smale} for $C^1$ vector fields is equivalent to answering Thom's conjecture.

A natural direction in this genericity type of results is to understand what happens in the generic case in higher regularity. We conclude this introduction with the following question:

\begin{questions}
Given any manifold $M$ with dimension $\dim M \geq 3$, does there exist $r>1$ sufficiently large and a $C^r$-open set $\mathcal{U} \subset \mathfrak{X}^r(M)$ such that for any $X\in \mathcal{U}$ the $C^s$-centralizer (for some $1\leq s \leq r$) of $X$ is not collinear? 
\end{questions}

\subsection{General notions on vector fields}

Before closing this section we recall some definitions and fix some notations that we shall use throughout the paper. 
Let  $M$ be a  smooth manifold of dimension $d \geq 2$,  which we assume to be compact and boundaryless. We shall also fix once and for all a Riemannian metric in $M$. For any $r \geq 1$, we denote by $\mathfrak{X}^r(M)$ the space of vector fields 
over $M$, endowed with the $C^r$ topology. A property $\mathcal{P}$ for vector fields in $\mathfrak{X}^r(M)$ is called $C^r$-generic if it is satisfied for any vector field in a  \textit{residual set} of $\mathfrak{X}^r(M)$. Recall that $\mathcal{R} \subset \mathfrak{X}^r(M)$ is \textit{residual} if it contains a dense $G_{\delta}$-subset of $\mathfrak{X}^r(M)$. In particular, it is dense in $ \mathfrak{X}^r(M)$, by Baire's theorem. 

In the following, given a vector field $X\in\mathfrak{X}^1(M)$, we denote  by $X_t$ the flow it generates. Recall that for any $Y \in \mathfrak{C}^1(X)$, and for any $s,t\in \R$, we have $Y_s \circ X_t = X_t \circ Y_s$.  Differentiating this relation with respect to $s$ at $0$, we get 
\begin{equation}\label{eq cmommut}
Y(X_t(x))= DX_t(x) \cdot Y(x),\quad \forall\, x \in M.
\end{equation} 
We denote by $\mathrm{Zero}(X):=\{x \in M: X(x)=0\}$ the set of \textit{zeros}, or \textit{singularities}, of the vector field $X$, and we set 
\begin{equation}
\label{def.MX}
M_X:=M-\mathrm{Zero}(X).
\end{equation}
For any $x \in M$ and any interval $I \subset \R$, we also let $X_I(x):=\{X_t(x):t \in I\}$. In particular, we denote by $\orb(x):=X_\R(x)$  the orbit of the point $x$ under $X$.  Note that if $x \in M_X$, then $\orb(x) \subset M_X$ too. \\

Let $X\in\mathfrak{X}^1(M)$ be some $C^1$ vector field. The \textit{non-wandering set} $\Omega(X)$ of $X$ is defined as the set of all points $x \in M$ such that for any open neighborhood $\mathcal{U}$ of $x$ and for any $T >0$, there exists a time $t >T$ such that $\mathcal{U}\cap X_t(\mathcal{U}) \neq \emptyset$.

Let us also recall another weaker notion of recurrence. Given two points $x,y \in M$, we write $x\sim_X y$ if for any $\varepsilon>0$ and $T>0$, there exists an $(\varepsilon,T)$-\textit{pseudo orbit} connecting them, i.e., there exist $n \geq 2$, $ t_1,t_2,\dots,t_{n-1} \in [T,+\infty)$,  and $x=x_1,x_2,\dots,x_{n}=y\in M$, such that $d(X_{t_j}(x_j),x_{j+1})< \varepsilon$, for $j \in \{1,\dots,n-1\}$.  The \textit{chain recurrent set} $\mathcal{CR}(X)\subset M$ of $X$ is defined as the set of all points $x \in M$ such that $x \sim_X x$. Restricted to $\mathcal{CR}(X)$, the relation $\sim_X$ is an equivalence relation. An equivalence class under the relation $\sim_X$ is called a \textit{chain recurrent class}: $x, y \in \mathcal{CR}(X)$ belong to the same chain recurrent class if $x \sim_X y$. In particular, chain recurrent classes define a partition of the chain recurrent set $\mathcal{CR}(X)$.

A point $x\in M$ is \emph{periodic} if there exists $T>0$ such that $X_T(x) = x$. The set of all periodic points is denoted by $\mathrm{Per}(X)$, observe that we are also including the singularities in this set.  

An $X$-invariant compact set $\Lambda$ is \textit{hyperbolic} if there is a continuous decomposition of the tangent bundle over $\Lambda$, $T_{\Lambda}M = E^s \oplus \langle X \rangle \oplus E^u$ into $DX_t$-invariant sub-bundles that verifies the following property: there exists $T>0$ such that for any $x\in \Lambda$, 
\[
\|DX_T(x)|_{E^s_x}\| < \frac{1}{2}, \textrm{ and } \|DX_{-T}(x)|_{E^u_x}\|< \frac{1}{2}.
\]   
A periodic point $x\in \mathrm{Per}(X)$ is \textit{hyperbolic}  if $\orb(x)$ is a hyperbolic set. Let $\gamma$ be a hyperbolic periodic orbit. We denote by $W^s(\gamma)$ the \textit{stable manifold} of the periodic orbit $\gamma$, which is defined as the set of points $y\in M$ such that $d(X_t(y),\gamma) \to 0$ as $t\to + \infty$. We define in an analogous way the \textit{unstable manifold} of $\gamma$. It is well known that the stable  and unstable manifolds are $C^1$-immersed submanifolds. A hyperbolic periodic orbit is a \textit{sink} if the unstable direction is trivial. It is a \textit{source} if the stable direction is trivial. A hyperbolic periodic orbit is a \textit{saddle} if it is neither a sink nor a source. For a hyperbolic periodic point $p$ we defined its \emph{index} by $\mathrm{ind}(p) := \dim E^s$.

{\bf Organization of the paper: }
The structure of this paper has two parts. The first part deals with general criteria for collinearity, quasi-triviality and triviality of the centralizer (Sections \ref{section collineaire}, \ref{sec.quasitrivial} and \ref{sec.trivial}). The second part deals with our generic results (Sections \ref{section generic} and \ref{proof generic und}).
In Section~\ref{section collineaire} we state and prove some criteria for collinear centralizer. In Section \ref{sec.quasitrivial} we prove Theorem \ref{res qu tri}. Theorems \ref{theo e}, \ref{thm.NUH}, \ref{thm.dim3triviality}, \ref{thm.CDtrivial} and \ref{thm.equivanlentcentralizer} are proved in Section \ref{sec.trivial}. Finally in Section \ref{section generic} we begin the proof of Theorem \ref{theo d}, which is completed in Section~\ref{proof generic und}.\\

{\bf Acknowledgments:} The authors would like to thank Javier Correa, Sylvain Crovisier, Alexander Arbieto, Federico Rodriguez-Hertz, Anna Florio, Martin Sambarino, Rafael Potrie, 
 and Adriana da Luz for useful conversations. We thank the anonymous referee for his suggestions and comments which helped to improve the presentation of this paper. During part of the preparation of this work, B.S. was supported by Marco Brunella's post-doctoral fellowship. We thank the Brunella family for their generous support.

\section{Collinearity}\label{section collineaire}

In this section we prove two main criteria which imply collinear centralizer. The first criterion appeals to the topological dynamics of the flow and is a very weak form of expansiveness called \textit{separating} (see Definition \ref{def.separating}). For instance, all the usual expansiveness-like notions for flows (Bowen-Walters expansive or Komuro expansive) imply that the flow is separating, see \cite{Artigue} for a throughout discussion. Our second criterion appeals to infinitesimal behaviour of the flow. We show that if the derivative is hyperbolic along critical elements (zeros and periodic orbits), the critical elements are dense in the chain recurrent set and along the wandering trajectories the flow has the unbounded normal distortion, then the centralizer is collinear.

Recall that $M$ is a compact Riemannian manifold. Let $r,k \geq 1$ be positive integers. 
Given $x\in M$ and $u,v\in T_xM$ we denote by $\langle u,v\rangle$ the subspace spanned by $u$ and $v$ in $T_xM$. 
\begin{definition}[Collinear centralizer]
	\label{def.collinear}
	We say that $X\in\mathfrak{X}^r(M)$ has a \emph{collinear} $C^k$-centralizer if 
	$$\dim\langle X(x),Y(x)\rangle\leq 1,$$ 
	for every $x\in M$ and every $Y\in\mathfrak{C}^k(X)$.
\end{definition} 

We have the following elementary result:

\begin{lemma}\label{lemme facile}
	Let $X\in\mathfrak{X}^r(M)$ and assume that the vector field $Y\in\mathfrak{C}^k(M)$ satisfies $\dim\langle X(x),Y(x)\rangle\leq 1$, 
	for every $x\in M$. Then, there exists a function $f \in C^s(M_X,\R)$, with   $s:=\min\{r,k\}$, such that 
	$$
	Y(x)=f(x)X(x),\quad \forall\, x \in M_X.
	$$
	Moreover,  the function $f$ is $X$-invariant, i.e.,
	$$
	f(X_t(x))=f(x),\quad \forall\, x \in M_X,\ \forall\, t \in \R.
	$$
\end{lemma}

\begin{proof}
	Let us denote by $(\cdot,\cdot)$ the   scalar product associated to the Riemannian structure on $M$. For any $x\in M_X$ and for any $v \in T_x M$, we set $\pi_X(x,v):=\frac{(X(x),v)}{(X(x),X(x))}$. In particular,  $\pi_X(x,v)X(x)$ is the orthogonal projection of the vector $v$ on the direction spanned by $X(x)$. Let $Y\in\mathfrak{C}^k(M)$ be a vector field that satisfies $\dim\langle X(x),Y(x)\rangle\leq 1$.  The function $f\colon M_X\to \R$, $x \mapsto \pi_X(x,Y(x))$ is of class $C^s$, with $s=\min\{r,k\}$. Moreover, by the collinearity of the vector fields $X$ and $Y$, we have $Y=fX$.
	
	By \eqref{eq cmommut}, it holds   $Y(X_t(\cdot))= DX_t \cdot Y(\cdot)$. Therefore, for any $x \in M_X$ and for any  $t \in \R$, we have
	$$
	f(X_t(x))X(X_t(x))=DX_t(x)  \cdot (f(x) X(x))=f(x) DX_t(x) \cdot X(x)=f(x) X(X_t(x)),
	$$
	where the last equality follows from \eqref{eq cmommut}, with $X$ in place of $Y$. 
	Since $X(X_t(x)) \neq 0$, we obtain $f(X_t(x))=f(x)$, which concludes the proof. 
\end{proof}

The following definition is a very weak form of expansiveness for flows. 

\begin{definition}
	\label{def.separating}
	A vector field $X\in\mathfrak{X}^1(M)$ is \emph{separating} if there exists $\eps>0$ such that the following holds: 
	if $d(X_t(x),X_t(y))<\eps$ for every $t\in\R$, then $y\in \orb(x)$. 
\end{definition}
We shall now prove the following topological criterion for collinearity, which generalises \cite{Oka}. The idea is very simple: the non collinearity of the centralizer gives rise to a continuum of non-separating trajectories.    

\begin{proposition}\label{premier theorem}
	If $X\in \mathfrak{X}^1(M)$ is separating, then $X$ has collinear $C^1$-centralizer.
\end{proposition}
\begin{proof}
	Let $X\in \mundo $ be a separating vector field with separating constant $\varepsilon>0$ and suppose that there exists $Y\in \cent^1(X) $ that is not collinear to $X$. Thus there is a point $x\in M$ such that $\dim \langle X(x), Y(x) \rangle =2$. 
	
	We fix $\delta>0$ so that for each $s\in (-\delta,\delta) $, we have $d_{C^0}(Y_s,\mathrm{id}) <\varepsilon$, and $\dim \langle Y(Y_s(x)), X(Y_s(x))\rangle = 2$. Consider a point $y=Y_s(x)$, for some $s\in(-\delta,\delta)$. Since $Y\in\mathfrak{C}^1(X)$, we have $X_t\circ Y_s=Y_s\circ X_t$, for every $s,t$, and thus   
	\begin{equation}
	\label{e.proximidade}
	d(X_t(y),X_t(x))=d(X_t(Y_s(x)),X_t(x)) = d(Y_s(X_t(x)),X_t(x)) < \varepsilon,\:\forall\: t\in \R,
	\end{equation}
	because $d_{C^0}(Y_s,\mathrm{id}) <\varepsilon$ due to our choice of $\delta$. To obtain a contradiction from \eqref{e.proximidade} it remains to prove that one can choose $s\in(-\delta,\delta)$ so that $y=Y_s(x)$ does not belong to the orbit of $x$.
	
	For that we apply the flowbox theorem \cite{PalisdeMelo}. Let $(\varphi, U)$  be a small flowbox for the vector field $X$ around the point $x$, that is, $\varphi\colon M \supset U  \to W \subset \R^d=\R^{d-1}\times \R$ is a local chart such that $\varphi_*X = (0,1)$. As $\{t\in\R : X_t(x)\in U\}\subset\R$ is an open set, there exists a countable 2-by-2 disjoint collection of intervals $\{I_j\}\subset\R$ so that 
	\[\{t\in\R : X_t(x)\in U\}=\bigcup_{j=1}^{\infty}I_j.\]
	Consider the set $J\eqdef\{t\in\R:\exists\:s\in(-\delta,\delta): X_t(x)=Y_s(x)\}$. Notice that, as $Y(Y_s(x))\neq 0$ for every $|s|<\delta$, then for each $t\in J$, there exists a unique $s\in(-\delta,\delta)$ so that $X_t(x)=Y_s(x)$. 
	
	We claim that $\#J\cap I_j\leq 1$, for every $j\in\N$. Let us show how to finish the proof from this claim: it implies that $J$ is at most countable, and so the there are at most countably many $s$ so that $y=Y_s(x)$ belongs to the $X$-orbit of $x$, which concludes. 
	
	We are left to prove the claim. Consider the decomposition  $\R^{d-1}\times \R$ of $\R^d$ into \textit{vertical} and \textit{horizontal} components, respectively. We have that $D\varphi(Y_s(x))\cdot Y(Y_s(x))= (Y_1(s), Y_2(s))$, with $Y_1(s) \in \R^{d-1}$ and $Y_2(s) \in \R$. Since $\varphi_*X$  and $\varphi_*Y$ are not collinear at $\varphi(x)$ and since $\varphi_*X=(0,1)$, there exists $\rho>0$ so that, for $\delta$ sufficiently small,  $\|Y_1(s)\|>\rho$, for every $|s|<\delta$. This implies that the curve
	\[s\in(-\delta,\delta)\mapsto\:\varphi(Y_s(x))\]
	meets each horizontal line $\{z\}\times\R$ at most once. As each orbit segment $X_{I_j}(x)$ is sent by $\varphi$ into horizontal segments we deduce that $\#J\cap I_j\leq 1$. This completes the proof.  
	%
	%
	%
\end{proof}

It is important to notice that the separating property is too weak to imply the quasi-triviality of the centralizer, as the simple example below demonstrates. 

\begin{example}
	\label{example.separating}
	Fix two positive real numbers $0<a<b$ and consider the annulus on $\mathbb{R}^2$ given by $A:=
	\{(x,y) \in \mathbb{R}^2: a\leq \|(x,y)\| \leq b \}$. Using polar coordinates $(r, \theta)$ on $A$, we consider the vector field $X(r,\theta) := \frac{\partial}{\partial \theta}$. Observe that every orbit of $X$ is periodic with different period. It is easy to see that this flow is separating. 
\end{example}
\begin{figure}[h!]
	\begin{center}
		\includegraphics [width=9cm]{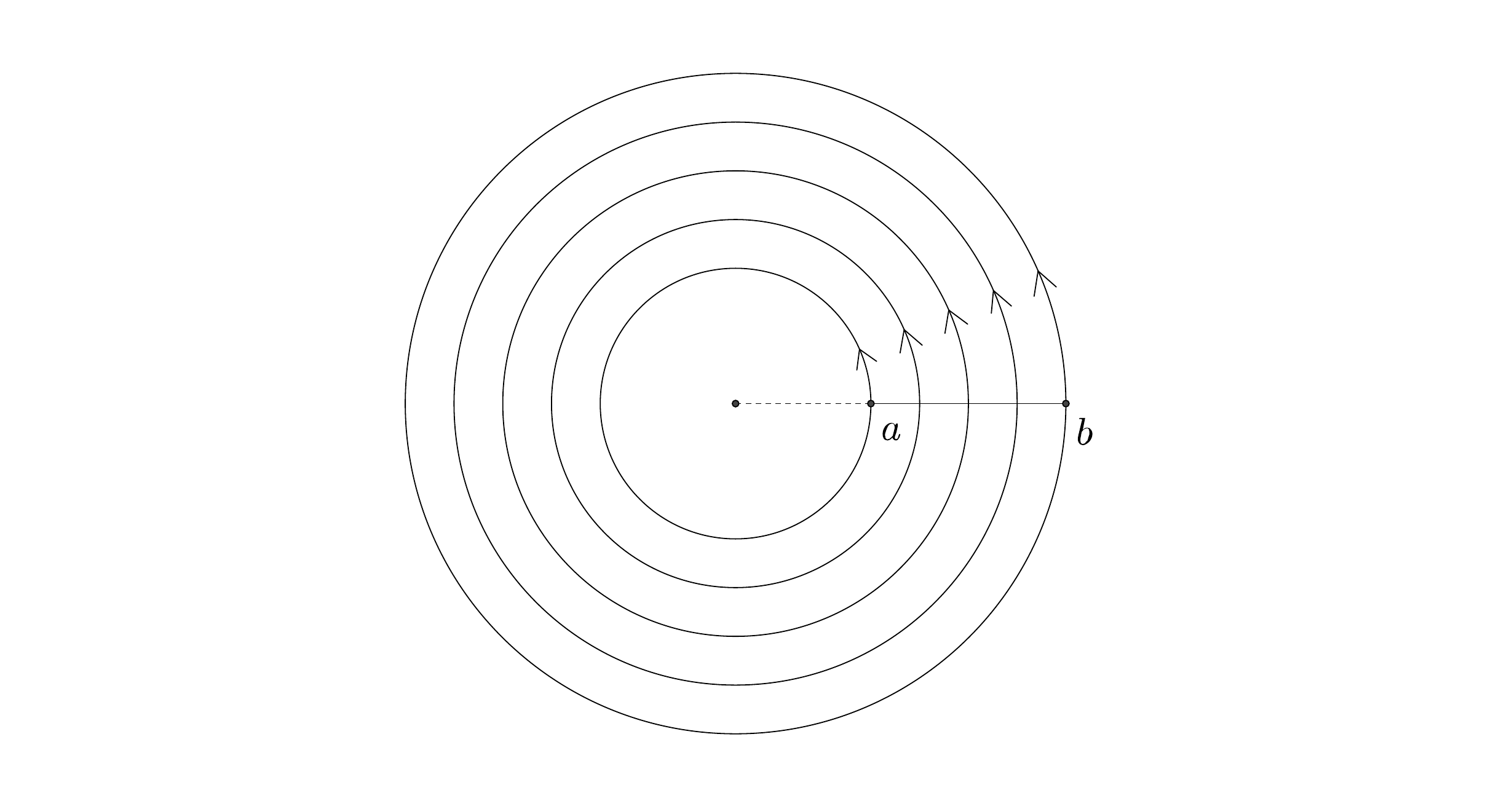}
	\end{center}
	\caption{Example \ref{example.separating}.}
\end{figure}
Given any smooth function $f\colon [a,b]\to\R$, the vector field $Y(r,\theta)=f(r)X(R,\theta)$ commutes with $X$ but is not a constant multiple of $X$. Thus, in view of Lemma~\ref{lemme facile}, the conclusion of Proposition~\ref{premier theorem} is optimal.

We remark that the separating property is not generic (see Appendix \ref{append sep}). So to obtain that the $C^1$-centralizer of a $C^1$-generic vector field is collinear we will need another criterion.

To overcome this, we shall define the notion of unbounded normal distortion. Let $X\in \mathfrak{X}^1(M)$ and let $M_X:=M-\mathrm{Zero}(X)$ be as in (\ref{def.MX}). Over $M_X$ we may consider the normal vector bundle $N_X$ defined as $N_{X,p} := \langle X(p)\rangle^{\perp}$, for $p\in M_X$, where $\langle X(p)\rangle^{\perp}$ is the orthogonal complement of the direction $\langle X(p)\rangle$ inside $T_pM$. Let $\Pi^X\colon TM_X \to N_X$ be the orthogonal projection on $N_X$. On $N_X$ we have a well defined flow, called the \emph{linear Poincar\'e flow}, which is defined as follows: for any $p \in M_X$, any $v\in N_{X,p}$, and $t\in \mathbb{R}$, the image of $v$ by the linear Poincar\'e flow is
\begin{equation}\label{eq.LPF}
P^X_{p,t}(v) := (\Pi^X_{X_t(p)} \circ DX_t(p))\cdot v.
\end{equation}  

The key criterion to study the centralizer of $C^1$-generic vector fields is based on the following property.

\begin{definition}[Unbounded normal distortion]
	\label{def.und}
	Let $X\in \mathfrak{X}^1(M)$ be a $C^1$ vector field. We say that $X$ verifies the \emph{unbounded normal distortion} property if the following holds: there exists a dense subset $\mathcal{D} \subset M- \mathcal{CR}(X)$, such that for any $x\in \mathcal{D}$, $y\in M-\mathcal{CR}(X)$ such that $y\notin \orb(x)$ and $K \geq 1$, there is $n\in (0,+\infty)$, such that
	\[
	|\log \det P^X_{x,n} - \log \det P^X_{y,n}| > K.
	\] 
\end{definition}

This is an adaptation for flows of the definition of unbounded distortion used in \cite{BonattiCrovisierWilkinson} to prove the triviality of the $C^1$-centralizer of a $C^1$-generic diffeomorphism. Using this property we obtain the following proposition. 

\begin{proposition}\label{res collin}
	Let $X\in \mathfrak{X}^1(M)$. Suppose that $X$ verifies the following properties:
	\begin{itemize}
		\item $X$ has unbounded normal distortion;
		\item every singularity and periodic orbit of $X$ is hyperbolic;
		\item $\mathcal{CR}(X) = \overline{\mathrm{Per}(X)}$.
	\end{itemize}
	Then $X$ has collinear $C^1$-centralizer.
\end{proposition}

Before giving the proof of Proposition~\ref{res collin} we give a criterion of collinearity which assumes that the derivative ``blows up" the size of any vector transverse to the flow direction. The proof, which is simple, will be used several times, for instance in the course of proving Proposition~\ref{res collin} and Theorem \ref{thm.NUH}. 

\begin{proposition}
\label{thm.expcollinear}
Let $X\in \mathfrak{X}^1(M)$. Suppose that $X$ verifies the following condition: there exists a dense set $\mathcal{D} \subset M$ such that for any $x\in \mathcal{D}$ and any non zero vector $v\in T_xM- \langle X(x) \rangle$, it holds
\[
\|DX_t(x)\cdot v\| \to +\infty, \textrm{ for $t\to +\infty$ or $t\to -\infty$.}
\]
Then $X$ has collinear centralizer.
\end{proposition}
\begin{proof}
Let $Y\in \mathfrak{C}^1(X)$. Then, by \eqref{eq cmommut}, for any $x\in M$, and $t \in \R$, it holds
\[
Y(X_t(x)) = DX_t(x)\cdot Y(x).
\]
Assume that $Y(x)$ is not collinear to $X(x)$. Since this is an open condition, we can take $x$ belonging to the set $\mathcal{D}$. By compactness of $M$, we also have $\sup_{p\in M} \|Y(p)\| < +\infty$. However, by hypothesis,
\[
\|DX_t(x)\cdot Y(x)\| \to + \infty, \textrm{ for $t\to +\infty$ or $t\to -\infty$,}
\]
which is a contradiction.
\end{proof}

Some examples of vector fields that verify the conditions of Proposition \ref{thm.expcollinear} globally in the manifold are non-uniformly hyperbolic divergence-free vector fields, such as suspensions of \cite{katok} and quasi-Anosov flows \cite{robinho}. However, as the proof demonstrates in any region where the derivative behaves as in the statement we can conclude that any element of the centralizer must be collinear with $X$ in that region. This observation, as we mentioned above, will be used several times in this paper.

We also remark that the global assumptions in Propositions \ref{premier theorem} and \ref{thm.expcollinear} are not generic (see Appendix \ref{append sep}). Therefore, to overcome this and to obtain that the $C^1$-centralizer of a $C^1$-generic vector field is collinear we will need the criterion given by Proposition \ref{res collin}.

Let us  now give the proof of Proposition~\ref{res collin}.
\begin{proof}[Proof of Proposition \ref{res collin}]
Let $X\in \mathfrak{X}^1(M)$ be a vector field with the unbounded normal distortion property and let $\mathcal{D}\subset M-\mathcal{CR}(X)$ be the set given in Definition \ref{def.und}. Take $Y\in \mathfrak{C}^1(X)$. Assume by contradiction that $Y$ is not collinear with $X$ on $M-\mathcal{CR}(X)$. The set of points $x \in M$ such that $X(x)$ and $Y(x)$ are non-collinear is open,  hence by density of the set $\mathcal{D}$, there exists a point $x\in \mathcal{D}$ such that $Y(x)$ and $X(x)$ are not collinear. 

By the same argument as in the proof of Proposition \ref{premier theorem}, we can always find $s>0$ arbitrarily close to $0$ such that $Y_s(x) \notin \mathrm{orb}^X(x)$. Observe that for any $t\in \mathbb{R}$, it holds
\[
|\det P^X_{Y_s(x),t}| = |\det\Pi^X_{X_t(Y_s(x))} .\det DX_t(Y_s(x))|_{N_{X,Y_s(x)}}|.
\]

Since $X$ commutes with $Y$, we have that 
\begin{equation}
\label{eq.undsimmetry}
DX_t(Y_s(x)) = DY_s(X_t(x))\circ DX_t(x)\circ (DY_s(x))^{-1}.
\end{equation}
Using the coordinates $N_X \oplus \langle X \rangle$ on $TM_X$, for each $s\in \mathbb{R}$, we obtain  a linear map $L_{s,x}\colon N_{X,x} \to \langle X(x) \rangle$ such that
\[
(DY_s(x))^{-1}(N_{X,Y_s(x)}) = \mathrm{graph}(L_{s,x}).
\] 
Furthermore, $\|L_{s,x}\|$ can be made arbitrarily small as $s\to 0$, since $Y_s$ is $C^1$-close to the identity. Using the coordinates $N_{X,x} \oplus \langle X(x) \rangle$, any vector $v\in \mathrm{graph}(L_{s,x})$ can be written as $v= (v_N, L_{s,x}(v_N))$, where $v_N := \Pi^X_x(v)$. For any such vector $v$, for each $t\in \mathbb{R}$ and using the coordinates $N_{X,X_t(x)} \oplus \langle X(X_t(x)) \rangle$, we have
\begin{equation}
\label{eq.und1}
\displaystyle DX_t(x) \cdot v = \left(P^X_{x,t} (v_N), L_{s,x}(v_N)\frac{\|X(X_t(x))\|}{\|X(x)\|} + \left( DX_t(x)\cdot  v_N, \frac{X(X_t(x))}{\|X(X_t(x))\|}\right) \right),
\end{equation}
where $(\cdot,\cdot)$ inside the second coordinate of the right side of (\ref{eq.und1}) denotes the  scalar product given by the Riemannian structure. 

On the other hand, for any vector $v_N \in N_{X,x}$ and any $t\in \mathbb{R}$, we have
\begin{equation}
\label{eq.und2}
DX_t(x)\cdot v_N = \left( P^X_{x,t}( v_N), \left( DX_t(x) \cdot v_N, \frac{X(X_t(x))}{\|X(X_t(x))\|}\right) \right).
\end{equation}

Set $c:=\|X(x)\| >0$, and let $\tilde{c}\geq 1$ be a constant such that $\sup_{p \in M}\|X(p)\|<\tilde{c}$. For any vector $v_N\in N_{X,x}$, we obtain
\[
|L_{s,x}(v_N)|\frac{\|X(X_t(x))\|}{\|X(x)\|} < \|L_{s,x}\|\cdot \|v_N\| \frac{\tilde{c}}{c},
\]
which can be made arbitrarily close to $0$ by taking $s$ small enough. This holds for any $t\in \mathbb{R}$. Hence, comparing \eqref{eq.und1} and \eqref{eq.und2} we conclude that $DX_t(x)|_{\mathrm{graph}(L_{s,x})}$ is arbitrarily close to $DX_t(x)|_{N_{X,x}}$, for any $t\in \mathbb{R}$.

By \eqref{eq.undsimmetry}, we obtain   
\[\arraycolsep=1.2pt\def\arraystretch{1.6}
\begin{array}{l}
\left|\det P^X_{Y_s(x),t}\right|=\\
 \big|\det \Pi^X_{Y_s(X_t(x))}|_{DY_s(X_t(x))DX_t(x)\cdot \mathrm{graph}(L_{s,x})}\big|\cdot \left|\det DY_s(X_t(x))|_{DX_t(x)\cdot \mathrm{graph}(L_{s,x})}\right|\cdot \\
 \cdot \left|\det DX_t(x)|_{\mathrm{graph}(L_{s,x})}\right|\cdot \left|(\det (DY_s(x))^{-1}|_{N_{X,Y_s(x)}}\right|=:A\cdot B\cdot C\cdot D.
\end{array}
\]
Observe that 
\[
\left| \det P^X_{x,t}\right| = \left|\det \Pi^X_{X_t(p)}|_{DX_t(x)N_{X,x}}\right|.\left| \det DX_t(x)|_{N_{X,x}}\right| =: \text{I}\cdot \text{II}.
\]

Notice that $B$ and $D$ are arbitrarily close to $1$ if $s\in \mathbb{R}$ is small enough. By our previous discussion, for any $t\in \mathbb{R}$ the value of $C$ is arbitrarily close to the value of II, for $s$ sufficiently small. 

Our previous discussion also implies that $DY_s(X_t(x))DX_t(x)\cdot \mathrm{graph}(L_{s,x})$ is close to $DX_t(x)\cdot N_{X,x}$, since $Y_s(X_t(x))$ is close to $X_t(x)$. Thus, the value of $A$ can be made arbitrarily close to the value of I, for $s\in \mathbb{R}$ small enough. Hence, we can take $s$ small such that $Y_s(x) \notin \orb(x)$ and 
\[
\displaystyle \frac{1}{2} < \frac{\left|\det P^X_{Y_s(x),t}\right|}{\left| \det P^X_{x,t}\right|} <2,\textrm{ for any $t\in \mathbb{R}$}.
\]
This is a contradiction with the unbounded normal distortion property. We conclude that any vector field $Y\in \mathfrak{C}^1(X)$ verifies that $Y|_{M-\mathcal{CR}(X)}$ is collinear to $X|_{M-\mathcal{CR}(X)}$. 

Suppose that for some $x\in \mathcal{CR}(X)$ we have that $Y(x)$ is not collinear to $X(x)$. Since this is an open condition and the hyperbolic periodic points are dense in $\mathcal{CR}(X)$, we can suppose that $x$ is a periodic point. By a calculation similar to the one made in the proof of Proposition \ref{thm.expcollinear} we would then have that $\|Y(X_t(x))\| \to + \infty$ for  $t\to +\infty$ or $t\to -\infty$, which contradicts the fact that $\sup_{p\in M} \|Y(p)\| < +\infty$. Thus we have that $Y$ is also collinear to $X$ on $\mathcal{CR}(X)$.  
\end{proof}

\section{Quasi-triviality}
\label{sec.quasitrivial}

This section has two parts. In the first part we construct an example of a vector field whose centralizer is collinear but not quasi-trivial. In the second part we prove that under the condition that every singularity is hyperbolic we can promote the collinearity to quasi-triviality.

\subsection{Collinear does not imply quasi-trivial} Let us recall in precise terms the notion of quasi-triviality.
\begin{definition}[Quasi-trivial centralizer]
	\label{def.quasitrivial}
	Given two positive integers  $1 \leq k \leq r$, we say that $X\in\mathfrak{X}^r(M)$ has a \emph{quasi-trivial} $C^k$-centralizer if for every $Y\in\mathfrak{C}^k(X)$, there exists a $C^1$ function $f\colon M\to\R$ such that $X\cdot f\equiv 0$ and $Y(x)=f(x)X(x)$, for every $x\in M$.
\end{definition}

This notion was referred to as \textit{unstable centralizer} in the work of \cite{Oka}, in the $C^0$ category. We use the same terminology as \cite{BonomoRochaVarandas}. 

It is easy to see that quasi-trivial centralizer implies collinear centralizer. Moreover, arguing as in Lemma \ref{lemme facile}, if $X\in \mathfrak{X}^r(M)$ has a quasi-trivial $C^k$-centralizer, then for any $Y\in \mathfrak{C}^k(X)$, the function $f$ in Definition \ref{def.quasitrivial} is in fact  of class $C^k$ in restriction to $M_X$.

On the other hand, Lemma~\ref{lemme facile} also shows that to obtain a quasi-trivial centralizer from a collinear centralizer is an issue of knowing whether an invariant function $f\colon M_X\to\R$ admits a $C^1$ extension to the set $\zero(X)$. In full generality this is not always possible, as we demonstrate below.

\begin{example}
\label{exem.notcolinear}	
In this example we shall exhibit a vector field $X\in\mathfrak{X}^{\infty}(\mathbb{T}^3)$ so that $\mathfrak{C}^1(X)$ is collinear but not quasi-trivial. The construction produces a vector field with an uncountable spectral decomposition since it is a fiberwise dynamics. Moreover, the restriction on almost each fiber is separating but the separation is taking longer and longer times to take place as the fibers converge to some singular fiber, where the vector field is identically zero. This property will be used to show that, when comparing different fibers, most orbits will separate. We push a little further the argument given in Proposition~\ref{premier theorem} to show that even this very weak version of the \emph{separating} property is still enough to imply that the centralizer is collinear. By making the singular fiber highly degenerate, we shall obtain a centralizer which is not quasi trivial.    

We consider $\mathbb{S}^1=\R/\Z$, and $\mathbb{T}^3=\mathbb{S}^1\times\mathbb{T}^2$ endowed with the maximum distance. To start the construction, let $V\in\mathfrak{X}^{\infty}(\mathbb{T}^2)$ generate an irrational flow. We assume that the inclination of $V$ is smaller than $1$. This will simplify some estimations later. Fix $p=(1/2,1/2)\in\mathbb{T}^2$ and consider a smooth function $\psi\colon\mathbb{T}^2\to[0,1]$ such that $\psi^{-1}(0)=\{p\}$. Let $Z\in\mathfrak{X}^{\infty}(\mathbb{T}^2)$ be defined by $Z=\psi V$. As it is described in Example $2.8$ in \cite{Artigue}, $Z$ is separating. Take any pair of smooth functions $f,g\colon\mathbb{S}^1\setminus\{1/2\}\to[1,+\infty)$ satisfying 
$$f(s)=\frac{1}{(1-2s)^2}\:\:\:\textrm{and}\:\:\:\:g(s)=\frac{1}{(1-4s^2)^2},$$
for $|s-1/2|<1/4$. We assume that for $|s-1/4|\geq1/4$ both functions are positive. In this way, they diverge to $+\infty$ when $s\to 1/2$, but the function $\frac{f}{g}$ extends smoothly to $\mathbb{S}^1$. 
Extend $Z$ to $\mathbb{T}^3$ by $Z(s,x)=(0,Z(x))$.

Define the vector field $X(s,x)=\frac{1}{g(s)}Z(s,x)$. By Theorem $3.8$ in \cite{Artigue} the restriction of $X$ to each fiber $\{s\}\times\mathbb{T}^2$, with $s\neq1/2$, is separating. The separation constant decreases to $0$ as $s\to1/2$ though. Notice that $X$ is not separating, as it has infinitely many singularities. 

Nevertheless, we claim that $X$ has collinear $C^1$ centralizer.

To prove this claim, we need to introduce some notation. Recall that $\mathbb{S}^1\times\{0\}$ is a global transverse section for the irrational flow $V_t$ on $\mathbb{T}^2$, and the first return map is an irrational rotation with angle given by the inclination of the constant vector $V$. We shall use first return maps to analyse the separation properties of $X$. For the sake of clarity, it will sometimes to be convenient use the notation $\mathbb{T}^2=\mathbb{S}^1\times\mathbb{S}^1$ and we shall use both notations $\mathbb{T}^3=\mathbb{S}^1\times\mathbb{S}^1\times\mathbb{S}^1$ and $\mathbb{T}^3=\mathbb{S}^1\times\mathbb{T}^2$ without further mention.  

\begin{figure}[h!]
	\centering
	\includegraphics[width=290pt,height=160pt]{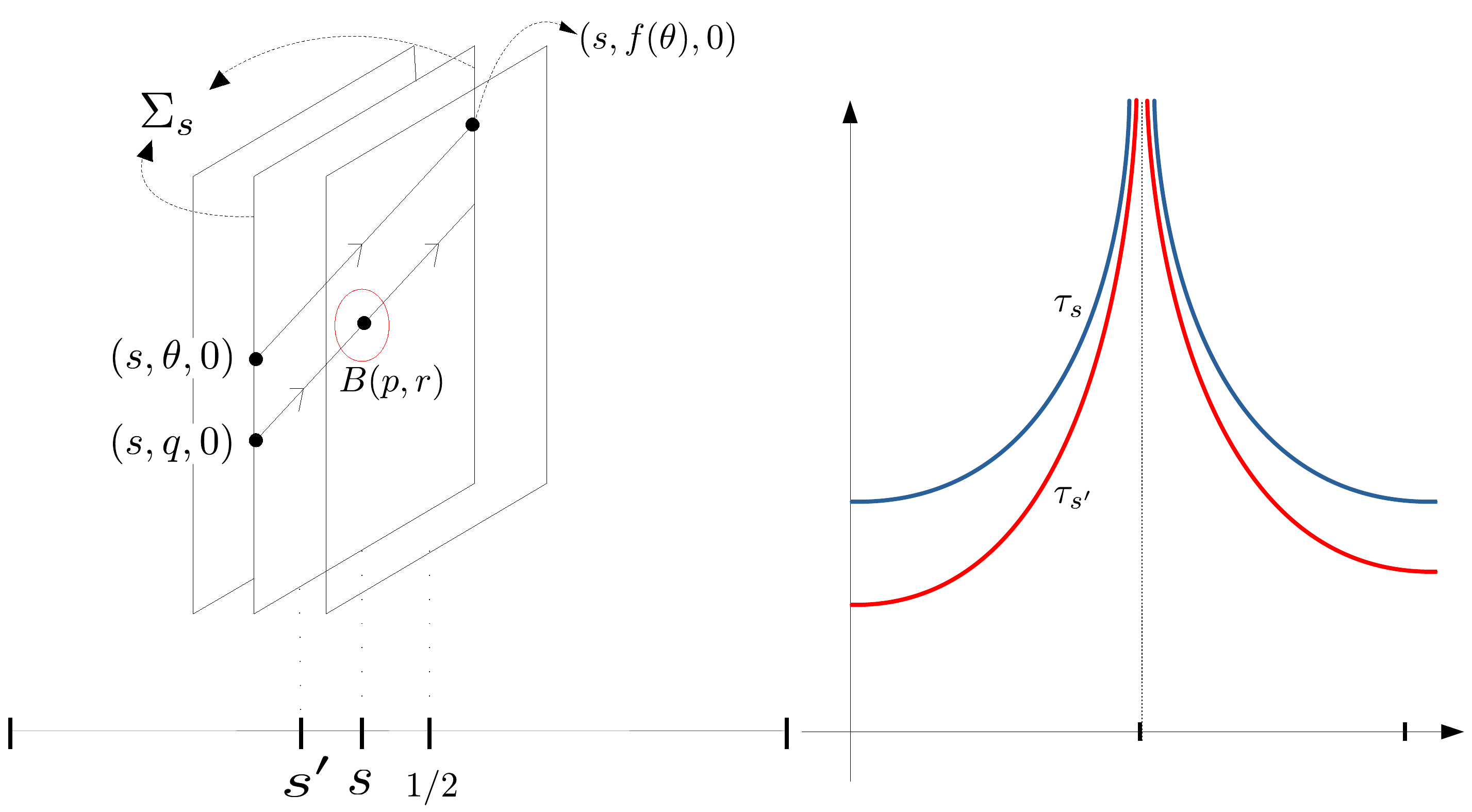}
	\caption{Example~\ref{exem.notcolinear}: two fibers close to the singular fibers and their corresponding return time functions.}
	\label{f.examplo}
\end{figure}

Let $(q,0)\in\mathbb{S}^1\times\{0\}$ be the first negative hit of $p$ under the irrational flow $V_t$. Consider the set $\Sigma\eqdef(\mathbb{S}^1\setminus\{1/2\})\times(\mathbb{S}^1\setminus\{q\})\times\{0\}\subset\mathbb{T}^3$. For each $\Sigma_s\eqdef\{s\}\times(\mathbb{S}^1\setminus\{q\})\times\{0\}$, with $s\neq 1/2$, there exists a well-defined return time function $\tau_s\colon \mathbb{S}^1\setminus\{q\}\to(0,\infty)$ and a well-defined first return map $f_s\colon \mathbb{S}^1\setminus\{q\}\to\mathbb{S}^1\setminus\{q\}$ defined by
\[X_{\tau_s(\theta)}(s,\theta,0)\eqdef(s,f_s(\theta),0)\] 
Then, the globally defined\footnote{that  is, $\tau(s,x)=\tau_s(x)$ and $f(s,x)=f_s(x)$.} maps $\tau\colon \Sigma\to(0,\infty)$ and $f\colon \Sigma\to\Sigma$ are smooth. Observe that $f_s$ is the same irrational rotation for every $s$, which is determined by the inclination if the constant vector field $V$. This is because the trajectories of $X|_{\{s\}\times\mathbb{T}^2}$, with $s\neq1/2$ are the same as the trajectories of $V$ but travelled with different speeds, and the speed is zero precisely at $(s,p)$. For this reason we shall make an abuse of notation and write simply $f\colon \mathbb{S}^1\setminus\{q\}\to\mathbb{S}^1\setminus\{q\}$ instead of $f_s$. Moreover, the trajectories on each fiber are not only the same, but travelled with smaller and smaller speeds, decreasing to zero as $s$ approaches $1/2$. This implies that $\tau_s(\theta)>\tau_{s^{\prime}}(\theta)$, for every $\theta\in\mathbb{S}^1\setminus\{q\}$, and every $s,s^{\prime}$ such that 
$0<|s-1/2|<|s^{\prime}-1/2|$. Also we have that 
\[\lim_{\theta\to q}\tau_s(\theta)=+\infty,\:\:\textrm{for every}\:\:s\neq 1/2.\]
Take $r<1/4$ and $\eps<1/16$. Thus, for each $s$, if $(s,x)\in\Sigma_s$ and $(s,y)$ satisfies $d\left((s,x),(s,y)\right)<4\eps$ then $y\notin B(p,r)$. Also, for each $s\neq 1/2$ there exists $\rho(s)>0$ such that if $x\notin B(p,r)$ then $\|X(s,x)\|\geq\rho(s)$. Since $X(s,x)=(1-4s^2)Z(s,x)$, we see that $\rho(s)$ decreases monotonically to 0 as $s\to 1/2$.  

Now, take $Y\in\mathfrak{C}^1(X)$ and a point $(s,x)\in\mathbb{T}^3$. Assume by contradiction that $\dim\langle Y(s,x),X(s,x)\rangle=2$. Since this is an open condition invariant under the flow of $X$ we can assume that $s\neq1/2$ and that $x=(\theta,0)$, with $\theta\neq q$.

Then, as in \eqref{e.proximidade}, there exists $\xi_0>0$ small enough so that for every $\xi\in(-\xi_0,\xi_0)$ and $(s^{\prime},y)=Y_{\xi}(s,x)$ we have 
\begin{equation}
\label{e.juntoseshalownow}
d(X_t(s^{\prime},y),X_t(s,x))<\eps,\:\:\:\textrm{for every}\:\:t\in\R.  
\end{equation}
Up to replacing $Y$ with  $-Y$, we can assume that $0<|s-1/2|\leq|s^{\prime}-1/2|$. As $\theta\neq q$, if $\xi_0$ is small enough we can ensure that $(s^{\prime},y)$ is not on the same orbit as $(s^{\prime},q)$. Thus, we can
consider the point $y^{\prime}=(\theta^{\prime},0)$ which is the first negative hit of the point $(s^{\prime},y)$, under the flow $X_t$, to the section $\Sigma_{s^{\prime}}$. In particular, $(s^{\prime},y^{\prime})=X_{\delta}(s^{\prime},y)$ with $|\delta|$ as small as we please, provided that $\xi_0$ is small enough. Therefore, we can assume without loss of generality that the point $y$ in \eqref{e.juntoseshalownow} is of the form $(\theta^{\prime},0)\in\Sigma_s$. 

We consider first the case $\theta=\theta^{\prime}$. Then, necessarily we must have $0<|s-1/2|<|s^{\prime}-1/2|$ and thus $\tau_s(f^{\ell}(\theta))-\tau_{s^{\prime}}(f^{\ell}(\theta))>0$ for every $\ell\geq 0$. Also, as $f$ is an irrational rotation, there is a subsequence $\ell_j$ such that $f^{\ell_j}(\theta)\to b\neq q$. This implies that 
\[\lim_{n\to\infty}\sum_{\ell=0}^{n-1}\tau_s(f^{\ell}(\theta))-\sum_{\ell=0}^{n-1}\tau_{s^{\prime}}(f^{\ell}(\theta))=+\infty.\] 
Denote $T=\sum_{\ell=0}^{n-1}\tau_s(f^{\ell}(\theta))$ and $t=\sum_{\ell=0}^{n-1}\tau_{s^{\prime}}(f^{\ell}(\theta))$, for some $n$ large enough so that $\rho(s)(T-t)>2\eps$. Notice that, \eqref{e.juntoseshalownow} and our choice of $\eps$ imply that the orbit segment $X_{[t,T]}(s^{\prime},y)$ is a line segment whose length is smaller than $2\eps$ and which is disjoint from $B((s^{\prime},p),r)$. Therefore, $\|X|_{X_{[t,T]}(s^{\prime},y)}\|\geq\rho(s^{\prime})>\rho(s)$ and thus
\[2\eps<\rho(s^{\prime})(T-t)\leq2\eps,\]   
a contradiction. 

So it remains to consider the case $\theta\neq\theta^{\prime}$. In this case, as $f$ is an irrational rotation, one finds a sequence $n_j\to\infty$ such that $f^{n_j}(\theta)\to q$ while $f^{n_j}(\theta^{\prime})\to b$, for some $b\neq q$. This implies that
\[\tau_s(f^{n_j}(\theta))-\tau_{s^{\prime}}(f^{n_j}(\theta^{\prime}))>\frac{4\eps}{\rho(s)},\:\:\textrm{for every}\:j\:\textrm{large enough}.\]
Denote similarly as before $T=\sum_{\ell=0}^{n_j-1}\tau_s(f^{\ell}(\theta))$ and $t=\sum_{\ell=0}^{n_j-1}\tau_{s^{\prime}}(f^{\ell}(\theta^{\prime}))$, for some $j$ large enough. Arguing as in the previous case, we can estimate
\[|T-t|\leq \frac{2\eps}{\rho(s)}.\]
Consider $\widetilde{T}:=T+\tau_s(f^{n_j}(\theta))$ and $\widetilde{t}:=t+\tau_{s^{\prime}}(f^{n_j}(\theta^{\prime}))$.
Again, \eqref{e.juntoseshalownow} allows us to estimate $\rho(s)(\widetilde{T}-\widetilde{t})\leq 2\eps$. On the other hand, 
\[\widetilde{T}-\widetilde{t}=T-t+\tau_s(f^{n_j}(\theta))-\tau_{s^{\prime}}(f^{n_j}(\theta^{\prime}))>\frac{2\eps}{\rho(s)},\]
a contradiction again. This establishes our claim. 

On the other hand, the vector field $Y=\frac{f}{g}Z$ is smooth and commutes with $X$. Indeed, both vector fields vanish at the fiber $\{1/2\}\times\mathbb{T}^2$. Moreover, both $f$ and $g$ are constant on each fiber and for $s\neq 1/2$ one has 
$$Y(s,x)=f(s)X(s,x).$$ 
As $X$ is tangent to each fiber $\{s\}\times\mathbb{T}^2$, we conclude that $[X,Y]=0$. Since $f(s)\to\infty$ as $s\to 1/2$, this proves that $X$ does  not have a quasi-trivial centralizer.	
\end{example} 

The vector field of Example~\ref{exem.notcolinear} is not separating and has uncountably many singularities. We believe that it is an important question to know whether of not this can be relaxed.

\begin{question}
Is there a separating vector field whose centralizer is not quasi-trivial? What about a vector field with just finitely many singularities?
\end{question}
    
We do not know what to expect as an answer to this question. Also, to illustrate the difficulty of this problem, let us mention another question, which remains open, despite of being a very special case of the former. 

\begin{question}
	\label{q.extention}
	Let $X\in\mathfrak{X}^1(M)$ be a Komuro-expansive vector field, having non-hyperbolic singularities. Is it true that $\mathfrak{C}^1(X)$ is quasi-trivial?
\end{question} 

We do not give the precise definition of Komuro-expansiveness (see for instance \cite{BonomoRochaVarandas} or \cite{Artigue}). Let us only mention that it is a notion of geometric orbit separation, which is stronger than \emph{separating} and \emph{kinematic expansiveness} (see Definitions~\ref{def.separating} and \ref{Def.kinematic}, respectively).

\subsection{The case of hyperbolic zeros}
    
The main result of this section is Theorem \ref{thm.quasitriviality} below, in which we obtain the quasi-triviality from collinearity of $\mathfrak{C}^1(X)$ assuming only that all the singularities of $X$ are hyperbolic. 
\begin{definition} 
A function $f \colon M \to \R$ is called a \textit{first integral} of $X$ if it is of class $C^1$ and satisfies  $X \cdot f\equiv 0$. 
We denote by $\mathfrak{I}^1(X)$ the set of all such maps.
\end{definition} 
In particular, for any $c \in \R$, the   constant map $\underline{c}(x): = c$  is in $\mathfrak{I}^1(X)$, and then, we always have $\R\simeq \{\underline{c}: c \in \R\} \subset \mathfrak{I}^1(X)$. The following theorem is a reformulation in terms of $\mathfrak{I}^1(X)$ of Theorem \ref{res qu tri}. 

\begin{theorem}
\label{thm.quasitriviality}
Let $X\in \mathfrak{X}^1(M)$. If $X$ has collinear centralizer and all the singularities of $X$ are hyperbolic, then $X$ has quasi-trivial $C^1$-centralizer, in the sense of Definition \ref{def.quasitrivial}. More precisely, we have
$$
\mathfrak{C}^1(X)=\{f X: f \in \mathfrak{I}^1(X)\}. 
$$
\end{theorem}
This theorem is an immediate consequence of Propositions \ref{prop.quasisaddles}, \ref{prop.quasisaddles2}, and \ref{p.sinkextension} below. We divide the proof into two subsections to emphasize that the technique to deal with singularities that are saddles is different from the technique to deal with sinks and sources. We also remark that Theorem \ref{thm.quasitriviality} gives a significant improvement compared with previous works on centralizers of vector fields, since we only need $C^1$ regularity. The  results that were known previously used Sternberg's linearisation results, which require higher regularity of the vector field and non-resonant conditions on the eigenvalues of the singularity, see for instance \cite{BonomoRochaVarandas} and \cite{BonomoVarandas}. 

\addtocontents{toc}{\protect\setcounter{tocdepth}{1}}
\subsection{When the  singularity is of saddle type}
Given any vector field $X\in \mathfrak{X}^1(M)$, and $Y \in \mathfrak{C}^1(X)$, by Lemma \ref{lemme facile}, we know that $Y|_{M_X}=fX|_{M_X}$, for some $C^1$, $X$-invariant function $f \colon M_X\to \R$. Assume that $\sigma \in \mathrm{Zero}(X)$ is a saddle type singularity. In Propositions \ref{prop.quasisaddles} and \ref{prop.quasisaddles2}, we show that $f$ can be extended to a $C^1$ function in a neighborhood of $\sigma$.

\begin{proposition}\label{prop.quasisaddles}
Let $X\in \mathfrak{X}^1(M)$ and let $f\colon M_X \to \mathbb{R}$ be an $X$-invariant continuous function. If $\sigma \in \mathrm{Zero}(X)$ is a saddle type singularity, then $f$ admits a continuous extension to $\sigma$.
\end{proposition} 
The argument below was already given in case 2 of Lemma 3.6 in \cite{BonomoRochaVarandas}. We include it here for the sake of completeness.

\begin{proof}
Recall that $M$ has dimension $d \geq 0$. For simplicity, denote $d^s\eqdef\operatorname{ind}(\sigma)$ and $d^u\eqdef d-d^s$. Fix a point $p_s \in W^s_{\mathrm{loc}}(\sigma)$. We claim that for any point $q_u\in W^u(\sigma)$ we have that $f(p_s) = f(q_u)$. By the $X$-invariance of $f$, it is enough to consider $q_u \in W^u_{\mathrm{loc}}(\sigma)$. Let $(D^s_n)_{n\in \N}$ be a sequence of discs of dimension $d^s$, centered at $q_u$, with radius $\frac{1}{n}$ and transverse to $W^u_{\mathrm{loc}}(\sigma)$. Similarly, consider a sequence $(D^u_n)_{n\in \N}$  of discs of dimension $d^u$, centered at $p_s$, with radius $\frac{1}{n}$, and transverse to $W^s_{\mathrm{loc}}(\sigma)$.

For each $n\in \N$, by the $\lambda$-lemma (see \cite{PalisdeMelo}, Chapter 2.7) there exists $t_n>0$ such that $X_{t_n}(D^u_n) \pitchfork D_n^s \neq \emptyset$. In particular, there exists a point $x_n \in D^u_n$ that verifies $X_{t_n}(x_n) \in D^s_n$. It is immediate that $x_n \to p_s$, as $n\to + \infty$. Since the function $f$ is continuous on $M_X$, we have that $f(x_n) \to f(p_s)$. We also have that $X_{t_n}(x_n) \to q_u$ as $n\to +\infty$. Hence, $f(X_{t_n}(x_n)) \to f(q_u)$. By the $X$-invariance of $f$,  we have
$$
f(p_s) = \displaystyle \lim_{n\to + \infty} f(x_n) = \lim_{n\to +\infty} f(X_{t_n}(x_n)) = f(q_u).
$$
Analogously, we can prove that for a fixed $q_u'\in W^u_{\mathrm{loc}}(\sigma)$ and for any $p_s'\in W^s(\sigma)$, it is verified $f(p_s') = f(q_u')$. We conclude that $f|_{W^s(\sigma) - \{\sigma\}} = f|_{W^u(\sigma)- \{\sigma\}} = c$, for some constant $c\in \R$. In particular, we can define  a continuous extension of $f$ to the singularity $\sigma$ by setting $f(\sigma):= c$. 
\end{proof}

\begin{figure}[h!]
	\centering
	\includegraphics[width=14.5cm]{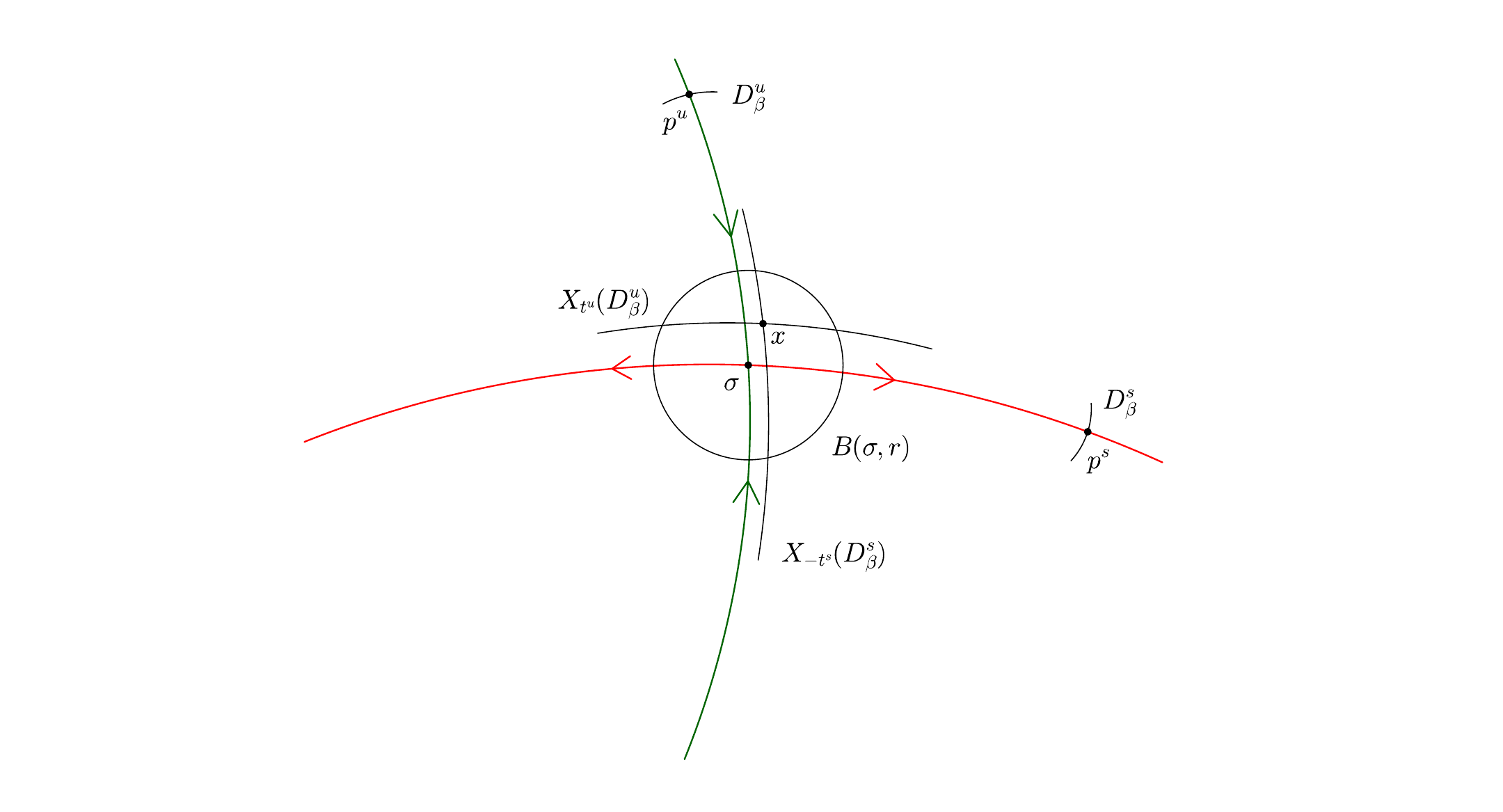}
    \caption{Proposition \ref{prop.quasisaddles2}.}
	\label{f.Internet}
\end{figure}  	 

The proposition below is our main novelty regarding the extension problem for saddle type singularities. We stress that the $C^1$ extension even in this case was not done in previous works.

\begin{proposition}\label{prop.quasisaddles2}
Let $X\in \mathfrak{X}^1(M)$ and let $f\colon M_X \to \mathbb{R}$ be an $X$-invariant function of class $C^1$. If $\sigma \in \mathrm{Zero}(X)$ is a saddle type singularity, then
\[\lim_{x\to\sigma}\nabla f(x)=0.\]
In particular, $f$ can be extended to a $C^1$ function in a neighborhood of $\sigma$, by setting $\nabla f(\sigma):=0$. 
\end{proposition} 
\begin{proof}
Given $\eps>0$	we shall find $r>0$ so that if $d(x,\sigma)<r$ then $\|\nabla f(x)\|<\eps$. For simplicity of notation, consider, for each $x\in M_X$, the linear map $Df(x)\colon T_xM\to\R$, given by 
\[Df(x)v=\left(\nabla f(x),v\right).\]
Fix $r_0>0$ so that $B(\sigma,2r_0)\cap\zero(X)=\{\sigma\}$ and also that $K^{\star}\eqdef W^{\star}_{\mathrm{loc}}(\sigma)\cap\partial B(\sigma,r_0)$ is a non-empty compact subset of $W^{\star}_{\mathrm{loc}}(\sigma)$, for $\star=s,u$. Consider
\[C\eqdef\sup\left\{\|Df(p)\cdot v\|:v\in T_pM,\:\|v\|=1,\, p\in K^s\cup K^u\right\}.\]
Then, since $f\colon M_X\to\R$ is $C^1$ and since $(K^s\cup K^u)\cap\zero(X)=\emptyset$, there exists $\beta_0>0$ such that is $x\in B(p,\beta_0)$, for some $p\in K^s\cup K^u$ and if $v\in T_xM$ has $\|v\|=1$ then $\|Df(x)\cdot v\|\leq 2C$.

Now, by the $\lambda$-lemma, given $0<r<r_0$ there exists $0<\beta<\beta_0$ so that the following property holds: given points $p^s\in K^s$ and $p^u\in K^u$, and given embedded disks $D^s_{\beta}$ and $D^u_{\beta}$, of dimension $d^s$ and $d^u$, centered in $p^u$ and $p^s$, with diameter smaller than $\beta$ and transverse to $W^u_{\mathrm{loc}}(\sigma)$ and $W^s_{\mathrm{loc}}(\sigma)$, respectively, there exists $t^s, t^u>0$ such that the set 
\[D^{su}_{\beta}\eqdef X_{t^u}(D^u_{\beta})\cap X_{-t^s}(D^s_{\beta})\]
is a singleton contained in $B(\sigma,r)$.  

Moreover, for each $x\in B(\sigma,r)\setminus\{\sigma\}$ there exists a choice of $0<\beta<\beta_0$, and $p^{\star}\in K^{\star}$, $\star=s,u$, so that $D^{su}_{\beta}=\{x\}$. For that, it suffices to work in local coordinates and extend the embedded disks $W^s_{\mathrm{loc}}(\sigma)$ and $W^s_{\mathrm{loc}}(\sigma)$ to a pair of transverse foliations by embedded disks and then take pre-images and images by the flow. Notice that $t^s,t^u\to+\infty$ uniformly in $x$ as $r\to 0$. 

Thus, given a point $x\in B(\sigma,r)\setminus\{\sigma\}$ we take the aforementioned points $p^s,p^u$ and the disks $D^u_{\beta}$ and $D^s_{\beta}$. Since $D^{su}_{\beta}=\{x\}$, there are points $x^u\in D^u_{\beta}$ and $x^s\in D^s_{\beta}$ so that $x=X_{t^u}(x^u)=X_{-t^s}(x^s)$. With them, define the subspaces $E^u_x\eqdef DX_{t^u}(x^u)T_{x^u}D^u_{\beta}$ and $E^s_x\eqdef DX_{-t^s}(x^s)T_{x^s}D^s_{\beta}$. The $\lambda$-lemma also implies that $\angle(E^{\star}_x,E^{\star}(\sigma))\to 0$ as $r\to 0$, where $T_{\sigma}M=E^s(\sigma)\oplus E^u(\sigma)$ is the hyperbolic decomposition. In particular, one has $T_xM=E^s_x\oplus E^s_x$. 

By continuity of the derivative $DX_t$ and by the hyperbolicity of the splitting $E^s(\sigma)\oplus E^u(\sigma)$, for every $w\in T_{x^u}D^u_{\beta}$ with $\|w\|=1$, we have $\|DX_{t^u}(x^u)\cdot w\|\to\infty$ uniformly as $r\to 0$ (recall that $t^u,t^s\to+\infty$ uniformly as $r\to 0$). Similarly, if $w\in T_{x^s}D^s_{\beta}$ has $\|w\|=1$ then $\|DX_{-t^s}(x^s)\cdot w\|\to 0$ uniformly as $r\to 0$. Therefore, for $r>0$ sufficiently small if $w\in T_{x^u}D^u_{\beta}$ is a unit vector then $\|DX_{t^u}(x^u)\cdot w\|>2C/\eps$ and similarly $\|DX_{-t^s}(x^s)\cdot w\|>2C/\eps$, if $w\in T_{x^s}D^s_{\beta}$ has $\|w\|=1$.  

We now use the fact that $f$ is an invariant function. Indeed, the equation $f=f\circ X_t$, which holds true in $M_X$ and for every $t\in\R$, by assumption, implies that 
\[Df(y)\cdot w=Df(X_t(y))DX_t(y)\cdot w,\:\:\:\textrm{for every}\:\:y\in M_X,\:t\in\R.\]
Since every unit vector $v\in E^u_x$ can be written as $v=\frac{DX_{t^u}(x^u)\cdot w}{\|DX_{t^u}(x^u)\cdot w\|}$, for some unit vector $w\in T_{x^u}D^u_{\beta}$, one deduces that if $r>0$ is small enough then
\begin{align*}
\|Df(x)\cdot v\|&=\frac{Df(X_{t^u}(x^u))DX_{t^u}(x^u)\cdot w}{\|DX_{t^u}(x^u)\cdot w\|}=\frac{Df(x_u)\cdot w}{\|DX_{t^u}(x^u)\cdot w\|} \\
&<\frac{2C}{2C/\eps}=\eps,
\end{align*}       
for every $v\in E^u_x$, with $\|v\|=1$. In a similar way we show that $\|Df(x)\cdot v\|<\eps$ for every unit vector $v\in E^s_x$. As $\angle(E^s_x,E^u_x)$ is uniformly bounded away from zero, from the cosine law we deduce that $\|Df(x)\cdot v\|<4\eps$, for every unit vector $v\in T_xM$. This completes the proof. 
\end{proof}

\subsection{When the  singularity is type sink or source}

We now deal with hyperbolic singularities of type sink or source. We stress that, together with Proposition~\ref{prop.quasisaddles2}, the result below is  the main novelty of our paper regarding criteria for quasi-triviality. We are able to obtain a $C^1$ extension without any generic or non-resonance assumption. We solve the extension problem only using the hyperbolicity of the singularity. Compare with \cite{Sad,BonomoRochaVarandas,BonomoVarandas}.  

\begin{proposition}
\label{p.sinkextension}
Let $X,Y\in\mathfrak{X}^1(M)$ such that $[X,Y]=0$ and $\dim\langle X(x),Y(x)\rangle \leq 1$, for every $x\in M$. Assume that $\sigma\in\zero(X)$ is a hyperbolic sink. Then, there exists $c\in\R$ such that $Y(x)=cX(x)$, for every $x\in W^s(\sigma)$. 
\end{proposition}

In the proof of Proposition~\ref{p.sinkextension} we shall use the following elementary lemma. 

\begin{lemma}
\label{l.compactos}
Let $(E,\|\cdot\|)$ be a finite-dimensional vector space endowed with a norm. Let $\Lambda$ be an infinite set and assume that for each $\lambda\in\Lambda$, there exists a non-empty compact subset $K_{\lambda}\subset\mathbb{S}:=\{v \in E: \|v\|=1\}$ of the sphere of unit vectors in $(E,\|\cdot\|)$, such that 
$$\lambda^{\prime}\neq\lambda\textrm{ in }\Lambda\quad\implies\quad K_{\lambda}\cap K_{\lambda^{\prime}}=\emptyset.$$
Suppose that $\dim E\geq 2$. Then, there exist a finite subset $\{\lambda,\lambda_1,\dots,\lambda_k\}\subset\Lambda$ and vectors  $\{u,u_1,\dots,u_k\}$ such that 
\begin{enumerate}
	\item $u\in K_{\lambda}$ and $u_\ell\in K_{\lambda_\ell}$, for each $\ell=1,\dots,k$;
	\item $u$ belongs to the subspace spanned by $\{u_1,\dots,u_k\}$;
	\item $\{u_1,\dots,u_k\}$ is a linearly independent set.
\end{enumerate} 
\end{lemma}
\begin{proof}
We begin with a simple observation that we will use repeatedly in this proof: for each $u\in\mathbb{S}$, $-u$ is the only other vector in $\mathbb{S}$ which is collinear with $u$.
 	
Now, since $\Lambda$ is infinite, we can pick a sequence $(\lambda_n)_{n \geq 0}\subset\Lambda$, whose terms are distinct. For each $n$, choose a vector $u_n\in K_{\lambda_n}$. Since $\dim E\geq 2$, and the sets $K_{\lambda}$ are pairwise disjoint, by the simple observation above, we can assume without lost of generality that the set $\{u_1,u_2\}$ is linearly independent. Assume by contradiction that the conclusion does not hold. Then, it follows by induction that for every $n$ the set $\{u_1,u_2,u_3,\dots,u_n\}$ must be linearly independent. But this is absurd as $E$ is finite dimensional.    
\end{proof}

\begin{proof}[Proof of Proposition~\ref{p.sinkextension}]
By Lemma \ref{lemme facile}, for any $x \in M_X$, we have $Y(x)=f(x)X(x)$, for some $C^1$ function $f\colon M_X\to\R$. Notice that, as $\sigma$ is an isolated zero of $X$, we have $\sigma\in\zero(Y)$. Take $\eps>0$ small so that $\overline{B}(\sigma,\eps)\subset W^s(\sigma)$ and let $S:=\partial B(\sigma,\eps)$. In particular, notice that $x\in S$ implies $\lim_{t \to+\infty} X_t(x)=\sigma$.
 
Also notice that for every $x\in W^s(\sigma)$, there exists $T\in\R$ such that $X_T(x)\in S$. Therefore, since $f(X_t(x))=f(x)$ for  every $x\in M_X$ and $t\in\R$, the proof of the proposition is reduced to the proof of the following claim. 

\begin{claim}
\label{claim.derivada}
$Df(p)=0$ for every $p\in S$.
\end{claim}

We shall postpone the proof of Claim~\ref{claim.derivada}. Take a point $p\in S$ and consider the set 
$$V(p):=\left\{u\in T_{\sigma}M:\exists\:t_n\to\infty,\, u=\lim_{n\to\infty}\frac{X(X_{t_n}(p))}{\|X(X_{t_n}(p))\|}\right\}.$$
By compactness, $V(p)$ is non-empty, and every $u\in V(p)$ is a unit vector; in particular, $0\notin V(p)$.
The following claims are the key arguments for this proof.
\begin{claim}
\label{claim.casoumd}
If $u\in V(p)$ then $DY(\sigma)\cdot u=f(p)DX(\sigma)\cdot u$.
\end{claim}
\begin{proof}
Fix some $t\in\R$. Since $Y(X_{t+s}(p))=f(p)X(X_{t+s}(p))$ for every $s\in\R$, taking the derivative with respect to $s$ on both sides we obtain
$$DY(X_t(p))\cdot \left(\frac{X(X_t(p))}{\|X(X_t(p))\|}\right)=f(p)DX(X_t(p))\cdot \left(\frac{X(X_t(p))}{\|X(X_t(p))\|}\right).$$ 
By using this formula with $t=t_n$ and letting $n\to\infty$ we conclude that
$DY(\sigma)\cdot u=f(p)DX(\sigma)\cdot u,$ proving the claim. 
\end{proof}

\begin{claim}
\label{claim.whatsapp}	
If $p,q\in S$ and $V(p)\cap V(q)\neq\emptyset$ then $f(p)=f(q)$.
\end{claim}
\begin{proof}
Assume that there exists $u\in V(p)\cap V(q)$. Then, by Claim~\ref{claim.casoumd}, one has
$$DY(\sigma)\cdot u=f(p)DX(\sigma)\cdot u=f(q)DX(\sigma)\cdot u.$$
As $DX(\sigma)$ is an invertible linear map (because all eigenvalues are negative) this implies that $(f(p)-f(q))u=0$, and since $u\neq 0$, the claim is proved.
\end{proof}
We are now in position to give the proof of Claim~\ref{claim.derivada}.
Assume by contradiction that the claim is not true. Then, there exist  $U\subset S$ and real numbers $a<b$ such that $f\colon U\to[a,b]$ is surjective. 

Now, for every $t\in[a,b]$, we choose some point $p_t\in U\cap f^{-1}(t)$, and we consider the family of compact subsets $\{V(p_t)\}_{t\in[a,b]}\subset T_{\sigma}M$ of unit vectors. As $t\neq s$ implies $f(p_t)\neq f(p_s)$, one obtains from Claim~\ref{claim.whatsapp} that the family $\{V(p_t)\}_{t\in[a,b]}$ satisfies all the assumptions of Lemma~\ref{l.compactos}.

Thus, there exists a finite set $\{p,p_1,\dots,p_k\}\subset U$ and vectors $u\in V(p)$, $u_\ell\in V(p_\ell)$, $\ell=1,\dots,k$, with $u\in\langle u_1,\dots,u_k\rangle$ and $\{u_1,\dots,u_k\}$ linearly independent, and such that $f(p_i)\neq f(p_j)\neq f(p)$, for every $i,j\in\{1,\dots,k\}$. 


Take $\alpha^1,\dots,\alpha^k\in\R$ such that $u=\sum_{\ell=1}^{k}\alpha^\ell u_\ell$. Using Claim~\ref{claim.casoumd} we can write
$$DY(\sigma)\cdot u=f(p)DX(\sigma)\cdot u=DX(\sigma)\cdot \left(\sum_{\ell=1}^k f(p)\alpha^\ell u_\ell\right).$$
Also
$$DY(\sigma)\cdot u_\ell=f(p_\ell)DX(\sigma)\cdot u_\ell,\:\forall\:\ell=1,\dots,k,$$
which implies that 
$$DY(\sigma)\cdot u=DX(\sigma)\cdot \left(\sum_{\ell=1}^kf(p_\ell)\alpha^\ell u_\ell\right).$$
Since $DX(\sigma)$ is invertible we must have $\sum_{\ell=1}^kf(p)\alpha^\ell u_\ell=\sum_{\ell=1}^kf(p_\ell)\alpha^\ell u_\ell$, and as $\{u_1,\dots,u_k\}$ are linearly independent, this gives 
$$f(p)\alpha^\ell=f(p_\ell)\alpha^\ell,\:\textrm{for every}\:\ell=1,\dots,k.$$ Since $u\neq 0$ there exists some $\alpha^\ell\neq 0$. However, this implies that $f(p)=f(p_\ell)$, a contradiction.      
\end{proof}

We now give the proof of Theorem~\ref{thm.quasitriviality}.

\begin{proof}[Proof of Theorem \ref{thm.quasitriviality}]
	Assume that  $\mathfrak{C}^1(X)$ is collinear and that each singularity $\sigma \in \mathrm{Zero}(X)$ is hyperbolic. Let us consider $Y \in \mathfrak{C}^1(X)$. By Lemma \ref{lemme facile}, there exists a $C^1$ function $f \colon M_X \to \R$ which satisfies $X \cdot f\equiv 0$ on $M_X$ and such that $Y(x)=f(x)X(x)$, for every $x \in M_X$. By assumption, the singularities of $X$ are hyperbolic, hence they are isolated, and  $Y(\sigma)=0$, for all $\sigma \in \mathrm{Zero}(X)$.  By Propositions \ref{prop.quasisaddles}, \ref{prop.quasisaddles2} and \ref{p.sinkextension}, we can extend $f$ to a $C^1$ invariant function on $M$. We conclude that $f$ is a first integral of $X$, and $Y=f X$. 
	
	Conversely, assume that $f \colon M \to \R$ is a first integral of $X$. We define a vector field $Y \in \mathfrak{X}^1(M)$ as $Y(x):=f(x) X(x)$, for every $x \in M$.  Indeed, both $f$ and $X$ are of class $C^1$, thus $Y$ is $C^1$ too. 
	Moreover, we have  $Y \in \mathfrak{C}^1(X)$, since
	$$
	[X,Y]=(X\cdot f) X+f [X,X]=0.\textrm{  } \qedhere 
	$$
\end{proof}

\addtocontents{toc}{\protect\setcounter{tocdepth}{2}}
\section{The study of invariant functions and trivial centralizers}
\label{sec.trivial}

%
%
%
%
%
%
%

The main focus of this section is the study of invariant functions. An invariant function is also called a first integral of the system. There are several works that study the existence of non trivial (non constant) first integrals, see for instance \cite{Mane,Hurley,  FathiSiconolfi,Pageault,FathiPageault,ABC,BernardiFlorio,BernardiFlorioWiseman}. In this work we study dynamical conditions that imply the non-existence of first integrals. 

First, it is easy to obtain examples of vector fields with quasi-trivial $C^1$-centralizer which is not trivial. Indeed consider the vector field in example \ref{example.separating}. Since $X$ is separating, it has collinear $C^1$-centralizer. This flow is non-singular, hence it has quasi-trivial $C^1$-centralizer. Now take any non-constant $C^1$-function $f$ which is constant on each orbit, that is, a function which depends only on the coordinate $r$. The vector field $Y= fX$ belongs to the $C^1$-centralizer of $X$, therefore the centralizer of $X$ is only quasi-trivial. 

Let $X\in \mathfrak{X}^1(M)$. Recall that a compact set $\Lambda$ is a basic piece for $X$ if $\Lambda$ is $X$-invariant and transitive, that is, it has a dense orbit. We say that $X$ admits a countable spectral decomposition if $\Omega(X) = \displaystyle \sqcup_{i\in \N} \Lambda_i$, where the sets $\Lambda_i$ are pairwise disjoint basic pieces.

\begin{theorem}
\label{thm.continuousfinitetrivial}
Let $X\in \mathfrak{X}^1(M)$. If $X$ admits a countable spectral decomposition then any continuous $X$-invariant function is constant. 
\end{theorem}
\begin{proof}
Let $f\colon M \to \mathbb{R}$ be a continuous $X$-invariant function. Suppose that $f$ is not constant. Since $M$ is connected, there exist two real numbers $a<b$ such that $f(M) = [a,b]$. It is easy to see that in each basic piece the function $f$ is constant: this follows from the transitivity of each basic piece. For each $i\in \N$ define $c_i := f(\Lambda_i)$. Since $X$ admits a countable spectral decomposition, the set $C:=\{c_1, c_2,\dots\}$ is at most countable and in particular $[a,b] - C$ is non-empty. Take any value $c\in [a,b] - C$ and consider $\Lambda := f^{-1}(\{c\})$.

The set $\Lambda$ is compact and $X$-invariant. Hence, for any point $p\in \Lambda$ we must have $\omega(p) \subset \Lambda$, where $\omega(p)$ is the set of all accumulations points of the future orbit of $p$. By the countable spectral decomposition, $\omega(p)$ must be contained in some basic piece $\Lambda_i$, which implies that $\Lambda \cap \Lambda_i \neq \emptyset$. Since $\Lambda$ is a level set of $f$, this implies that $c_i =f(\Lambda_i) = f(\Lambda)=c$ and this is a contradiction with our choice of $c$.  
\end{proof}

Theorem \ref{theo e} follows easily from Theorems \ref{res qu tri} and \ref{thm.continuousfinitetrivial}. Let us now give some applications.

In \cite{Peixoto2}, Peixoto proved that a $C^1$-generic vector field on a compact surface is Morse-Smale. Recall that a vector field is Morse-Smale if the non-wandering set is the union of finitely many hyperbolic periodic orbits and hyperbolic singularities, and it verifies some transversality condition. In particular, the non-wandering set is finite. As a consequence of this result of Peixoto and Theorems \ref{theo d} and \ref{theo e}, we have the following corollary. 
\begin{maincorollary}\label{cor.surfaces}
	Let $M$ be a compact connected surface. Then, there exists a residual set $ \mathcal{R}_\dagger \subset \mathfrak{X}^1(M)$ such that for any $X \in \mathcal{R}_\dagger$,  the $C^1$-centralizer of $X$ is trivial. 
\end{maincorollary}

A $C^1$-vector field $X$ is \emph{Axiom A} if the non-wandering set is hyperbolic and $\Omega(X) = \overline{\mathrm{Per}(X)}$. It is well known that Axiom A vector fields admits a spectral decomposition, with finitely many basic pieces. As a corollary of our Theorems \ref{theo d} and \ref{theo e}, we obtain the following result which is the main theorem in \cite{BonomoVarandasaxioma}.

\begin{maincorollary}[Theorem $1.1$ in \cite{BonomoVarandasaxioma}]\label{cor.gentrivialityaxioma}
	A $C^1$-generic Axiom A vector field has trivial $C^1$-centralizer.
\end{maincorollary}

\begin{remark}
	Corollary \ref{cor.gentrivialityaxioma} actually holds for more a general type of hyperbolic system called \emph{sectional Axiom A}, in any dimension. We refer the reader to  Definition $2.14$ in \cite{metmorales} for a precise definition. In \cite{BonomoVarandasaxioma}, the authors also proved the triviality of the $C^1$-centralizer for sectional Axiom A flows in dimension three. 
\end{remark}

Another corollary is for $C^1$-vector fields far from homoclinic tangencies in dimension three. Let us make it more precise. Recall that a vector field $X \in \mathfrak{X}^1(M)$ has a \textit{homoclinic tangency} if there exists a hyperbolic non-singular closed orbit $\gamma $ and a non-transverse intersection between $W^s(\gamma)$ and $W^u(\gamma)$. By the proof of Palis conjecture in dimension three given in \cite{CrovisierYang},  a $C^1$-generic $X \in \mathfrak{X}^1(M)$ which cannot be approximated by such vector fields admits a finite spectral decomposition, hence:

\begin{maincorollary}\label{cor.dim3}
	Let $M$ be a compact connected $3$-manifold. Then there exists a residual subset $\mathcal{R}_\ddagger\subset \mathfrak{X}^1(M)$ such that any vector field $X \in \mathcal{R}_\ddagger$ which cannot be approximated by vector fields exhibiting a homoclinic tangency has trivial $C^1$-centralizer.  
\end{maincorollary}

As a simple application of Proposition \ref{premier theorem} and Theorem \ref{theo e}, we obtain the triviality of the centralizer of the flow introduced in \cite{Artigue0}. This example is a transitive Komuro expansive flow on the three-sphere such that all its singularities are hyperbolic. In particular, by the discussion in \cite{Artigue}, this flow is separating.  

\subsection{First integrals and trivial $C^1$-centralizers}
Recall that for any $X \in \mathfrak{X}^1(M)$, we let $\mathfrak{I}^1(X):=\{f \in C^1(M,\R): X \cdot f\equiv 0\}$ be the set of all $C^1$ functions  which are invariant under $X$. As an easy consequence of Theorem \ref{res qu tri}, we obtain the following lemma.

\begin{lemma}\label{deuxie cor}
Let  $X \in \mathfrak{X}^1(M)$. Assume that the singularities of $X$ are hyperbolic and that the $C^1$-centralizer of $X$ is collinear. Then $X$ has trivial $C^1$-centralizer if and only if the set of first integrals of $X$ is trivial, i.e., $ \mathfrak{I}^1(X)\simeq \R$. 
\end{lemma}

As an immediate consequence of Theorem \ref{thm.continuousfinitetrivial} and Lemma \ref{deuxie cor}, we obtain:
\begin{corollary}
Let  $X \in \mathfrak{X}^1(M)$ be such that $X$ admits a countable spectral decomposition and all its singularities are hyperbolic. 
If  the $C^1$-centralizer of $X$ is collinear, then it is trivial. 
\end{corollary}
The following lemma will be used several times in this section.

\begin{lemma}\label{lemme periodic points}
Let $M$ be a compact and boundaryless (closed) manifold of dimension $d \geq 1$ and let $X \in \mathfrak{X}^1(M)$. Then, for any $f \in \mathfrak{I}^1(X)$ and for any hyperbolic singularity or hyperbolic periodic point $p \in\mathrm{Zero}(X) \cup \mathrm{Per}(X)$, we have $\nabla f(p)= 0$. 
\end{lemma}

\begin{proof}
Let $X \in \mathfrak{X}^1(X)$ be as above and let $f \in \mathfrak{I}^1(X)$. If $\sigma \in \mathrm{Zero}(X)$ is a hyperbolic singularity, then it follows from Propositions  \ref{prop.quasisaddles2} and \ref{p.sinkextension} that $\nabla f(\sigma)=0$.  Assume now that for some regular hyperbolic  periodic point $p \in \mathrm{Per}(X)$, we have $\nabla f(p) \neq 0$. Then, we have the hyperbolic decomposition along its orbit given by
\[
T_{\orb(p)}M = E^s \oplus \langle X \rangle \oplus E^u.
\]
Note that $f|_{W^s(p)} = f|_{W^u(p)} = f(p)$: this follows easily from the $X$-invariance of $f$. Since $\nabla f(p) \neq 0$, by the local form of submersion, we have that $\Sigma := f^{-1}(\{f(p)\})$ is locally contained in a submanifold $D$ of dimension $d-1$. In particular, $T_pD$ is a subspace of dimension $d-1$ contained in $T_pM$. However, our previous observation implies that $W^s_{\mathrm{loc}}(p)\subset \Sigma$ and $W^u_{\mathrm{loc}}(p) \subset \Sigma$. This implies that $E^s(p) \oplus \langle X(p) \rangle \oplus E^u(p) \subset T_pD$. By the hyperbolicity of $p$, we have that $T_pM =E^s(p)\oplus \langle X(p) \rangle \oplus E^u(p)$, but this is a contradiction with the fact that $T_pD$ has dimension $d-1$.  
\end{proof}

For surfaces where the Poincar\'e-Bendixson Theorem holds true, any level set of an invariant function $f$ has to contain a singularity or a periodic orbit, which forces $f$ to be constant in the generic case where the latter are hyperbolic. 
\begin{proposition}\label{proposition surfaces}
Let $M:=\mathbb{S}^2$ be the two dimensional sphere, 
 and let $X \in \mathfrak{X}^1(M)$ 
be such that every singularity and periodic orbit of $X$ is hyperbolic.
Then any continuous function that is invariant under the flow $X$ is constant.
\end{proposition}

\begin{proof}
Let $X \in \mathfrak{X}^1(M)$ be as above, and let
  $f \colon X \to \R$ be a continuous function which satisfies $f(X_t(x))=f(x)$ for all $x \in M$ and $t \in \R$. Assume that $f$ is non-constant. Then $f(M)=[a,b]$, with $a<b\in \R$. By assumption, each singularity of $X$ is hyperbolic, hence there are finitely many of them. 
  Let $c \in [a,b]-f(\mathrm{Zero}(X))$. For any $x \in f^{-1}(\{c\})$, it follows from Poincar\'e-Bendixson Theorem that $\omega(x)$ is a closed orbit formed by regular points, and by our assumption, $\omega(x)$ is hyperbolic. Moreover, $\omega(x) \subset f^{-1}(\{c\})$, since $f$ is invariant under $X$.  In particular, for each $c \in  [a,b]-f(\mathrm{Zero}(X))$, the level set $f^{-1}(\{c\})$ contains a hyperbolic periodic orbit. This is a contradiction, since $ [a,b]-f(\mathrm{Zero}(X))$ is uncountable, while there can be at most countably many hyperbolic periodic orbits.
\end{proof}

\subsection{Some results in higher regularity}
Using Sard's theorem and Pesin's theory we can obtain more information about the invariant functions. 

\begin{theorem}
\label{thm.sardinvariantfunction}
Let $M$ be a closed and connected Riemannian manifold of dimension $d \geq 1$ and let $X\in \mathfrak{X}^1(M)$. Suppose that $X$ verifies the following conditions:
\begin{itemize}
\item every singularity and periodic orbit of $X$ is hyperbolic;
\item $\Omega(X) = \overline{\mathrm{Per}(X)}$.
\end{itemize}
Then any   function $f \colon M_X \to \R$ which is $X$-invariant and such that $f|_{M_X}$ is of class $C^d$  is constant.
\end{theorem}
\begin{proof}
Let $f \colon M_X \to \R$ be an $X$-invariant function such that $f|_{M_X}$ is of class $C^d$. By assumption, each singularity $\sigma \in \mathrm{Zero}(X)$ is hyperbolic, thus by Propositions \ref{prop.quasisaddles}  and \ref{p.sinkextension}, $f$ admits a continuous  extension to the whole manifold $M$. Suppose that $f$ is not constant. Then, there exist two real numbers $a<b$ such that $f(M)=[a,b]$. All the singularities are hyperbolic, hence there are at most finitely many of them.
In particular, there exists a non-trivial open interval $I \subset f(M)-f(\mathrm{Zero}(X))$.  Since $f|_{M_X}$ is of class $C^d$, then by Sard's theorem, there exists a set $R\subset I$ of full Lebesgue measure, such that each $c\in R$ is a regular value of $f$, that is, any $x\in f^{-1}(\{c\})$ verifies $\nabla f(x) \neq 0$.

 Fix a value $c\in R - f(\mathrm{Zero}(X))$. By the same reason as in the proof of Theorem \ref{thm.continuousfinitetrivial}, we have that $f^{-1}(\{c\}) \cap \Omega(X) \neq 0$. The fact that $c$ is a regular value implies that there exists $y\in \Omega(X)\cap M_X$ such that $\nabla f(y) \neq 0$, thus by the continuity of $X$  and $\nabla f$, there exists a neighborhood $\mathcal{V} \subset M_X$ of $y$ such that  the gradient of $f$ is non-zero at any $q \in \mathcal{V}$.
Using the density of periodic points in the non-wandering set, we conclude that there exists a regular periodic point $p\in \mathrm{Per}(X)\cap \mathcal{V}$ such that $\nabla f(p) \neq 0$. By Lemma \ref{lemme periodic points}, we get a contradiction, since by assumption, the point $p$ is hyperbolic. 
\end{proof}

As a consequence of Theorem \ref{thm.sardinvariantfunction}, we can prove Theorem \ref{thm.CDtrivial}. 
\begin{proof}[Proof of Theorem \ref{thm.CDtrivial}]
Let $X\in \mathfrak{X}^d(M)$ be as above and let $Y \in  \mathfrak{C}^d(X)$. By the collinearity of $\mathfrak{C}^d(X)$, and since all the  singularities  of $X$ are hyperbolic, Lemma \ref{lemme facile} and Theorem  \ref{thm.quasitriviality} imply that $Y=fX$, where $f$ is a $X$-invariant $C^1$ function such that $f|_{M_X}$ is of class $C^d$. We deduce from  Theorem \ref{thm.sardinvariantfunction} that $f$ is constant. Therefore, $\mathfrak{C}^d(X)$ is trivial. 
\end{proof}

Using the ideas from \cite{Mane}, we are able to prove Theorem \ref{thm.equivanlentcentralizer}.
\begin{proof}[Proof of Theorem \ref{thm.equivanlentcentralizer}]
By Kupka-Smale Theorem (see Theorem $3.1$ in \cite{PalisdeMelo}), there exists an open and dense subset $\mathcal{U}_{KS}\subset \mathfrak{X}^d(M)$ such that for any $X \in \mathcal{U}_{KS}$, any singularity of $X$ is hyperbolic. Let $S(M)$ be the pseudometric space of subsets of $M$ with the Hausdorff
pseudometric. By \cite{Takens}, there exists a residual subset $\mathcal{R}_d\subset \mathfrak{X}^d(M)$ such that the function $\Omega \colon \mathcal{R}_d \to S(M)$ which assigns to $X \in \mathcal{R}_d$ its non-wandering set is continuous.  Let us define the residual set $\mathcal{R}_T:=\mathcal{U}_{KS} \cap \mathcal{R}_d\subset \mathfrak{X}^d(M)$, and let $X \in \mathcal{R}_T$. Notice that $X$ has finitely many singularities, since they are hyperbolic. 

Suppose that $X$ has collinear $C^d$-centralizer and let $Y \in \mathfrak{C}^d(X)$. By the collinearity, as a consequence of Lemma \ref{lemme facile} and Theorem  \ref{thm.quasitriviality}, we have $Y=fX$, for some $X$-invariant $C^1$ function $f$ such that $f|_{M_X}$ is of class $C^d$. Assume that $f$ is non-constant. Then, as in the proof of Theorem \ref{thm.sardinvariantfunction}, $f(M)-f(\mathrm{Zero}(X))$ contains a non-trivial open interval $I\subset \R$. Consider a regular value $c \in I$ (this set is non-empty by Sard's theorem) and let $M_c = f^{-1}(\{c\})$. We now describe Ma\~n\'e's argument from Theorem 1.2 in \cite{Mane}. Let $U$ be a small open neighborhood of $M_c$. Since $\Omega(X) \cap U \neq \emptyset$, by the continuity of $\Omega(\cdot)$ at $X$, for any $X'$ in a neighborhood of $X$ verifies $\Omega(X') \cap U \neq \emptyset$. Consider the gradient $\nabla f|_{M_c}$, since it is nonzero on $M_c$ we can extend it to a vector field $V\colon U \to TU$ without singularities. We can take a $C^1$ vector field $Z$ that is $C^1$-arbitrarily close to the zero vector field, with the following property: for any $x\in U$, $(Z(x), V(x))>0$. For the vector field $X' = X + Z$, it is easy to verify that $\Omega(X') \cap U =\emptyset$, a contradiction. We conclude that $f$ is constant, and thus,  $\mathfrak{C}^d(X)$ is trivial. 
\end{proof}

\subsubsection{$C^2$ flows with hyperbolic measures}
Consider a probability measure $\mu$ on $M$ and $X\in \mathfrak{X}^1(M)$. We say that $\mu$ is $X$-invariant if for any measurable set $A\subset M$ and any $t\in \mathbb{R}$ we have $\mu(A) = \mu(X_t(A))$. By Oseledets theorem, for $\mu$-almost every point $x$, there exist a number $1 \leq l(x)\leq d$ and $l(x)$-numbers $\lambda_1(x)<  \hdots< \lambda_{l(x)}(x) $ with the following properties: there exist $l(x)$-subspaces $E_1(x), \dots, E_{l(x)}(x)$ such that $T_xM = E_1(x) \oplus \cdots \oplus E_{l(x)}(x)$ and for each $i=1, \dots, l(x)$ and for any non zero vector $v\in E_i(x)$ we have
\[
\displaystyle \lim_{t\to \pm \infty} \frac{ \log \|DX_t(x)\cdot v\|}{t} = \lambda_i(x).
\]
The numbers $\lambda_i$ are called Lyapunov exponents. We say that $\mu$ is \emph{non-uniformly hyperbolic} if for $\mu$-almost every point all the Lyapunov exponents are non-zero except the direction generated by the vector field $X$. 

Using Pesin's theory and ideas similar to the proof of Lemma 
\ref{lemme periodic points}, we can prove Theorem \ref{thm.NUH}.

\begin{proof}[Proof of Theorem \ref{thm.NUH}]
Since the support of $\mu$ is the entire manifold, and by non-uniform hyperbolicity, we have that $X$ verifies the conditions of Proposition \ref{thm.expcollinear}, in particular, $\mathfrak{C}^1(X)$ is collinear. Let $Y$ be a vector field in the $C^1$-centralizer of $X$. there exists a $C^1$-function $f\colon M_X \to \mathbb{R}$ such that $Y=fX$ on $M_X$.

Notice that $M_X$ is a connected open and dense subset of $M$. If $f$ were not constant, then it would exist a point $p\in M_X$ such that $\nabla f(p) \neq 0$. Since this condition is open we may take the point $p$ to be a regular point of the measure $\mu$. By Pesin's stable manifold theorem, there exists a $C^1$-stable manifold, $W^s_{\mathrm{loc}}(p)$, which is tangent to $E^-(p) \oplus \langle X(p) \rangle$ on $p$. Similarly, there exists a $C^1$-unstable manifold which on $p$ is tangent to $ \langle X(p) \rangle \oplus E^+(p)$. The non-uniform hyperbolicity implies that $E^-(p) \oplus \langle X(p) \rangle \oplus E^+(p) = T_pM$.

Since $p$ is a non-singular point, we have that $f|_{W^s_{\mathrm{loc}}(p)} = f|_{W^u_{\mathrm{loc}}(p)} = f(p)$. An argument similar to the one in the proof of Theorem \ref{thm.sardinvariantfunction} gives a contradiction and we conclude that $f|_{M_X}$ is constant. This implies that the centralizer of $X$ is trivial. 
\end{proof}


\subsubsection{The $C^3$ centralizer of a $C^3$ Kinematic expansive vector field}

In dimension three, under enough regularity assumptions, we are also able to obtain triviality, for a slightly stronger notion of expansiveness.

\begin{definition}
	\label{Def.kinematic}	
	We say that $X\in\mundo$ is \emph{Kinematic expansive} if for every $\eps>0$ there exists $\delta>0$ such that if $x,y\in M$ satisfy $d(X_t(x),X_t(y))<\delta$, for every $t\in\R$ then there exists $0<|s|<\eps$ such that $y=X_s(x)$. 
\end{definition}  

The difference between the separating property and Kinematic expansiveness is that for the later even points on the same orbit must eventually separate. In \cite{Artigue} it is described a vector field on the M\"obius band which is separating but is not Kinematic expansive. 

Recall that in the statement of Theorem \ref{thm.dim3triviality}, we claim that a $C^3$-kinematic expansive flow in dimension $3$ has trivial $C^3$-centralizer. We remark that the Kinematic expansive condition does not imply that the system admits a countable spectral decomposition. Hence, we cannot use Theorem \ref{theo e} to conclude Theorem \ref{thm.dim3triviality}.

The proof of Theorem~\ref{thm.dim3triviality} is a combination of two results: Sard's Theorem and the proposition below.

\begin{proposition}
\label{p.noaseparatingt2}
Let $\mathbb{T}^2$ denote the two dimensional torus. If $X\in\mathfrak{X}^2(\mathbb{T}^2)$ and if $\zero(X)=\emptyset$ then $X$ is not Kinematic expansive.
\end{proposition} 
\begin{proof}
The argument follows closely some ideas in \cite{Artigue}. We present it here for the sake of completeness.
 	
Assume by contradiction that there exists $X\in\mathfrak{X}^2(\mathbb{T}^2)$ a non-singular kinematic expansive vector field. In particular it is separating. We fix $\eps>0$ to be the separation constant. Since $X$ is $C^2$ we can apply Denjoy-Schwartz's Theorem \cite{denjoyschwartz} and we have three possibilities for the dynamics:
\begin{enumerate}
	\item each orbit is periodic and $X$ is a suspension of the identity map $\mathrm{id}\colon\mathbb{S}^1\to\mathbb{S}^1$;
	\item there exist two distinct periodic orbits $\gamma^s,\gamma^u$ and a non-periodic point $x$ such that $\omega(x)=\gamma^s$ and $\alpha(x)=\gamma^u$;
	\item $X$ is a suspension of a $C^3$ diffeomorphism $f\colon\mathbb{S}^1\to\mathbb{S}^1$, which is topologically conjugate to an irrational rotation.
\end{enumerate} 
We shall prove that each case leads us to a contradiction. In the first case, let $\tau\colon\mathbb{S}^1\to(0,+\infty)$ be the first return time function. Then, $\tau(x)$ is the period of the orbit of $x$. As $\tau$ is a continuous function on the circle, there exists a maximum point $x_0$ and arbitrarily close to $x_0$ there are points $x_1,x_2$ such that $\tau(x_1)=\tau(x_2)$. This implies that one can choose those points so that 
$$d(X_t(x_1),X_t(x_2))\leq\eps,\:\:\forall\:t\in\mathbb{R},$$
a contradiction.

Let us deal now with case (2). Fix an arbitrarily small number $\delta>0$.

Take a small segment $I$ transverse to $X$ at a point $p\in\gamma^s$ and let $f\colon I\to I$ be the first return map, with $\tau\colon I\to(0,+\infty)$ the first return time function. There exists a time $T^s>0$ such that $X_{T^s}(x)\in I$. Consider the fundamental domain $I^s_0=[f(x),x]$  for the dynamics of $f$ and the sequence of image intervals $I^s_n=[f^{n+1}(x),f^n(x)]$, $n\geq 0$. Then, there exists $N^s>0$ such that for $n\geq N^s$, it holds that $I^s_n\subset B(p,\delta)$. Pick $a,b\in I^s_0$ arbitrarily close. 

Let $C>0$ be the Lipschitz constant of $\tau$. Then, 
$$\left|\sum_{\ell=0}^n\tau(f^\ell(a))-\sum_{\ell=0}^n\tau(f^\ell(b))\right|\leq C\sum_{\ell=0}^n|f^\ell(a)-f^\ell(b)|.$$
The hight-hand side of above inequality is bounded by $\sum_n|I^s_n|=|I|<\infty$. Therefore, the left-hand side converges. Moreover, by continuity of $f$, if $d(a,b)$ is small enough then 
$\sum_{\ell=0}^{N^s}|f^\ell(a)-f^\ell(b)|<\delta$. Since $I^s_n\subset B(p,\delta)$ for every $n\geq N^s$, we have $\sum_{\ell=N^s}^\infty|f^\ell(a)-f^\ell(b)|<\delta$. We conclude that 
$$\left|\sum_{\ell=0}^\infty\tau(f^\ell(a))-\sum_{\ell=0}^\infty\tau(f^\ell(b))\right|\leq 2C\delta.$$
Taking $\delta$ small enough, as the flow of $X$ is the suspension of $f$ with return time $\tau$, we conclude that $d(X_t(a),X_t(b))<\eps$, for every $t\geq 0$.

Considering a small transverse segment to a point $q\in\gamma^u$ and arguing similarly with backwards iteration we obtain two arbitrarily close points $a,b$ whose orbits are distinct and such that $d(X_t(a),X_t(b))<\eps$ for every $t\in\mathbb{R}$, a contradiction.

Finally, let us see that case (3) leads to a contradiction. This is essentially contained in the proof of Theorem 4.11 from \cite{Artigue} with a minor adaptation. We will sketch the main points of the proof. Let $f\colon\mathbb{S}^1\to\mathbb{S}^1$ be a $C^3$ diffeomorphism with irrational rotation number $\theta$, and let $\tau\colon\mathbb{S}^1\to(0,+\infty)$ be a $C^1$ function. It is well known that the Lebesgue measure is the only ergodic measure for an irrational rotation. Since $f$ is $C^3$ by the usual Denjoy's theorem on the circle, $f$ is conjugated with an irrational rotation, in particular, $f$ has only one ergodic $f$-invariant probability measure $\mu$. 

Write $T := \int_{S^1} \tau (x) d\mu(x)$ and let $\left(\frac{p_n}{q_n}\right)_{n\in \mathbb{N}}$ be the approximation of $\theta$ by rational numbers given by the continued fractions algorithm. From the corollary in \cite{NavasTriestino}, which is a version of Denjoy-Koksma inequality (Corollary C in \cite{AvilaKocsard}), we obtain the following
\[
\lim_{n\to +\infty}\displaystyle \sup_{x\in S^1} \left|\sum_{l=0}^{q_n-1} \tau(f^l(x)) - Tq_n\right|=0. 
\] 
Following the same calculations in the proof of Theorem 4.11 from \cite{Artigue}, for any $\epsilon>0$ and for $n\in \N$ large enough, the points $x$ and $f^{q_n}(x)$ are always $\epsilon$-close for the future. One can argue similarly for $f^{-1}$ and find points that are not separated for the past. Therefore, the flow cannot be Kinematic expansive. 
\end{proof}

\begin{remark}
	We do not know if there exists a separating suspension of an irrational rotation. The above proof shows that this is the only possibility for a separating non-singular vector field on $\mathbb{T}^2$.
\end{remark}

\begin{proof}[Proof of Theorem \ref{thm.dim3triviality}]
	Since all the singularities are hyperbolic, by Proposition \ref{thm.expcollinear} and Theorem \ref{thm.quasitriviality}, we have that $\mathfrak{C}^3(X)$ is quasi-trivial. Let $f\colon M\to \mathbb{R}$ be a $C^1$, $X$-invariant function such that $f|_{M_X}$ is $C^3$. We will prove that $f$ is constant. Suppose not.
	
	Since there are only finitely many singularities, then as in the proof of Theorem \ref{thm.sardinvariantfunction}, if $f$ were  not constant, we would have 
	$I \subset f(M)-f(\mathrm{Zero}(X))$, for some non-trivial open interval $I \subset \R$.  By Sard's theorem, almost every value in $I$ is a regular value. 
	
	Take a regular value $c\in I$. Hence, $S_c := f^{-1}(\{c\})$ is a compact surface that does not contain any singularity of $X$. Furthermore, since $f$ is $X$-invariant, we have that $X|_{S_c}$ is a $C^3$ non-singular vector field on $S_c$. Up to considering a double orientation covering, this implies that $S_c$ is a torus, since it is the only orientable closed surface that admits a non-singular vector field. 
	
	Notice that $X|_{S_c}$ induces a Kinematic expansive flow. However this contradicts Proposition~\ref{p.noaseparatingt2}. We conclude that $f$ is constant, and this implies that the $C^3$-centralizer of $X$ is trivial.  
\end{proof}

In the higher dimensional case, and at a point of continuity of $\Omega(\cdot)$, we also have:

\begin{proposition}\label{propo sepa triv}
Assume that $X \in \mathfrak{X}^d(M)$ is separating, that all its singularities are hyperbolic, and that $X$ is a point of continuity of the map $\Omega(\cdot)$. Then the $C^d$-centralizer of $X$ is trivial. 
\end{proposition}

\begin{remark}
As noted in the proof of Theorem \ref{thm.equivanlentcentralizer}, the last two assumptions are satisfied by a residual subset of vector fields in $\mathfrak{X}^d(M)$. 
\end{remark}

\begin{proof}[Proof of Proposition \ref{propo sepa triv}]
Since $X$ is separating and its singularitis are hyperbolic, it follows from Proposition \ref{premier theorem} and Theorem \ref{res qu tri} that its $C^1$-centralizer is quasi-trivial. Take any vector field $Y$ in the $C^d$-centralizer of $X$. By the quasi-triviality, and by Lemma \ref{lemme facile}, there exists a $C^1$ function $f\colon M \to\R$ such that $f|_{M_X}$ is of class $C^d$ and $Y=fX$. If $f$ is not constant, then as in the proof of Theorem \ref{thm.equivanlentcentralizer}, by continuity of $\Omega(\cdot)$ at $X$, and by considering a regular value  $c\in f(M)-f(\mathrm{Zero}(X))$ of $f|_{M_X}$, we reach a contradiction. We conclude that the $C^d$-centralizer is trivial.
\end{proof}

%
%
%

\addtocontents{toc}{\protect\setcounter{tocdepth}{1}}
\section{The generic case}\label{section generic}
Our goal in this section is to prove the result below from which Theorem~\ref{theo d} follows immediately.

\begin{theorem}
\label{thm.headinggenericsection}
There exists a residual subset $\mathcal{R} \subset \mathfrak{X}^1(M)$ such that if $X\in \mathcal{R}$ then $X$ has quasi-trivial $C^1$-centralizer. Furthermore, if $X$ has at most countably many chain recurrent classes then its $C^1$-centralizer is trivial.
\end{theorem}
To prove this theorem, we will use a few generic results. In the following statement we summarize all the results we shall need.
\begin{theorem}[\cite{BonattiCrovisier}, \cite{Crovisier}, \cite{PalisdeMelo} and \cite{PughRobinson}]
\label{thm.genericbackground}
There exists a residual subset $\mathcal{R}_* \subset \mathfrak{X}^1(M)$ such that if $X\in \mathcal{R}_*$, then the following properties are verified:
\begin{enumerate}
\item\label{iitem 1} $\overline{\mathrm{Per}(X)} = \Omega(X) = \mathcal{CR}(X)$;
\item\label{iitem 2} every periodic orbit, or singularity, is hyperbolic;
\item\label{iitem 3} if $\mathcal{C}$ is a chain recurrent class, then there exists a sequence of periodic orbits $(\gamma_n)_{n\in \N}$ such that $\gamma_n \to C$ in the Hausdorff topology.
\end{enumerate}
\end{theorem}

Item \eqref{iitem 3} in Theorem \ref{thm.genericbackground} was proved for diffeomorphisms in \cite{Crovisier}. However, the same statement has been used many times for vector fields (for instance \cite{GanYang}). It is folklore that the same proof given by Crovisier in \cite{Crovisier} works for vector fields. In Appendix \ref{append.croflow} we briefly explain this adaptation.

We first prove that $C^1$-generically the centralizer is collinear. This proof is an adaptation for flows of Theorem A in \cite{BonattiCrovisierWilkinson}. Once we have collinearity, using the criterion for quasi-triviality given by Theorem \ref{thm.quasitriviality}, we conclude that quasi-triviality of the $C^1$-centralizer is a $C^1$-generic property.  At the end of this section we will show that for a $C^1$-generic vector field $X$ that has at most countably many chain recurrent classes has trivial $C^1$-centralizer.

\subsection{Unbounded normal distortion and its consequences}
 
The following result is at the core of the $C^1$-generic results obtained in this paper.  Its proof is rather  technical and occupies   Section \ref{proof generic und} below.  
\begin{theorem}
\label{t.unboundednormalgeneric}
There exists a residual subset of $ \mathcal{R} \subset \mathfrak{X}^1(M)$ such that if $X\in \mathcal{R}$ then $X$ has unbounded normal distortion.
\end{theorem}

\subsubsection{Collinearity}

Once we have established Theorem~\ref{t.unboundednormalgeneric}, by combining Proposition \ref{res collin} and some known generic results one obtains the collinearity of the centralizer of a $C^1$-generic vector field.    

\begin{theorem}
	\label{t.collineargeneric}
	There exists a residual subset of $ \mathcal{R} \subset \mathfrak{X}^1(M)$ such that if $X\in \mathcal{R}$ then the $C^1$-centralizer of $X$ is collinear.
\end{theorem}
\begin{proof}
The result follows directly from Proposition \ref{res collin} and Theorems \ref{thm.genericbackground}-\ref{t.unboundednormalgeneric}.
%
%
%
\end{proof}

\subsubsection{Quasi-triviality}
By Theorem \ref{thm.genericbackground}, we have that $C^1$-generically all the singularities are hyperbolic. As a consequence of Theorem \ref{thm.quasitriviality}, since $C^1$-generically the $C^1$-centralizer is collinear and all the singularities are hyperbolic, we conclude that $C^1$-generically the $C^1$-centralizer is quasi-trivial. More precisely, we have 

\begin{theorem}\label{premier cor}
Let $M$ be  a compact manifold. There exists a residual subset $ \mathcal{R}_1 \subset \mathfrak{X}^1(M)$ such that if $X\in \mathcal{R}_1$, then any singularity and periodic orbit of  $X$ is hyperbolic, $\overline{\mathrm{Per}(X)}=\Omega(X) = \mathcal{CR}(X)$, and 
$$
\mathfrak{C}^1(X)=\{f X: f \in \mathfrak{I}^1(X)\},\text{ where }  \mathfrak{I}^1(X)=\{f\in C^1(M,\R),\ X\cdot f\equiv 0\}. 
$$ 
\end{theorem}

\begin{proof}
By Theorem \ref{t.collineargeneric}, there exists a residual subset $\mathcal{R}\subset \mathfrak{X}^1(M)$ whose elements have collinear $C^1$-centralizer. Moreover, by Theorem \ref{thm.genericbackground}, there exists a residual subset $\mathcal{R}_*\subset \mathfrak{X}^1(M)$ such that for any $X \in \mathcal{R}_*$, any singularity and periodic orbit of  $X$ is hyperbolic, and $\overline{\mathrm{Per}(X)}=\Omega(X) = \mathcal{CR}(X)$. Then, $\mathcal{R}_1:= \mathcal{R}\cap  \mathcal{R}_*$ is residual, and any $X \in \mathcal{R}_1$ satisfies the hypotheses of Theorem \ref{thm.quasitriviality}, which concludes. 
\end{proof}

\subsubsection{Triviality}

We can now conclude the proof of the second part of Theorem \ref{thm.headinggenericsection} about $C^1$-generic triviality for systems with a countable number of chain recurrent classes. To prove that we need the following lemma.

\begin{lemma}
\label{lem.gentrivialitylemma1}
There exists a residual subset $\mathcal{R}_{\mathcal{CR}} \subset \mathfrak{X}^1(M)$ such that if $X\in \mathcal{R}_{\mathcal{CR}}$ and $f\in C^0(M)$ is an $X$-invariant function, then $f$ is constant on chain-recurrent classes.
\end{lemma}
\begin{proof}
By Theorem $1$ in \cite{Crovisier}, there exists a residual subset $\mathcal{R}_{\mathcal{CR}}\subset \mathfrak{X}^1(M)$ that verifies the following: if $X\in \mathcal{R}_{\mathcal{CR}}$ and $C\subset \mathcal{CR}(X)$ is a chain-recurrent class, then there exists a sequence of periodic orbits $(O(p_n))_{n\in \mathbb{N}}$ that converges to $C$ in the Hausdorff topology.

By this property, for any two points $x,y\in C$, there exist two sequences of points $(q_n)_{n\in \mathbb{N}}$ and $(q_n')_{n\in \mathbb{N}}$, with $q_n, q_n' \in O(p_n)$, such that $q_n \to x$ and $q_n' \to y$ as $n \to +\infty$. Let $f$ be a continuous function which is $X$-invariant. By continuity,
\[
\displaystyle \lim_{n\to +\infty} f(q_n) = f(x) \textrm{ and } \lim_{n\to +\infty} f(q_n') = f(y).
\]
However, since $f$ is $X$-invariant and by our choice of $q_n$ and $q_n'$, we have that $f(p_n) = f(q_n) = f(q_n')$, which implies that $f(x) = f(y)$.  
\end{proof}

\begin{proof}[Proof of Theorem \ref{thm.headinggenericsection}]
Take $\mathcal{R}: = \mathcal{R}_1 \cap \mathcal{R}_{\mathcal{CR}}$, where $\mathcal{R}_1$ is the residual subset given by Theorem \ref{premier cor}. Using the conclusion of Lemma \ref{lem.gentrivialitylemma1} and arguments analogous to the proof of Theorem \ref{thm.continuousfinitetrivial} we can easily obtain the conclusion of Theorem \ref{thm.headinggenericsection}. 
\end{proof}

\section{Proof that the unbounded normal distortion is $C^1$-generic}\label{proof generic und}

In this part, we give the proof of Theorem \ref{t.unboundednormalgeneric} about the $C^1$-genericity of the unbounded normal distortion property. Since this section is very technical, let us first summarize the main steps of the proof.\\

\paragraph{\textbf{Idea of the proof}}
In \cite{BonattiCrovisierWilkinson} the authors prove that a version of the unbounded normal distortion holds $C^1$-generically for diffeomorphisms. Their proof can be divided in two steps. The first part is a key perturbative result made on a linear cocycle over $\mathbb{Z}$ (see Proposition~\ref{prop.proposition9} below). Then, they reduce the proof to this linear cocycle scenario, by using some change of coordinates that linearises the dynamics around an orbit segment of finite length. Both steps are quite delicate and involve careful control of estimates which appear along the way.



Our strategy is also to reduce the problem to a perturbation of a linear cocycle over $\mathbb{Z}$, and then apply the result of \cite{BonattiCrovisierWilkinson}, so that we only need to translate to the vector field scenario the second step of Bonatti-Crovisier-Wilkinson's proof. It is clear that, in order to do that, one needs to discretize the dynamics and so we study the Poincar\'e maps between a sequence of transverse sections. The goal is then apply the perturbation of \cite{BonattiCrovisierWilkinson} to the Poincar\'e maps. The key difficulty we face with this strategy is that we need to prove that any finite family of perturbations of a long sequence of Poincar\'e maps, which verifies some conditions, can be realized as the corresponding sequence of Poincar\'e maps for a perturbed vector field. 

Apart from that, as in \cite{BonattiCrovisierWilkinson}, all the perturbations in the reduction procedure have to be done with precise control on the estimates that appear. 

It is also important to point out that, since we are dealing with wandering points, the Poincar\'e map can be defined for a sequence of times arbitrarily large. We also introduce some change of coordinates to linearize the dynamics given by these maps for a finite time. However, the space where this can be defined is no longer compact, since the Poincar\'e map is only defined over non-singular points. Nevertheless, we can obtain uniform estimates for the $C^1$-norm of these change of coordinates.

Once we have the realisation lemma, we can adapt the proof of Bonatti-Crovisier-Wilkinson in \cite{BonattiCrovisierWilkinson} and obtain that the unbounded normal distortion property is $C^1$-generic.

\subsubsection{Notation}
We summarize the main notations that will appear below; here, we let $X$ be a $C^1$ vector field, $p$ be a point in $M_X$, and $n \geq 0$ be an  integer:
\begin{itemize}
	\item $N_{X,p}$:  subspace of $T_p M$ orthogonal to $X(p)$;
	\item $\mathcal{N}_{X,p}$: image of $N_{X,p}$ under the exponential map $\exp_p$;
	\item  $(P_{p,t}^X)_{t \in \R}$: linear Poincar\'e flow, where for $t \in \R$, $P_{p,t}^X$ is the map  induced by $DX_t(p)$ between  $N_{X,p}$ and  $N_{X,X_t(p)}$;
	\item $\mathcal{P}_{X,p,n}^Y$: Poincar\'e map between $\mathcal{N}_{X,p}$ and $\mathcal{N}_{X,X_n(p)}$ for the flow generated by a $C^1$ vector field $Y$ close to $X$;   when $X=Y$, we drop $X$ in the bottom;
	\item $\widetilde{\mathcal{P}}^X_{p,n} = \exp_{X_n(p)}^{-1} \circ \mathcal{P}^X_{p,n} \circ \exp_p$: lifted Poincar\'e map;
	\item  $\mathcal{I}^X(p,U,n)$: set of pairs $(y,t)$ with $y$ in some  small neighborhood $U$ of $p$ in $\mathcal{N}_{X,p}$, and $t\geq 0$ a time before the trajectory through $y$ first hits $\mathcal{N}_{X,X_n(p)}$; 
	\item $\mathcal{U}^X(p,U,n)$:   image of $\mathcal{I}^X(p,U,n)$ in phase space; in other words, it is the tube of flow lines from $U$ to $\mathcal{N}_{X,X_n(p)}$;
	\item $\psi_{p,n}$: linearizing coordinates (to go from lifted  to linear Poincar\'e flow);
	\item $\Psi_{p,n}:=\psi_{p,n} \circ \exp_{p}^{-1}$.
\end{itemize}

\subsubsection{Linearizing coordinates} 

Let $X\in \mathfrak{X}^1(M)$, and as before, set  $M_X := M-\zero(X)$. For $p\in M_X$ and $t\in \mathbb{R}$, for any two submanifolds $\Sigma_1$ and $\Sigma_2$ which are transverse to the orbit segment  $O: = X_{[0,t]}(p)$,  each of which intersects $O$ only at one point, we define the \emph{Poincar\'e map} between these two transverse sections as follows: let $p_1 := O \cap \Sigma_1$ and $p_2 := O \cap \Sigma_2$. If a point $q\in \Sigma_1$ is sufficiently close to $p_1$, then $X_{[-t,2t]}(q)$ intersects $\Sigma_2$ at a unique point  $\mathcal{P}_{\Sigma_1,\Sigma_2}^X(q)$. The map $q\mapsto \mathcal{P}_{\Sigma_1,\Sigma_2}^X(q)$ is called the Poincar\'e map between $\Sigma_1$ and $\Sigma_2$.

This map is a $C^1$-diffeomorphism between a neighborhood of $p_1$ in $\Sigma_1$ and its image in $\Sigma_2$. It also holds that for any vector field $Y\in \mathfrak{X}^1(M)$ sufficiently $C^1$-close to $X$, the Poincar\'e map $\mathcal{P}_{\Sigma_1,\Sigma_2}^Y$ for $Y$ is well defined in some neighborhood of $p_1$ in $\Sigma_1$.  

Let $R>0$ be smaller than the radius of injectivity of $M$. Using the exponential map, for each $p\in M_X$ and $r\in (0, R)$, we define the submanifold $\mathcal{N}_{X,p}(r):= \exp_p( N_{X,p}(r))$, where $N_{X,p}(r)$ is the ball of center $0$ and radius $r$ contained in $N_{X,p}$. 

\begin{remark}
	\label{rmk.closetonormal}
	Considering $R$ to be small enough, for each $p\in M_X$ and for each $q\in \mathcal{N}_{X,p}(R)$ we have that the $C^1$-norm of $\Pi^X_q|_{T_q\mathcal{N}_{X,p}(R)}$ is close to $1$. 
\end{remark}

It is a result from \cite{GanYang} that for each $t\in \mathbb{R}$, there exists a constant $\beta_t = \beta(X,t)>0$ such that for any point $p\in M_X$, the Poincar\'e map is a $C^1$ diffeomorphism from $\mathcal{N}_{X,p}(\beta_t \|X(p)\|)$ to its image inside $\mathcal{N}_{X,X_t(p)}(R)$. We denote this map by $\mathcal{P}_{p,t}^X$. For a fixed $\delta>0$, we can choose $\beta \in (0,\beta_1)$ such that for any $p\in M_X$ and any $q\in \mathcal{N}_{X,p}(\beta\|X(p)\|)$, it holds
\begin{equation}
\label{eq.poincareclose}
\|D\mathcal{P}^X_{p,1}(q) - D\mathcal{P}^X_{p,1}(p)\| < \delta.
\end{equation}
The existence of $\beta_t$ and $\beta$ above is guaranteed by Lemmas $2.2$ and $2.3$ in \cite{GanYang}. It follows from the proof that these constants can be taken uniformly in a sufficiently small $C^1$-neighborhood of $X$. By our choices of transversals, we remark that $D\mathcal{P}_{p,1}^X(p) = P^X_{p,1}$, where $P^X_{p,t}$ is the linear Poincar\'e flow (see \eqref{eq.LPF}). 

\begin{definition}
	\label{def.boundedbyc}
	For any $C>1$ we say that a vector field $X\in \mathfrak{X}^1(M)$ is \emph{bounded by $C$} if it holds
	\begin{enumerate}[label=(\alph*)]
		\item\label{condition aa} $\displaystyle \sup_{x\in M} |X(p)| < C$;\\
		\item\label{condition bb} $ \sup_{p\in M} \|DX(p)\| < C$;\\
		\item\label{condition cc} $ C^{-1}< \displaystyle \inf_{x\in M} \inf_{t\in [-1,1]} \|(DX_t(x))^{-1}\|^{-1} \leq \sup_{x\in M} \sup_{t\in [-1,1]} \|DX_t(x)\| < C;$
		\item\label{condition dd} $C^{-1} < \displaystyle \inf_{p\in M_X} \inf_{t\in [-1,1]} \|(P^X_{p,t})^{-1}\|^{-1} \leq \sup_{p\in M_X} \sup_{t\in [-1,1]} \|P^X_{p,t}\| < C$;
		\item\label{condition ee} there exists $\beta>0$ small, such that 
		\[
		C^{-1} < \|(D\mathcal{P}_{p,1}^X(q))^{-1}\|^{-1} \leq \|D\mathcal{P}_{p,1}^X(q)\| < C, \textrm{ for any $q\in \mathcal{N}_{X,p}(\beta \|X(p)\|)$.}
		\]
	\end{enumerate}
\end{definition}
Next lemma justifies that for any $C^1$-vector field there is a constant $C>1$ such that this vector field is bounded by $C$.

\begin{lemma}\label{lem.boundedbyc}
	Let $X\in \mathfrak{X}^1(M)$ be a vector field such that $M_X \neq \emptyset$. There exists $C>1$ such that $X$ is bounded by $C$. Moreover, this constant can be taken uniform in a sufficiently small neighborhood of $X$.
	
\end{lemma}
\begin{proof}
	Since the set $[-1,1] \times M$ is compact, the existence of a constant $C>1$ that verifies Conditions \ref{condition aa}, \ref{condition bb} and \ref{condition cc} in Definition \ref{def.boundedbyc} is immediate. 
	
	Let us justify Condition \ref{condition dd}. Recall that on $M_X$ the linear Poincar\'e flow is defined by $P^X_{p,t} = \Pi^X_{X_t(p)} \circ DX_t(p)|_{N_{X,p}}$. By Condition \ref{condition cc} and since $\|P^X_{p,t}\| \leq \|DX_t(p)\|$ we conclude the upper bound in Condition \ref{condition dd}. Since $M_X$ is not compact, the possible problem that could appear for the lower bound in Condition \ref{condition dd} is if the angle between the hyperplane $DX_t(p)N_{X,p}$ and $X(X_t(p))$ is not bounded from below for $t\in [-1,1]$ and $p\in M_X$. 
	
	Let $SM$ be the unit tangent bundle of $M$ and consider the application $F\colon SM \times \R  \to  \R$ defined by
	\begin{equation}\label{eq.defF}
	F(p,v,t)= \measuredangle(DX_t(p)\cdot v_p^\perp, DX_t(p)\cdot  v),
	\end{equation}
	where $v_p^\perp$ is the $d-1$-dimensional subspace in $T_pM$ which is orthogonal to the vector $v$, and $\measuredangle(DX_t(p)\cdot v_p^\perp, DX_t(p)\cdot v)$ is the angle between the subspace $DX_t(p) \cdot v_p^\perp$ and the vector $DX_t(p)\cdot v$.
	
	By the continuity of $DX_t$, we have that the map $F$ is also continuous. Moreover, for each $p\in M$ and $t\in \R$ the map $DX_t(p)$ is an isomorphism between $T_pM$ and $T_{X_t(p)}M$, therefore, $F(p,v,t)$ is positive for any $(p,v) \in SM$ and $t\in \R$. Define $\gamma := \inf \{F(p,v,t): (p,v,t) \in SM \times [-1,1]\}$, and observe that since $SM \times [-1,1]$ is compact and $F$ is continuous, $\gamma$ is strictly positive. 
	
	Note that if $p\in M_X$, we have that $F(p, \frac{X(p)}{\|X(p)\|}, t) = \measuredangle (DX_t(p) \cdot X_p^\perp, \frac{X(X_t(p))}{\|X(p)\|})$. Thus, for  $(p,t)\in M_X\times    [-1,1]$, we obtain $F(p,\frac{X(p)}{\|X(p)\|},t)>\gamma$. This together with the lower bound in Condition \ref{condition cc} gives the uniform lower bound in Condition \ref{condition dd}. 
	
	Condition \ref{condition ee} follows easily from Condition \ref{condition dd} and  (\ref{eq.poincareclose}) above. We conclude that for any vector field $X\in \mathfrak{X}^1(M)$, there is a constant $C>1$ such that $X$ is bounded by $C$, and the same is true for every $Y$ close enough to $X$.  \qedhere
	
\end{proof}

Let $X\in \mathfrak{X}^1(M)$ be a vector field bounded by $C>1$. Using the exponential map, for $p \in M_X$, we consider the \emph{lifted Poincar\'e map} 
\[
\widetilde{\mathcal{P}}^X_{p,1} = \exp_{X_1(p)}^{-1} \circ \mathcal{P}^X_{p,1} \circ \exp_p,
\]
which goes from $N_{X,p}(\beta \|X(p)\|)$ to $N_{X,X_1(p)}(R)$. The advantage of using the lifted Poincar\'e map is that we can perform perturbations using canonical coordinates. Observe that 
\begin{equation}
\label{eq.lowerboundvectorfield}
\|X(X_1(p))\|> C^{-1} \|X(p)\|.
\end{equation}
By (\ref{eq.lowerboundvectorfield}) and the last item in Definition \ref{def.boundedbyc}, for any $n \in \N$, the map $\mathcal{P}^X_{p,n}$ is well defined on $\mathcal{N}_{X,p}\big(\frac{\beta}{C^{n}}\|X(p)\|\big)$, while the lifted map $\widetilde{\mathcal{P}}^X_{p,n}$  is well defined on $V^X_{p,n} :=N_{X,p}\big(\frac{\beta}{C^{n}}\|X(p)\|\big)$. 

For each $n\in \N$ and $p\in M_X$, we define the change of coordinates $\psi_{p,n} = P^X_{X_{-n}(p),n} \circ (\widetilde{\mathcal{P}}^X_{X_{-n}(p),n})^{-1}$, which is a $C^1$ diffeomorphism from $\widetilde{\mathcal{P}}^X_{p,n}(V^X_{X_{-n}(p),n})$ to $P^X_{X_{-n}(p),n}(V^X_{X_{-n}(p),n})\subset N_{X,p}$. Observe that $\psi_{p,0} = \mathrm{id}$. The sequence $(\psi_{X_j(p),j})_{j\in \N}$  verifies the following equality:
\[
\psi_{X_n(p),n} \circ \widetilde{\mathcal{P}}^X_{p,n} = P^X_{p,n} \circ \psi_{p,0},
\]
which holds on $V_{p,n}^X$. In other words, this change of coordinates linearizes the dynamics of $\widetilde{\mathcal{P}}^X_{p,n}$:
\begin{mycapequ}[H]
\begin{equation*}
\ \xymatrixcolsep{3pc}\xymatrix{
	V_{X_{-1}(p),n+1} \ar[d]_-{\psi_{X_{n+1}(p),n+1}} \ar[r]^-{\widetilde{\mathcal{P}}^{X}_{X_{-1}(p),1}} & V_{p,n} \ar[d]_-{\psi_{X_n(p),n}} \ar[r]^-{\widetilde{\mathcal{P}}^{X}_{p,n}} & N_{X_{n}(p)}(R) \ar[d]^-{\imath_{p,n}}\\
	N_{X_{-1}(p)} \ar[r]_-{P^X_{X_{-1}(p),1}}  & N_p \ar[r]_-{P^{X}_{p,n}} & N_{X_{n}(p)} }
\end{equation*}
	\caption{Change of coordinates}
\end{mycapequ}
\noindent where $\imath_{p,n}\colon N_{X_{n}(p)}(R) \to N_{X_{n}(p)}$ stands for the inclusion map. 

For all $y \in \mathcal{N}_{X,p}(\frac{\beta}{C^n}\|X(p)\|)$, we define the \textit{hitting time} $\tau^X_{p,n}(y)$ as the first positive time where the trajectory starting at $y$ hits the transverse section $\mathcal{N}_{X,X_n(p)}(R)$:
$$
\tau^X_{p,n}(y):=\inf\{t \geq 0: X_t(y)\in \mathcal{N}_{X,X_n(p)}(R)\}.
$$

\begin{notation}
	Let $p\in M_X$ and $n\in \N$. Suppose that for $Y\in \mathfrak{X}^1(M)$ the submanifolds $\mathcal{N}_{X,p}\big(\frac{\beta}{C^{n}}\|X(p)\|\big)$ and $\mathcal{N}_{X,X_n(p)}(R)$ are transverse to $Y$, and that the Poincar\'e map for $Y$ between these transverse sections is well defined on $\mathcal{N}_{X,p}\big(\frac{\beta}{C^{n}}\|X(p)\|\big)$. Then we denote this Poincar\'e map for $Y$ by $\mathcal{P}^Y_{X,p,n}$. Accordingly, we denote its lift by $\widetilde{\mathcal{P}}^Y_{X,p,n}$ and its hitting time by $\tau^Y_{X,p,n}$. We also extend those notations for non-integer times: given an integer $n \geq 1$ and $t \in [n-1,n]$, we let $\mathcal{P}^Y_{X,p,t}$ be the  Poincar\'e map between the transversals $\mathcal{N}_{X,p}\big(\frac{\beta}{C^{n}}\|X(p)\|\big)$ and $\mathcal{N}_{X,X_t(p)}(R)$. 
\end{notation}

In the next definition we introduce   the type of perturbations of the Poincar\'e map that we will consider in the sequel. Observe that, in this definition, we are perturbing the nonlinear transverse dynamics of the flow.
\begin{definition}
	\label{def.perturbationpoincare}
	For each $\delta>0$ and given an open set $U\subset \mathcal{N}_{X,p}(\beta \|X(p)\|)$, a $C^1$ map $g\colon \mathcal{N}_{X,p}(\beta \|X(p)\|) \to \mathcal{N}_{X,X_1(p)}(R)$ is called a $\delta$\textit{-perturbation of $\mathcal{P}^X_{p,1}$ with support in $U$} if the following holds:
	\begin{itemize}
		\item $d_{C^1}(\mathcal{P}^X_{p,1},g)<\delta$;
		\item the image of $g$ coincides with the image of $\mathcal{P}^X_{p,1}$;
		\item the map $g$ is a $C^1$ diffeomorphism into its image;
		\item the support of $(\mathcal{P}^X_{p,1})^{-1} \circ g$ is contained in $U$.
	\end{itemize}   
\end{definition}

For any $n\in \N$ and any $U\subset \mathcal{N}_{X,p}\big(\frac{\beta}{C^n}\|X(p)\|\big)$, we define
\begin{equation}\label{defi set Idef}
\mathcal{I}^X(p,U,n):=\{(y,t): y\in U,\ t\in [0, \tau^X_{p,n}(y)]\},
\end{equation}
and we let $\mathcal{U}^X(p,U,n)$ be the image of $\mathcal{I}^X(p,U,n)$ under the map $(y,t) \mapsto X_t(y)$:
\begin{equation}\label{defi set Urond}
\mathcal{U}^X(p,U,n):=\displaystyle \bigcup_{y\in U} \bigcup_{t\in [0, \tau^X_{p,n}(y)]} X_t(y).
\end{equation}

We will need the following lemma, which translates the fact that $X_t\to Id$ when $|t|\to 0$ in terms of the linear Poincar\'e map and the hitting time.

\begin{lemma}\label{lem.justify}
	Let $X\in \mathfrak{X}^1(M)$. There exists a small constant $\alpha= \alpha(X)>0$ such that for any $t\in [-\alpha, \alpha]$ and $p\in M_X$, it holds that $|\det P^X_{p,t}-1| < \frac{\log 2}{2}$. Furthermore, take $C>1$ to be a constant such that $X$ is bounded by $C$. Then, we can fix $\beta>0$ small such that for any $p\in M_X$ and $q\in \mathcal{N}_{X,p}\big(\frac{\beta}{C^n}\|X(p)\|\big)$, it holds that $\tau^X_{p,n}(q) \in [n-\alpha, n+\alpha]$.
\end{lemma}
\begin{proof}
	
	First observe that for $t$ small, $DX_t$ is uniformly close to $DX_0=\mathrm{id}$ for any point in $M$. Using the continuity of the function $F$ defined in (\ref{eq.defF}), for any $\varepsilon>0$, there exists $\alpha_1>0$ such that for any $t\in [-\alpha_1,\alpha_1]$ we have the following: for any $p\in M_X$, the angle between $DX_t(p) N_{X,p}$ and $N_{X,X_t(p)}$ is smaller than $\varepsilon$. Since $\det P^X_{p,t} = \det\Pi^X_{X_{t}(p)}|_{DX_t(p)N_{X,p}} . \det DX_t(p)|_{N_{X,p}}$, from these two observations above, by fixing $\alpha$ small enough we conclude the first part of Lemma \ref{lem.justify}. 
	
	Let us now prove the second part of the lemma. We may take $\beta$ much smaller than $\alpha$ such that $\beta C \ll 10\alpha C$ and $\beta C\ll R$. This implies that for any $p\in M_X$ we have that $B(p, \beta\|X(p)\|) \subset \displaystyle \bigcup_{t\in [-\alpha,\alpha]} X_t(\mathcal{N}_{X,p}(R))$. Observe that for any $n\in \N$, if $q\in \mathcal{N}_{X, p}(\frac{\beta}{C^n}\|X(p)\|)$, then we have
	\[
	d(X_n(p), X_n(q)) \leq C^nd(p,q) \leq \beta\|X(p)\|.
	\] 
	The conclusion then follows. 
\end{proof}

\begin{remark}
	From now on given $X$ we will always assume that $\alpha$ and $\beta$ verify the conclusion of Lemma \ref{lem.justify}. 
\end{remark}

\subsubsection{A realization lemma}

We state and prove below a lemma that allows us to realize a non-linear perturbation of the linear Poincar\'e flow as the lifted Poincar\'e map of a vector field nearby. This result is one of the key differences between our work and the diffeomorphism result of \cite{BonattiCrovisierWilkinson}. Moreover, it provides an efficient of converting the problem of perturbing a vector field into giving a sequence of perturbations of a discrete dynamics. 

\begin{lemma}
	\label{lemma.fundamentalperturbation}
	For any $C,\varepsilon >0$, there exists $\delta=\delta(C,\varepsilon)>0$ that verifies the following. For any vector field $X\in \mathfrak{X}^1(M)$ that is bounded by $C$, any $0< \delta_1< \delta$ and any integer $n\in \N$, there is $\rho=\rho(X,\varepsilon,\delta_1,n)>0$ with the following property.
	
	For any $p\in M_X$ and $U \subset \mathcal{N}_{X,p}(\rho\|X(p)\|)$   such that the map $(y,t) \mapsto X_t(y)$ is injective restricted to the set $\mathcal{I}^X(p,U,n)$, then the following holds:
	\begin{enumerate}
		\item\label{item un} Set $\widetilde{U} := \exp^{-1}_p(U)$. Then for every $i\in \{0, \dots, n\}$, the map $\Psi_{X_i(p),i} := \psi_{X_i(p), i} \circ \exp^{-1}_{X_i(p)}$ induces a $C^1$ diffeomorphism from $\mathcal{P}^X_{p,i}(U)$ onto $P^X_{p,i}(\widetilde{U})$ such that 
		\begin{equation}\label{controle Psi}
		\max \{ \|D\Psi_{X_i(p),i}\|, \|D\Psi^{-1}_{X_i(p),i}\|, |\det D\Psi_{X_i(p),i}|, | \det D\Psi^{-1}_{X_i(p),i}|\} <2.
		\end{equation}
		
		\item\label{item deux} For $i\in \{1, \dots, n\}$, let $\tilde{g}_i\colon N_{X,X_{i-1}(p)} \to N_{X,X_i(p)}$ be any $C^1$ diffeomorphism  such that the support of $(P^X_{X_{i-1}(p),1})^{-1}\circ \tilde{g}_i$ is contained in $P^X_{p,i-1}(\widetilde{U})$, and which satisfies $d_{C^1}(\tilde{g}_i, P^X_{X_{i-1}(p),1}) < \delta_1$. Let $g_i$  be the map defined as follows:
		\begin{itemize}
			\item $g_i(y) := \mathcal{P}^X_{X_{i-1}(p),1}(y)$, if $y\notin \mathcal{P}^X_{p,i-1}(U)$;
			\item $g_i(y) := \Psi^{-1}_{X_i(p),i} \circ \tilde{g}_i \circ \Psi_{X_{i-1}(p), i-1}(y)$, if $y\in \mathcal{P}^X_{p,i-1}(U)$.
		\end{itemize}
		Then the map $g_i$ is a $\delta$-perturbation of $\mathcal{P}^X_{X_{i-1}(p),1}$ with support in $\mathcal{P}^X_{p,i-1}(U)$.  
		
		\item\label{item trois}  There exists $Y\in \mathfrak{X}^1(M)$ such that $d_{C^1}(X,Y) < \varepsilon$, and the Poincar\'e map $\mathcal{P}^Y_{X,X_i(p),1}$ for the vector field $Y$ between $\mathcal{N}_{X_{i-1}(p)}(\rho \|X(p)\|)$ and $\mathcal{N}_{X_i(p)}( R)$ is well defined and is given by $g_i$, for each $i\in \{1, \dots, n\}$. Moreover, the support of $X-Y$ is contained in $\mathcal{U}^X(p,U,n)$ and the image of $\tau^Y_{X,p,n}$ is contained in $[n-\alpha, n+\alpha]$. 
	\end{enumerate} 
	
\end{lemma}
Before proving this lemma, let us say a few words on Items \eqref{item deux} and \eqref{item trois} in the statement. Item \eqref{item deux} states that we can obtain perturbations of the Poincar\'e map by perturbing its lift, with precise estimates on the size of each of these perturbations we consider. Observe that this only gives $C^1$ diffeomorphisms between certain transverse sections. Item \eqref{item trois} states that any such perturbation can be realized as the Poincar\'e map of a vector field $C^1$-close to $X$, with precise estimates on its distance to $X$. Furthermore, the  hitting time of $Y$ has the same image as the hitting time of $X$. 

This lemma will be very important in our proof. It will allow us to reduce the proof of the theorem to the perturbation of a linear cocycle over $\Z$. This is done after several steps and adaptations. One important remark is that we will find some number $n$ such that throughout our proof, the perturbations will happen in pieces of orbit of ``size'' $[0,n]$. So Lemma \ref{lemma.fundamentalperturbation} will give us the uniformity needed to realize the pertubations of the linear cocycles as the Poincar\'e map of a vector field.

\begin{proof}
	We will obtain $\delta$ later, as consequence of a finite number of inequalities. In the following, we always assume that $0<\rho\leq \frac{\beta}{C^n}$. By the previous discussion, this ensures that $\mathcal{P}_{p,n}^X$ is well defined on $\mathcal{N}_{X,p}\big(\rho\|X(p)\|\big)$, for all $p \in M_X$.
	
	For Item \eqref{item un}, first observe that 
	\begin{equation}\label{eq.chainrule}
	D\Psi_{X_i(p),i} = P^X_{p,i} \circ  D (\widetilde{\mathcal{P}}^X_{p,i})^{-1} \circ D \exp^{-1}_{X_i(p)}.
	\end{equation}
	Notice that since $i\in \{0,\cdots, n\}$, for any $l>0$ one may fix $\rho>0$ sufficiently small such that for any $p\in M_X$ we have $\rho \|X(X_i(p))\|< l$. In other words, the submanifolds $\mathcal{N}_{X,X_i(p)}(C^i \rho \|X(x)\|)$ can be made uniformly arbitrarily small. In particular, we obtain that for any $\gamma>0$, for $\rho$ sufficiently small, for any $p\in M_X$ the map $D\exp_{X_i(p)}^{-1}$ is $\gamma$-$C^1$-close to the identity in the ball of center $X_i(p)$ and radius $C^i \rho\|X(p)\|$. Hence, to control $D\Psi_{X_i(p),i}$ we are left to control the term $P^X_{p,i} \circ  D (\widetilde{\mathcal{P}}^X_{p,i})^{-1}$. This will follow from the following points:
		
	\begin{itemize}
		\item It holds
		\begin{equation}\label{control Ppt}
		\displaystyle C^{-n} < \inf_{p\in M_X,\ t\in [-n,n]} \|(P^X_{p,t})^{-1}\|^{-1} \leq \sup_{p\in M_X,\ t\in [-n,n]} \|P^X_{p,t}\| < C^n,
		\end{equation}
		and by \eqref{eq.poincareclose}, we have similar estimates for the Poincar\'e maps $\mathcal{P}^{X}_{p,t}$, uniformly in  $p \in M_X$ and $t\in [-n,n]$.
		\item 
		
		In order to have a uniform control on the norm of $D\Psi_{X_i(p),i}$ and its inverse, the difficulty to overcome is that the set $M_X$ where $p$ ranges is not compact. Let us note that by \eqref{eq.chainrule}, it is sufficient to control the map $P^X_{p,i} \circ D((\widetilde{\mathcal{P}}^X_{p,i})^{-1})$; moreover, by choosing $\rho>0$ sufficiently small, the linear maps $A=P^X_{p,i}$ and $B=D\widetilde{\mathcal{P}}^X_{p,i}$ from $N_p$ to $N_{X_i(p)}$ can be made arbitrarily close to each other. The point is thus to control the product $AB^{-1}$  knowing that the norm of the difference $A-B$ is small. The idea for  that is to find an extension in order to view $A$ and $B$ as two maps between the fibers at $p$ and $q:=X_i(p)$ of a compact bundle. We consider the  following commutative diagram:
\begin{equation*}
\ \xymatrixcolsep{3pc}\xymatrix{
	\mathrm{Gr}^{d-1}_p M  \ar[r]^-{D_p X_i} & \mathrm{Gr}^{d-1}_{q}  M \\
	 S_p M \ar[u]^-{\imath_p}  \ar[r]^-{D_p X_i} & \ar[u]_-{\imath_{q}}  S_{q}  M \\
	 N_p \ar[u]^-{\jmath_p} \ar[r]^-{P_{p,i}^X} & \ar[u]_-{\jmath_{q}} N_{q} }
\end{equation*}
where  $S M$, $\mathrm{Gr}^{d-1} M$ respectively denote the unit tangent bundle  and the bundle  of $(d-1)$-Grassmannians, $\imath_p$ is the natural inclusion from $S_p M$ to $\mathrm{Gr}_p^{d-1} M$, and  $\jmath_{p}$ denotes the map which associates to the hyperplane $N_p$  the unit normal vector to it (i.e., the unit vector tangent to  the flow). Here, by a slight abuse of notation, we denote by $D X_i$ the natural action of the differential of the flow both on $S M$ and  $\mathrm{Gr}^{d-1} M$.  By composition with the map  $\imath_p \circ \jmath_p$, we can see $P_{p,i}^X$ and $D\widetilde{\mathcal{P}}^X_{p,i}$ as two maps between the fibers at $p$ and $q$ of $\mathrm{Gr}^{d-1} M$. Besides, as   $\mathrm{Gr}^{d-1} M$ is compact,  for any $C>0$, and for any $\varepsilon>0$, there exists $\delta>0$ such that for any $p,q \in M$, 
\begin{align*}
\forall\, A \colon \mathrm{Gr}_p^{d-1} M \to \mathrm{Gr}_q^{d-1} M,\quad  & \|A\|< C,\\
\forall\, B \colon \mathrm{Gr}_p^{d-1} M \to \mathrm{Gr}_q^{d-1} M,\quad  & \|B\|< C,
\end{align*}
it holds:
$$
\|A-B\|< \delta \quad \Rightarrow \quad \|A B^{-1}- \mathrm{id}\|< \varepsilon. 
$$
From this fact, and our assumption that the vector field $X$ is bounded by $C$, we get the sought control on $P^X_{p,i}\circ D((\widetilde{\mathcal{P}}^X_{p,i})^{-1})$ over all $p \in M_X$. 
	\end{itemize}
	In particular, we obtain a uniform control of $\Psi_{X_i(p),i}$ for $p\in M_X$ and $i\in \{0,\dots,n\}$ even though the space  $M_X$ is not compact. 
	
	By Definition \ref{def.perturbationpoincare}, the proof of \eqref{item deux} follows easily from the first point. Indeed, given $i\in \{1,\dots,n\}$ and $p \in M_X$, we use the maps $\Psi_{X_{i-1}(p), i-1}$ and $\Psi_{X_{i}(p), i}$ to conjugate $\mathcal{P}_{X_{i-1}(p), 1}^X$ to the linear Poincar\'e map $P_{X_{i-1}(p),1}^X$.  By the previous discussion, for $\rho>0$ small enough, the   maps $\Psi_{X_{i-1}(p),i-1}$ and $\Psi_{X_i(p),i}$ are arbitrarily $C^1$-close to  $\exp_{X_{i-1}(p)}^{-1}$ and $\exp_{X_i(p)}^{-1}$ respectively. The estimate on the $C^1$ distance between $g_i$  and $\mathcal{P}_{X_{i-1}(p),1}^X$ follows, since we assume $d_{C^1}(\tilde{g}_i, P_{X_{i-1}(p),1}^X) < \delta_1$, and $\delta_1<\delta$. 
	
	The proof of Point \eqref{item trois}  follows from arguments similar to those presented in Pugh-Robinson \cite{PughRobinson} (see in particular Lemma $6.5$ in that paper).
	
	More precisely, let $i\in \{1,\dots,n\}$, and let $\tilde{g}_i\colon N_{X_{i-1}(p)} \to N_{X_i(p)}$ be a $C^1$ diffeomorphism satisfying the assumptions of Point \eqref{item deux}. We pull back $\tilde{g}_i$ to a $C^1$ diffeomorphism $\hat{g}_i\colon N_{X_{i-1}(p)}\to N_{X_{i-1}(p)}$ by letting $\hat{g}_i:=P_{X_i(p),-1}^{X}\circ \tilde{g_i}$. By assumption, the support of $\hat{g}_i$ is contained in $P^{X}_{p,i-1} (\widetilde{U})$, with $\widetilde{U} := \exp^{-1}_p(U)$ and $U \subset \mathcal{N}_{X,p}(\rho\|X(p)\|)$, hence by \eqref{control Ppt}, we get
	\begin{equation}\label{control czero de hat gi}
	d_{C^0}(\hat{g}_i,\mathrm{id})\leq 2 C\rho \max_{p \in M}\|X(p)\|.
	\end{equation}
	Then for all $t \in [i-1,i]$, we define a map $\tilde{g}_t\colon N_{X_{i-1}(p)} \to N_{X_t(p)}$ as $\tilde{g}_t:=P^{X}_{X_{i-1}(p),t-i+1} \circ \hat{g}_i$.  By the above estimate, and by \eqref{control Ppt},  we deduce that
	\begin{equation}\label{controle g chapeau t c0}
	d_{C^0}(\tilde{g}_t,P^{X}_{X_{i-1}(p),t-i+1})\leq 2 C^2\rho \max_{p \in M}\|X(p)\|,\qquad \forall\,  t \in [i-1,i].
	\end{equation}
	Moreover, for any $t \in [i-1,i]$, we have $D\tilde{g}_t=P^{X}_{X_{i-1}(p) ,t-i+1}\cdot D\hat{g}_i=P^{X}_{X_{i-1}(p),t-i+1}\circ P_{X_i(p),-1}^{X}\cdot  D\tilde{g_i}$. Since $d_{C^1}(\tilde{g}_i, P^X_{X_{i-1}(p),1}) < \delta_1$, we obtain
	\begin{equation}\label{controle g chapeau t}
	d_{C^1}(\tilde{g}_t,P^{X}_{X_{i-1}(p),t-i+1})\leq C^2 \delta_1,\qquad \forall\, t \in [i-1,i].
	\end{equation}  
	
	Let us fix a $C^\infty$ bump function $\chi\colon \R \to [0,1]$ which is $0$ near $0$ and $1$ near $1$.  Fix $i\in \{1,\dots,n\}$ and set $\chi_{i-1}(\cdot):=\chi(\cdot-i+1)$. For $k \in \{0,\dots,n\}$, we also let $\mathcal{N}_{p,k}:=\mathcal{N}_{X,X_{k}(p)}\big(\frac{\beta}{C^{n-k}}\|X(p)\|\big)$. 
	Then for any $t \in [i-1,i]$, we let $h_t^{(i)}\colon \mathcal{N}_{p,i-1}\to \mathcal{N}_{p,i}$  be the map defined as
	\begin{itemize}
		\item $h_t^{(i)}(y) := \mathcal{P}_{X_{i-1}(p), t-i+1}^X(y)$, if $y\notin \mathcal{P}^{X}_{p,i-1}(U)$;
		\item $h_t^{(i)}(y) := \Psi^{-1}_{X_t(p),t} \circ \left(\chi_{i-1}(t)\tilde{g}_t+(1-\chi_{i-1}(t)) P^{X}_{X_{i-1}(p),t-i+1}\right) \circ \Psi_{X_{i-1}(p), i-1}(y)$, if $y\in \mathcal{P}^{X}_{p,i-1}(U)$,
	\end{itemize}
	where 
	we have extended the previous notation by setting 
	$$
	\Psi_{X_t(p),t}:=P^X_{p,t} \circ \widetilde{\mathcal{P}}^{X}_{X_t(p),-t} \circ \exp_{X_t(p)}^{-1}.
	$$  
	\begin{figure}[H]
		\begin{center}
			\includegraphics [width=13cm]{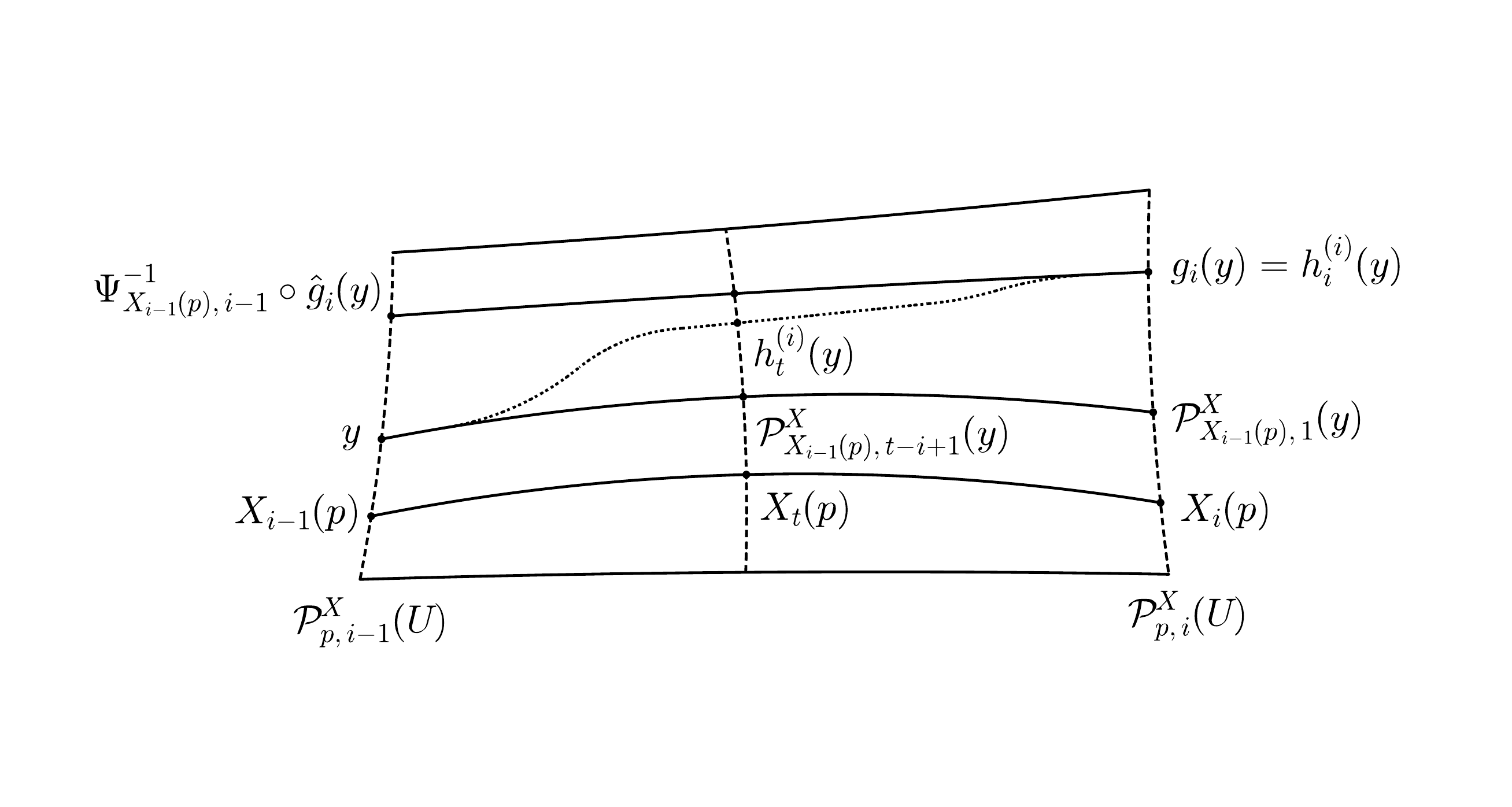}
		\end{center}
		\caption{Interpolation between the initial Poincar\'e map and $g_i$.}
	\end{figure}
	
	In particular, we note that for $t=i-1$,  we have $h_t^{(i)}=h_{i-1}^{(i)}=\mathrm{id}$, while for $t=i$, $h_t^{(i)}=h_i^{(i)}$ coincides with the map $g_i$ defined in Item \eqref{item deux}. 
	
	By \eqref{controle g chapeau t c0}, for all $t \in [i-1,i]$, we have
	$$
	d_{C^0}\left(\chi_{i-1}(t)\tilde{g}_t+(1-\chi_{i-1}(t)) P^{X}_{X_{i-1}(p),t-i+1} , P^{X}_{X_{i-1}(p),t-i+1}\right)\leq 2 C^2\rho \max_{p \in M}\|X(p)\|.
	$$
	Since $\mathcal{P}_{X_{i-1}(p), t-i+1}^X=\Psi^{-1}_{X_t(p),t} \circ P^{X}_{X_{i-1}(p),t-i+1}  \circ \Psi_{X_{i-1}(p), i-1}$, by the definition of $h_t^{(i)}$ and by \eqref{controle Psi}, we can thus make the $C^0$ distance between $h_t^{(i)}$ and $\mathcal{P}_{X_{i-1}(p), t-i+1}^X$ arbitrarily small, provided that $\rho>0$ is taken small enough. 
	
	For any $t\in [i-1,i]$, we have 
	\begin{gather*}
	D\left(\chi_{i-1}(t)\tilde{g}_t+(1-\chi_{i-1}(t)) P^{X}_{X_{i-1}(p),t-i+1} \right)\\
	=DP^{X}_{X_{i-1}(p),t-i+1}
	+\chi_{i-1}(t)\big(D\tilde{g}_t-P^{X}_{X_{i-1}(p),t-i+1}\big).
	\end{gather*}
	By \eqref{control Ppt} and \eqref{controle g chapeau t},  we thus get
	\begin{equation}\label{control diff ht}
	d_{C^1}(h_t^{(i)},\mathcal{P}_{X_{i-1}(p), t-i+1}^X)\leq 4C^2 \delta_1.
	\end{equation}
	
	For any $t \in [i-1,i]$, we also have:
	\begin{align*}
	&\partial_t\left(\chi_{i-1}(t)\tilde{g}_t+(1-\chi_{i-1}(t)) P^{X}_{X_{i-1}(p),t-i+1} \right)-\partial_t P^{X}_{X_{i-1}(p),t-i+1}\\
	=\ & \chi_{i-1}'(t)  \big(\tilde{g}_t- P^{X}_{X_{i-1}(p),t-i+1}\big)+\chi_{i-1}(t)  \partial_t\big(\tilde{g}_t- P^{X}_{X_{i-1}(p),t-i+1}\big)\\
	=\ & \chi_{i-1}'(t)P^{X}_{X_{i-1}(p),t-i+1} \circ  \big(\hat{g}_i- \mathrm{id}\big)+\chi_{i-1}(t) \partial_t P^{X}_{X_{i-1}(p),t-i+1}\circ \big(\hat{g}_i- \mathrm{id}\big).
	\end{align*}
	By 
	\eqref{controle Psi}, \eqref{control Ppt} and \eqref{control czero de hat gi}, we deduce that
	\begin{align}
	&\max_{t \in [i-1,i]}\max_{y \in U}|\partial_t \mathcal{P}^X_{X_{i-1}(p),t-i+1}(y)- \partial_t h_t^{(i)}(y)| \nonumber\\ 
	&\leq 8C \max\left(C,\sup_{t\in [0,1]} \|\partial_t P^X_{X_{i-1}(p),t}\|\right)\|\chi\|_{C^1}\rho \max_{p \in M}\|X(p)\|. \label{control partial t}
	\end{align}
	
	Recall that for $k \in \{0,\dots,n\}$, we denote $\mathcal{N}_{p,k}:=\mathcal{N}_{X,X_{k}(p)}\big(\frac{\beta}{C^{n-k}}\|X(p)\|\big)$. As in \eqref{defi set Idef}, given a set $V \subset \mathcal{N}_{p,0}$, we set
	\begin{equation*}
	\mathcal{I}^X(p,V,n):=\left\{(y,t): y\in V,\   t\in [0, \tau_{p,n}^X(y)] \right\}.
	\end{equation*} 
	Let us assume  that $U \subset \mathcal{N}_{X,p}(\rho\|X(p)\|)$ is  such that the map $(y,t) \mapsto X_t(y)$ is injective on the set 
	$\mathcal{I}^X(p,U,n)$.
	For $\rho>0$ small, the hitting time function $\tau_{p,n}^X$ is uniformly close to $n$ on $\mathcal{N}_{X,p}(\rho\|X(p)\|)$, and the $C^1$ distance between the maps $(y,t)\mapsto \mathcal{P}_{p,t}^X(y) $ and $(y,t) \mapsto X_t(y) $ restricted to $\mathcal{I}^X(p,\mathcal{N}_{X,p}(\rho\|X(p)\|),n) $ is  small.  Given $i\in \{1,\dots,n\}$, 
	let us consider the map $h^{(i)} \colon (y,t)\mapsto h_t^{(i)}(y)$ defined on $\mathcal{N}_{p,i-1}\times [i-1,i]$ as above. By \eqref{control diff ht} and \eqref{control partial t}, and since $0< \delta_1 < \delta$,  the maps $ \mathcal{N}_{p,i-1}\times [i-1,i]\ni(y,t)\mapsto \mathcal{P}_{X_{i-1}(p),t-i+1}^X(y)$ and $h^{(i)}$ can be made arbitrarily $C^1$-close by taking $\delta>0$  small enough. 
	For $\delta>0$ sufficiently small, we deduce that   the map $h^{(i)}$ is locally injective on the interior of $\mathcal{P}^{X}_{p,i-1}(U)\times [i-1,i]$. 
	Besides, as we have seen, $h_{i-1}^{(i)}|_{\mathcal{N}_{p,i-1}}=\mathrm{id}|_{ \mathcal{N}_{p,i-1}}$, while $h_i^{(i)}|_{\mathcal{N}_{p,i-1}}=g_i|_{\mathcal{N}_{p,i-1}}$ is a $C^1$ diffeomorphism. 
	
	Now, we define a map $H$ on $\mathcal{N}_{p,0} \times [0,n]$ by setting \begin{gather}\label{defin Hyt}
	H(y,t):=h^{(i)}_t \circ g_{i-1}\circ g_{i-2}\circ \cdots \circ g_{1}(y),\\ 
	\forall\, y\in \mathcal{N}_{p,0},\ \forall\, t \in [i-1,i], \  \forall\, i\in \{1,\dots,n\}. \nonumber
	\end{gather}
	By what precedes, the map $H$ is locally injective on the interior of the set $U \times [0,n]$. Moreover, $\partial \left(U \times [0,n]\right)=(U \times \{0\}) \cup ( U \times \{n\}) \cup (\partial U \times [0,n])$. On the one hand, we have $H(\cdot,0)|_{U}=\mathrm{id}|_{U}$, and by construction, the map $H(\cdot,n)|_{U}$  coincides with $g_{n}\circ g_{n-1}\circ \cdots \circ g_{1}|_{U}$, hence it is a $C^1$ diffeomorphism from $U$ to $\mathcal{P}_{p,n}^{X}(U)\subset \mathcal{N}_{p,n}$. On the other hand, by Point \eqref{item deux}, each diffeomorphism $g_i$ is a $\delta_2$-perturbation of $\mathcal{P}_{X_{i-1}(p), 1}^X$ with support in $\mathcal{P}^{X}_{p,i-1}(U)$. Therefore the restriction of $H$ to the set $\partial U \times [0,n]$ 
	coincides with the restriction of the map $(y,t)\mapsto \mathcal{P}^X_{p,t}(y)$. In particular, we deduce that the restriction $H|_{\partial \left(U \times [0,n]\right)}$ of $H$ to the boundary of $U \times [0,n]$ is injective.   From Lemma 6.5 in Pugh-Robinson \cite{PughRobinson}, we conclude that $H$ embeds $U \times [0,n]$  into the set $\mathcal{U}^X(p,U,n)$ introduced in \eqref{defi set Urond}. 
	\\
	
	In the same way as before, for any $y \in \mathcal{N}_{X,p}(\rho\|X(p)\|)$ and $t \in [0,n]$, we set
	$$
	\tau_{p,t}^X(y):=\min\{s \geq 0: X_s(y) \in \mathcal{N}_{X,X_t(p)}(R)\}.
	$$
	By definition,  $\mathcal{P}_{p,t}^X(y)=X_{\tau_{p,t}^X(y)}(y)$, for any $(y,t) \in \mathcal{N}_{X,p}(\rho\|X(p)\|)\times [0,n]$, thus
	\begin{equation}\label{remarque X P}
	X(\mathcal{P}_{p,t}^X(y))=(\partial_t \tau_{p,t}^X(y))^{-1}\partial_t \mathcal{P}_{p,t}^X(y).
	\end{equation}
	Moreover, $\tau_{p,\cdot}^X(p)=\mathrm{id}$, and the map $(y,t) \mapsto \tau_{p,t}^X(y)$ is $C^1$ on $\mathcal{N}_{X,p}(\rho\|X(p)\|)\times [0,n]$, hence for $\rho>0$ sufficiently small, we have 
	\begin{equation}\label{control Ppt der}
	\frac 12<|\partial_t \tau_{p,t}^X(y)| < 2, \qquad \forall\, p \in M_X,\ y \in \mathcal{N}_{X,p}(\rho\|X(p)\|),\ t \in [0,n].
	\end{equation}
	
	As we have noted above, on the complement of $U \times [0,n]$, the maps $H$   and $(y,t)\mapsto \mathcal{P}_{p,t}^X(y)$ coincide. 
	We thus define a vector field $Y\in \mathfrak{X}^1(M)$ on $M$ by setting
	\begin{itemize}
		
		\item $Y(q):=X(q)$, if $q \in M- \mathcal{U}^X(p,U,n)$;
		\item $Y(q):=(\partial_t|_{t=t_0} \tau_{p,t}^X(y))^{-1}\partial_t|_{t=t_0} H(y_0,t)$, if $q \in \mathcal{U}^X(p,U,n)$, where $(y_0,t_0):=H^{-1}(q)\in U \times [0,n]$.
	\end{itemize}
	
	For each $i\in \{1, \dots, n\}$, by the definition of $H$ in \eqref{defin Hyt} and since $h_i^{(i)}=g_i$, it follows that the Poincar\'e map $\mathcal{P}_{X,X_{i-1}(p),1}^Y$ for the vector field $Y$ between $\mathcal{N}_{p,i-1}$ and $\mathcal{N}_{p,i}$ is given by $g_i$. By definition, the support of $X-Y$ is contained in $\mathcal{U}^X(p,U,n)$. Moreover, given any point $q=\mathcal{P}^X_{p,t}(y)=H(y',t) \in \mathcal{U}^X(p,U,n)$, say $(y,t)\in U \times [i-1,i]$, letting $z:=\mathcal{P}_{p,i-1}^{X}(y)$ and $z':=g_{i-1}\circ g_{i-2}\circ \cdots \circ g_1(y')$, we obtain
	\begin{align*}
	\mathcal{P}^X_{p,t}(y)&=\mathcal{P}^{X}_{X_{i-1}(p),t-i+1}(z)\\
	&=\Psi^{-1}_{X_t(p),t} \circ P^{X}_{X_{i-1}(p),t-i+1}\circ \Psi_{X_{i-1}(p), i-1}(z);\\
	H(y',t)&=h_{t}^{(i)}(z')\\
	&= \Psi^{-1}_{X_t(p),t} \circ \left(\chi_{i-1}(t)\tilde{g}_t+(1-\chi_{i-1}(t)) P^{X}_{X_{i-1}(p),t-i+1}\right) \circ \Psi_{X_{i-1}(p), i-1}(z')\\
	&= \Psi^{-1}_{X_t(p),t} \circ P^{X}_{X_{i-1}(p),t-i+1}\circ \left(\chi_{i-1}(t)(\hat{g}_i-\mathrm{id})+\mathrm{id}\right) \circ \Psi_{X_{i-1}(p), i-1}(z').
	\end{align*}
	Set 
	$$
	w:=\Psi_{X_{i-1}(p), i-1}(z)=\left(\chi_{i-1}(t)(\hat{g}_i-\mathrm{id})+\mathrm{id}\right) \circ \Psi_{X_{i-1}(p), i-1}(z').
	$$
	We deduce that 
	\begin{align*}
	\partial_t \mathcal{P}^X_{p,t}(y)&=\partial_t\left(\Psi^{-1}_{X_t(p),t} \circ P^{X}_{X_{i-1}(p),t-i+1}\right)(w),\\
	\partial_t H(y',t) &=\partial_t\left(\Psi^{-1}_{X_t(p),t} \circ P^{X}_{X_{i-1}(p),t-i+1}\right)(w)
	+ D_w\left(\Psi^{-1}_{X_t(p),t} \circ P^{X}_{X_{i-1}(p),t-i+1}\right)\\
	&\quad \cdot \partial_t  \left(\left(\chi_{i-1}(t)(\hat{g}_i-\mathrm{id})+\mathrm{id}\right) \circ \Psi_{X_{i-1}(p), i-1}(z')\right)\\
	&= \partial_t \mathcal{P}^X_{p,t}(y)+\chi_{i-1}'(t)   D_{\Psi_{X_t(p),t}(q)}\Psi^{-1}_{X_t(p),t} \circ P^{X}_{X_{i-1}(p),t-i+1}\circ\\
	&\quad \circ \left(\hat{g}_i-\mathrm{id}\right) \left( \Psi_{X_{i-1}(p), i-1}(z')\right),
	\end{align*}
	and 
	\begin{align*}
	&Y(q)-X(q)=\\  
	&\frac{\chi_{i-1}'(t)}{\partial_t \tau_{p,t}^X(y)}  D_{\Psi_{X_t(p),t}(q)}\Psi^{-1}_{X_t(p),t} \circ P^{X}_{X_{i-1}(p),t-i+1}
	\circ \left(\hat{g}_i-\mathrm{id}\right) \left( \Psi_{X_{i-1}(p), i-1}(z')\right),
	\end{align*}
	where the last equality follows from \eqref{remarque X P} and the definition of $Y$. 
	In particular, the difference between the vector fields $X$ and $Y$ is essentially controlled by the $C^0$ distance between $\hat{g}_i$ and $\mathrm{id}$. 
	More precisely, by \eqref{controle Psi},  \eqref{control Ppt},  \eqref{control czero de hat gi}, and \eqref{control Ppt der}, we deduce that 
	$$
	|X(q)-Y(q)|\leq 8 \|\chi\|_{C^1} C^2 \rho \max_{p \in M}\|X(p)\|,
	$$
	and we argue similarly for the derivatives. Therefore, by taking $\rho$ sufficiently small, we can ensure that $d_{C^1}(X,Y) < \varepsilon$, which concludes the proof of point \eqref{item trois}, and then, of Lemma \ref{lemma.fundamentalperturbation}.  
\end{proof}

\subsubsection{Producing unbounded normal distortion by perturbation} We are now in position to prove the main perturbation result (Proposition~\ref{prop sept} below) that will allow us to obtain unbounded normal distortion generically. The key tool behind this is a perturbation result for linear cocycles taken from \cite{BonattiCrovisierWilkinson}.   The next proposition roughly says that for any $C^1$ vector field $X$, any compact subset $\Delta \subset M_X$ of a section transverse to the flow, any $x \in M_X$ whose orbit is far from $\Delta$, we can produce another vector field $Y$ that is $C^1$-close to $X$ and such that for any point $y \in \Delta$,  we see a large distortion between the images of $x$ and $y$ under the linear Poincar\'e flow of $Y$. Moreover, $Y$ can be made arbitrarily $C^0$-close to $X$ if the support of the perturbation is chosen large enough. 

\begin{proposition}
	\label{prop.proposition8}
	For any $d \geq 2$, $C>1$, $K,\varepsilon>0$, let  $\delta = \delta(C,\varepsilon)$  be the constant given by Lemma \ref{lemma.fundamentalperturbation}. There exists  $n_0=n_0(d,C,K,\varepsilon) \in \N$ with the following property.
	
	For any $d$-dimensional manifold $M$, any vector field $X \in \mathfrak{X}^1(M)$ which is bounded by $C$, there exists $\rho_0 = \rho_0(d,C,K,\varepsilon, n_0)>0$ such that for any $\eta>0$, any compact set $\Delta \subset M_X$ and $x,p\in M_X$  satisfying:
	\begin{enumerate}[label=(\alph*)]
		\setlength{\itemsep}{0.2em}
		\item there exists an open set $U$ inside $\mathcal{N}_{X,p}(\rho_0\|X(p)\|)$, such that $\Delta \subset U$;
		\item the map $(y,t) \mapsto X_t(y)$ is injective on $\mathcal{I}^X(p,U,n_0)$ (see \eqref{defi set Idef});
		\item $\mathrm{\orb}(x) \cap U =\emptyset$,
	\end{enumerate}
	there exists  a vector field $Y\in \mathfrak{X}^1(M)$ such that 
	\begin{enumerate}
		\setlength{\itemsep}{0.2em}
		\item the support of $X-Y$ is contained in $\mathcal{U}^X(p, U, n_0)$ (see \eqref{defi set Urond});
		\item $d_{C^1}(X,Y) < \varepsilon$;
		\item for any $i\in \{0, \dots, n_0 -1\}$, it is verified $d_{C^1}(\mathcal{P}^X_{X_i(p),1}, \mathcal{P}^Y_{X,X_i(p),1}) < \delta $, where $\mathcal{P}^Y_{X,X_i(p),1}$ is the Poincar\'e map between $\mathcal{N}_{X,X_i(p)}(\beta \|X(X_i(p))\|)$ and $\mathcal{N}_{X,X_{i+1}(p)}(R)$;
		\item   $d_{C^0}(\mathcal{P}^X_{X_i(p),t}, \mathcal{P}^Y_{X,X_i(p),t}) < \eta$, for all $t \in [0,1]$;
		\item for all $y \in \Delta$, there exists an integer $n \in \{1,\dots,n_0\}$ such that
		$$
		|\log \det P^Y_{x,n}-\log \det P^Y_{y,n}|>K.
		$$
	\end{enumerate}
\end{proposition}

	\begin{figure}[H]
	\begin{center}
		\includegraphics [width=10cm]{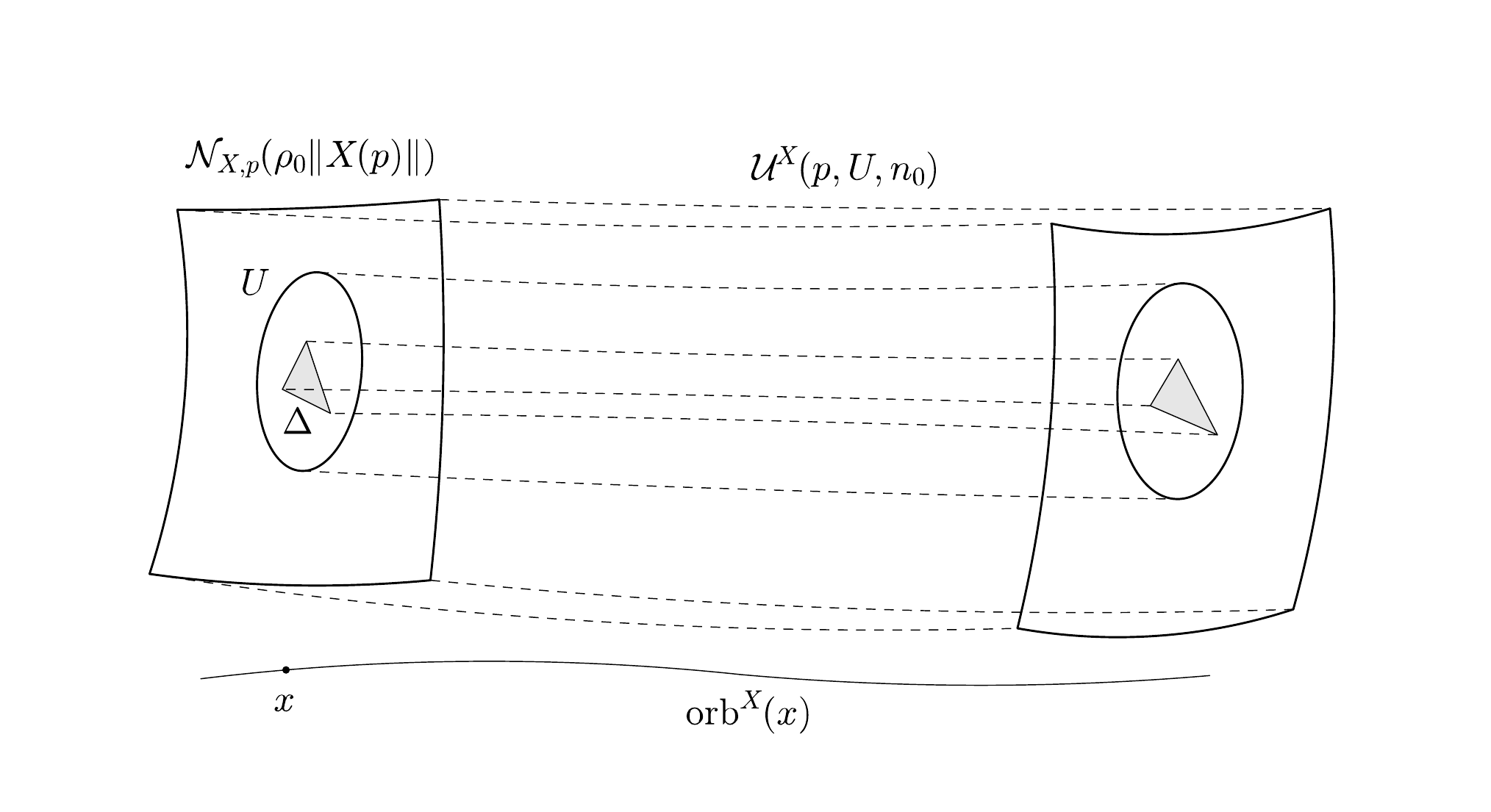}
	\end{center}
\end{figure}

Proposition \ref{prop.proposition8} is the analogue for flows of Proposition 8 in \cite{BonattiCrovisierWilkinson}. Using Lemma \ref{lemma.fundamentalperturbation}, we will reduce the proof of this proposition to a discrete scenario where we can apply the following proposition from \cite{BonattiCrovisierWilkinson}. This result tells us that given a linear cocycle 
$\tilde f$ acting on $\Z \times \R^d$ for some integer $d \geq 1$, two sets $\Delta \subset U \subset \R^d$, with $\Delta$ compact, we can perturb $\tilde f$ into a new cocycle that is $C^1$-close to it, coincides with $\tilde f$ outside of the tube obtained by flowing $U$ for a finite time, and such that for any point $y\in\Delta$, we can find an iterate of $y$ for which the difference between the Jacobians of the cocycles is very large. Moreover, we can keep the $C^0$-distance between the two cocycles arbitrarily small provided that the tube is long enough. 

\begin{proposition}[Proposition 9 in \cite{BonattiCrovisierWilkinson}]
	\label{prop.proposition9}
	
	For any $d\geq 1$, and any $C,K,\varepsilon >0$, there exists $n_1 = n_1(d,C,K, \varepsilon)\geq 1$ with the following property.
	
	Consider any sequence $(A_i) \in \mathrm{GL}(d,\mathbb{R})$ with $\|A_i\|, \|A_i^{-1}\| <C$ and the associated cocycle $\tilde{f}$ on $\mathbb{Z} \times \mathbb{R}^d$ defined by $\tilde{f}(i,v):=(i+1, A_i  v)$. Then, for any open set $U\subset \mathbb{R}^d$, for any compact set $\Delta \subset U$ and any $\eta>0$, there exists a diffeomorphism $\tilde{g}$ of $\mathbb{Z} \times \mathbb{R}^d$ such that:
	\begin{itemize}
		\item $d_{C^1}(\tilde{f},\tilde{g}) < \varepsilon$;
		\item $d_{C^0}(\tilde{f},\tilde{g}) < \eta$;
		\item $\tilde{f}=\tilde{g}$ on the complement of $\bigcup_{i=0}^{2n_1-1} \tilde{f}^i(\{0\} \times U)$;
		\item for all $y\in \{0\} \times \Delta$, there exists $n\in \{1, \dots, n_1\}$ such that 
		\[
		|\log \det D\tilde{f}^n(y) - \log \det D\tilde{g}^n(y)|>K.
		\]
	\end{itemize}
\end{proposition}

Below we prove Proposition \ref{prop.proposition8} assuming Proposition \ref{prop.proposition9}. The idea of the proof is that we first perturb the discretized linearized dynamics using Proposition \ref{prop.proposition9}, then  we relate this perturbation with the dynamics of a vector field by using the realization lemma (Lemma \ref{lemma.fundamentalperturbation}).
\begin{proof}[Proof of Proposition \ref{prop.proposition8} from Proposition \ref{prop.proposition9}]
	
	Fix any $\delta_1 \in (0, \delta)$ and $K_0 > 2K + 10 \log 2$. Let $n_1 = n_1(d-1, C,K_0, \delta_1)$ be the constant given by Proposition \ref{prop.proposition9} for $d-1$, $C,K_0,\varepsilon$ and let $n_0 = 2n_1$. Let $X\in \mathfrak{X}^1(M)$ be a vector field bounded by $C$ and let $\rho>0$ be the constant given by Lemma \ref{lemma.fundamentalperturbation} for $C,$ $\varepsilon$, $\delta_1$, $n_0$ and $X$. Fix $\rho_0 \in (0, \frac{\rho}{C^{n_0}})$.
	
	Let $\Delta\subset M_X$, $x,p\in M_X$ and $\eta>0$ be such that conditions $(a)$, $(b)$ and $(c)$ in Proposition \ref{prop.proposition8} are verified. Let $U\subset \mathcal{N}_{X,p}(\rho_0\|X(x)\|)$ be the open set given by condition $(a)$. Consider  $O_{X_1}(p) =\{\dots, X_{-1}(p),  p, X_1(p), \dots\}$ and observe that this set is naturally identified with $\mathbb{Z}$. We consider the normal bundle, with respect to $X$, over $O_{X_1}(p)$ and the linear cocycle defined as follows: for $i\in \mathbb{Z}$ and $v\in N_{X,X_i(p)}$ set $\tilde{f}(i,v) := (i+1, P^X_{X_i(p),1} v)$.
	
	Recall that $\widetilde{U} = \exp^{-1} _p(U)$. By Item \eqref{item un} in Lemma \ref{lemma.fundamentalperturbation}, for any $i\in \{0, \dots, n_0\}$, we obtain $C^1$ diffeomorphisms $\Psi_i : = \Psi_{X_i(p),i}\colon \mathcal{P}^X_{p,i}(U) \to P^X_{p,i}(\widetilde{U})$, such that for any $q\in \mathcal{P}^X_{p,i}(U)$ it holds that 
	\begin{equation}
	\label{eq.conjugation}
	P^X_{X_i(p),1}(\Psi_i(q)) = \Psi_{i+1}(\mathcal{P}^X_{X_{i}(p),1}(q)).
	\end{equation}
	Write $\Psi\colon \bigcup_{i=0}^{n_0} \mathcal{P}^X_{p,i}(U) \to \bigcup_{i=0}^{n_0} P^X_{p,i}(\widetilde{U})$ as the $C^1$ diffeomorphism which is equal to $\Psi_i$ on $\mathcal{P}^X_{p,i}$.
	
	For the cocycle $\tilde{f}$, we apply Proposition \ref{prop.proposition9} and obtain a $\delta_1$-perturbation $\tilde{g}$ of $\tilde{f}$ supported on $\bigcup_{i=0}^{n_0-1} \tilde{f}^i(\{0\}\times\widetilde{U})$, such that for every $q\in \Psi_0(\Delta)$, it holds:
	\begin{itemize}
		\item $d_{C^0}(\tilde{f}, \tilde{g}) < \frac{\eta}{2};$
		\item $\tilde{f} = \tilde{g}$ on the complement of $\bigcup_{i=0}^{n_0-1} \tilde{f}^i(\{0\}\times\widetilde{U})$;
		\item for every $q\in \Psi_0(\Delta)$, there exists $n\in \{1, \dots, n_0\}$ such that
		\[
		|\log \det  D\tilde{f}^n(q) - \log \det  D\tilde{g}^n(q)|> K_0.
		\]
	\end{itemize}
	
	For each $i\in \{0, \dots, n_0-1\}$, let $\tilde{g}_i:= g|_{\{i\} \times N_{X,X_i(p)}}$ and observe that $d_{C^1}(\tilde{g}_i, P^X_{X_i(p),1})< \delta_1$. By Item \eqref{item deux} in Lemma \ref{lemma.fundamentalperturbation}, we obtain a $\delta$-pertubation $g_i$ of $\mathcal{P}^X_{X_{i}(p),1}$. By \eqref{controle Psi} and \eqref{eq.conjugation}, we have 
	\[
	d_{C^0}(g_i, \mathcal{P}^X_{X_i(p),1}) < 2 d_{C^0}(\tilde{g_i}, P^X_{X_i(p),1}) < \eta.
	\]
	Moreover, by the estimates in \eqref{controle Psi}, we conclude that for any $q\in \Delta$, there exists $n\in \{1, \dots, n_0-1\}$ such that 
	\begin{equation}
	\label{estimatedet1}
	|\log \det D\mathcal{P}^X_{p,n}(q) - \log \det  D (g^n)(q)| > K_0 - 4 \log 2,
	\end{equation}
	where $g^n(q) := g_n \circ \cdots \circ g_1(q)$.
	
	Recall that for $n\in \{ 0, \dots, n_0-1\}$, the maps $\mathcal{P}^X_{p,n}$ and $P^X_{p,n}$ are conjugated on $\Delta$ by $\Psi$. By \eqref{controle Psi}, we obtain that for any $q\in \Delta$, it holds 
	\begin{equation}
	\label{eq.estimateprop8.1}
	|\log \det  D\mathcal{P}^X_{p,n}(q) - \log \det  P^X_{p,n}| \leq 2\log 2.
	\end{equation}
	Suppose there exists $n\in \{0, \dots, n_0 -1\}$ such that $|\log \det  P^n_p - \log \det  P^n_x| > K + 3\log 2$. By (\ref{eq.estimateprop8.1}) and Remark \ref{rmk.closetonormal}, for any $q\in \Delta$ it holds that
	\[
	\left|\log \det  P^X_{q,\tau^X_{p,n}(q)} - \log \det  P^X_{x,n}\right| > K+ \log 2.
	\]
	By Lemma \ref{lem.justify} and item $3$ of Lemma \ref{lemma.fundamentalperturbation}, we conclude that
	\[
	|\log \det P^X_{q,n} - \log \det P^X_{x,n}| > K.
	\]
	In this case we do not make any perturbation. Suppose that for every $n\in \{0, \dots, n_0-1\}$ and every $q\in \Delta$ we have
	\[
	|\log \det  P^X_{q,n} - \log \det  P^X_{x,n} | \leq K + 3\log2. 
	\]
	Consider the maps $g_1, \dots ,g_{n_0}$ as it was explained above (obtained using Proposition \ref{prop.proposition9}). Applying Lemma \ref{lemma.fundamentalperturbation}, we obtain a $C^1$ vector field $Y$ that verifies the following properties:
	\begin{itemize}
		\item $d_{C^1}(X,Y)< \varepsilon$;
		\item the support of $X-Y$ is contained in $\mathcal{U}^X(p, U, n_0)$;
		\item for each $i\in \{1, \dots, n_0\}$, we have that $\mathcal{P}^Y_{X,X_i(p),1} = g_i$. 
	\end{itemize}
	By (\ref{estimatedet1}) and (\ref{eq.estimateprop8.1}), we conclude that for each $q\in \Delta$, there exists $n\in \{1, \dots, n_0\}$ such that
	\[\arraycolsep=1.4pt\def\arraystretch{2.5}
	\begin{array}{rcl}
	\displaystyle \left| \log \left( \frac{ \det  P^Y_{x,n}}{\det  P^Y_{q,n} }\right) \right| & \geq & \displaystyle \left| \log \left( \frac{\det  D\mathcal{P}^X_{p,n}(q)}{\det  P^Y_{q,n} }\right)\right| - \left| \log \left( \frac{\det  P^X_{x,n}}{ \det  D\mathcal{P}^X_{p,n}(q) }\right) \right|\\ 
	& \geq & \displaystyle  \left| \log \left( \frac{\det  D \mathcal{P}^X_{p,n}(q)}{ \det  Dg^n(q) }\right)\right| - \left| \log \left( \frac{\det  Dg^n(q)}{ \det  P^Y_{q,n} }\right)\right| -\\
	&-& \displaystyle \left| \log \left( \frac{ \det  P^X_{x,n}}{ \det P^X_{q,n} } \right)\right| - \left| \log \left( \frac{\det  P^X_{q,n}}{ \det  D\mathcal{P}^X_{p,n}(q) }\right)\right|\\
	&> & \displaystyle  K_0 - \log 2 - K - 4\log 2 - \log 2 > K.
	\end{array}
	\]
	This concludes the proof of Proposition \ref{prop.proposition8}.    
\end{proof}

The following proposition is the version for flows of Proposition $7$ in \cite{BonattiCrovisierWilkinson}. 

\begin{proposition}\label{prop sept}
	Consider a vector field  $X \in \mathfrak{X}^1(M)$, a compact set $D\subset M_X$, an open set $O \subset M_X$ and a point $x \in M_X$  satisfying:
	\begin{itemize}
		\item for any $y \in O$, any $t \geq 0$, we have $X_t(y) \in O$ and $X_1(\overline{O}) \subset O$;
		\item $D \subset O - X_1(\overline{O})$;
		\item $\mathrm{\orb}(x) \cap D=\emptyset$. 
	\end{itemize}
	Then for any $K,\varepsilon>0$, there exists a vector field $Y \in \mathfrak{X}^1(M)$ with $d_{C^1}(X,Y)<\varepsilon$ which satisfies the following property: for all $y \in D$, there exists $n \geq 1$ such that 
	$$
	|\log \det  P^Y_{x,n}-\log \det P^Y_{y,n}|>K.
	$$
	Moreover, the support of $X-Y$ is contained in the complement of the chain recurrent set of $X$.  
\end{proposition}

\begin{proof}
	Let $X,D,O,x$ be as in the statement of Proposition \ref{prop sept}. Fix $K, \varepsilon>0$. By Lemma \ref{lem.boundedbyc}, we may fix $C>1$ such that the vector field $X$ is bounded by $C$, and let $n_0=n_0(d,C,3K,\varepsilon)$, $\rho_0=\rho_0(d,C,K,\varepsilon)$ be chosen according to Proposition \ref{prop.proposition8}. We set $N:=2^{d} n_0$. Without loss of generality, we also assume that $K$ satisfies $K>2d \log(2C)>0$. 
	
	We fix a finite cover $\mathcal{F}=\{D_1,\dots,D_\ell\}$ of $D$ by compact sets satisfying:
	\begin{enumerate}
		\item\label{point un} $D \subset \bigcup_{j=1}^\ell \mathrm{int}(D_j) \subset O - X_{1}(\overline{O})$; 
		\item\label{point deux} for each $j \in \{1,\dots,\ell\}$, there exists a real number $\tau_j \in (0,1)$, a point $p_j \in O - X_{1}(\overline{O})$, an open set $U_j\subset  \mathcal{N}_{X,p_j}(\rho_0|X(p_j)|)$, and a compact set $\Delta_j\subset U_j$,  such that  the following properties hold:
		\begin{enumerate}
			\item 
			we have 
			\begin{equation}\label{small cubes}
			D_j= \{X_t(y): y \in \Delta_j,\ t \in [0,\tau_j]\},
			\end{equation}
			and 
			$$\mathrm{int}(D_j)\subset \{X_t(y) : y \in U_j,\ t \in (0,\tau_j)\} \subset O - X_1(\overline{O});$$
			\item for each $t \in [0,N]$, we have $\mathcal{P}_{p_j,t}^X(U_j) \subset  \mathcal{N}_{X,X_t(p_j)}(\rho_0|X_t(p_j)|)$; 
			\item for each $t \in [0,N-1]$, for each $t' \in [0,1]$, and for each $y_1,y_2\in \mathcal{P}_{p_j,t}^X(\Delta_j)$, it holds
			\begin{equation}\label{controle distance p x}
			d(\mathcal{P}_{X_t(p_j),t'}^X(y_1),\mathcal{P}_{X_t(p_j),t'}^X(y_2)) \leq 2C d(y_1,y_2);
			\end{equation}
		\end{enumerate}
		\item\label{point trois} $\mathrm{\orb}(x) \cap \bigcup_{j=1}^\ell  U_j =\emptyset$;
		\item\label{point quatre} for each $j\in \{1,\dots,\ell\}$, the map $(y,t) \mapsto X_t(y)$ is injective restricted to the set $U_j \times [0,1]$, and thus, it is also injective on the whole set $\mathcal{I}^X(p_j,U_j,N)$;\footnote{Indeed, for $t >1$, we have $X_t(U_j)=X_1(X_{t-1}(U_j))$, and  $X_{t-1}(U_j) \subset O-X_1(\overline{O})$.}
		\item\label{point cinq} there exists a partition $\{1,\dots,\ell\}=J_0\sqcup \dots \sqcup J_{2^{d} -1}$ such that for each $k\in \{0,\dots,2^d -1\}$, and for each $j_1\neq j_2 \in J_{k}$, we have
		$$
		\mathcal{U}^X(p_{j_1},U_{j_1},1)\ \cap\ \mathcal{U}^X(p_{j_2},U_{j_2},1)=\emptyset.
		$$
	\end{enumerate}
	
	One can obtain $\mathcal{F}$ by tiling  the compact set $D$ by arbitrarily small cubes as in \eqref{small cubes}, i.e., obtained by flowing small transversals $\Delta_j$ under $X$, for  $j=1,\dots,\ell$. Besides, since we assume that $D \subset M_X$, Properties \eqref{point un}-\eqref{point quatre} are satisfied provided that $D_j$, $U_j$ and $\Delta_j$ are chosen sufficiently small, for all $j\in \{1,\dots,\ell\}$. In particular, \eqref{controle distance p x} is true provided that $D_j$ and $\Delta_j$ are chosen  small enough, for all $j\in \{1,\dots,\ell\}$, since $X$ is bounded by $C$. Moreover, Item \eqref{point cinq}  holds true provided that the diameter of the sets $U_1,\dots,U_\ell$ is small enough,  since $M$ has dimension $d$. 
	
	
	For each $j \in \{1,\dots,\ell\}$, and for each $i,m\geq 0$, we set 
	$$
	\mathcal{V}_j^X(i,m):=\mathrm{int}(\mathcal{U}^X(X_i(p_j),\mathcal{P}_{ p_j,i}^{X}(U_j),m)).
	$$
	Each set $\mathcal{V}_j^X(i,m)$ is open: it is the interior of the ``tube'' obtained by flowing points in the transversal $\mathcal{P}_{p_j,i}^{X}(U_j)$ under $X$ 
	until they hit the section $\mathcal{P}_{p_j,i+m}^{X}(U_j)$. 
	We have the following properties:
	\begin{itemize}
		\item for  each $j \in \{1,\dots,\ell\}$, the sets 
		$\mathcal{V}_j^X(0,1),\mathcal{V}_j^X(1,1),\dots,\mathcal{V}_j^X(N-1,1)$ are pairwise disjoint; 
		\item for each $j\in \{1,\dots,\ell\}$, the orbit $\mathrm{\orb}(x)$ is disjoint from  $\mathcal{U}^X(p_j,U_j,N)$;
		\item for each $(k_1,j_1)\neq (k_2,j_2)$ with $k_1,k_2 \in \{0,\dots,2^{d}-1\}$ and $j_1 \in J_{k_1}$, $j_2 \in J_{k_2}$, we have 
		\begin{equation}\label{disj supporttt}
		\mathcal{V}_{j_1}^X(n_0 k_1,n_0) \cap \mathcal{V}^X_{j_2}(n_0 k_2,n_0) = \emptyset.
		\end{equation}
	\end{itemize}
	Indeed, the first item is a consequence of Point \eqref{point quatre} above, the second one follows from Point \eqref{point trois}  above, and the third one is a consequence of Points \eqref{point quatre}-\eqref{point cinq} above. 
	
	\begin{figure}[H]
		\begin{center}
			\includegraphics [width=11cm]{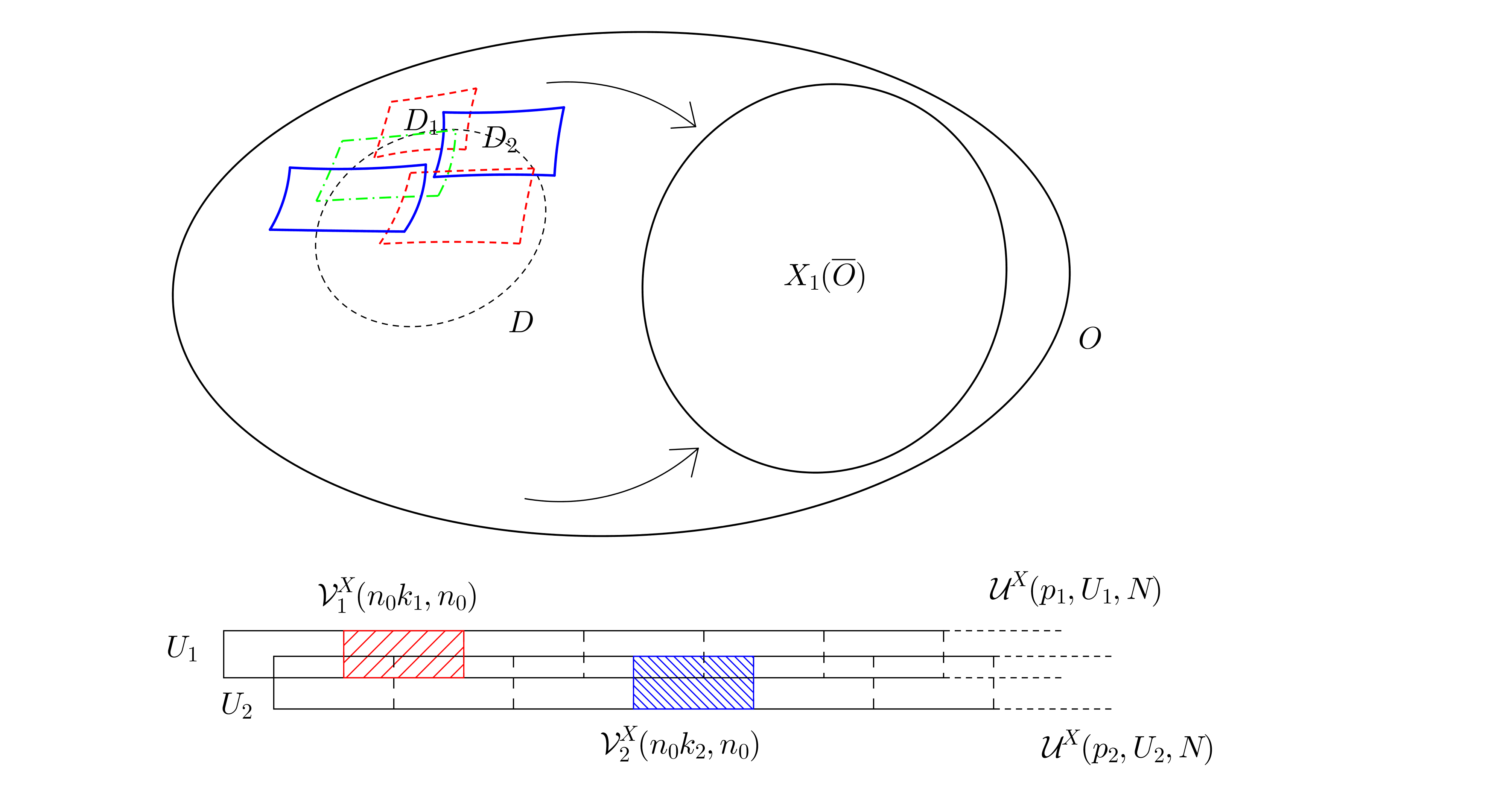}
		\end{center}
		\caption{Selection of the perturbation times for the different tiles.}
	\end{figure}
	
	\begin{claim}\label{claim lebs}
		There exists $\lambda>0$ such that for each $y \in \bigcup_{j=1}^\ell D_j$, there exist $j \in \{1,\dots,\ell\}$, $z\in \Delta_j$ and $u \in [0,1]$ such that $y=X_u(z)$,  and $\mathcal{N}_{X,z}(2\lambda)\subset \Delta_j$. 
	\end{claim}
	
	\begin{proof}
		Let $\lambda_1>0$ be a Lebesgue number of the cover $\mathcal{F}$. We choose  $\lambda_2 >0$ such that $\mathcal{N}_{X,y}(\lambda_2) \subset B(y,\lambda_1)$, for any $y \in \bigcup_{j=1}^\ell D_j$, and we take $\lambda>0$ such that $\mathcal{P}_{z,u}^X(\mathcal{N}_{X,z}(2\lambda))\subset \mathcal{N}_{X,X_u(z)}(\lambda_2)$ for any $z \in \bigcup_{j=1}^\ell \Delta_j$ and $u \in [0,1]$. The existence of $\lambda>0$ follows from the compactness of $\bigcup_{j=1}^\ell \Delta_j$ and from the fact that $X$ is bounded by $C>0$. By the definition of $\lambda_1$ and $D_1,\dots,D_\ell$, for each $y \in \bigcup_{j=1}^\ell D_j$, there exist $j \in \{1,\dots,\ell\}$,  $z\in \Delta_j$, and $u \in [0,1]$ such that $y=X_u(z)$, and $B(y,\lambda_1)\subset D_j$. By the definition of $\lambda_2$,  we also have $\mathcal{N}_{X,y}(\lambda_2) \subset B(y,\lambda_1)$. Then, by the definition of $\lambda$ and $D_j$,  and since $y=X_u(z)\in\mathcal{N}_{X,y}(\lambda_2) \subset D_j$, we deduce that $\mathcal{N}_{X,z}(2\lambda)\subset (\mathcal{P}_{z,u}^X)^{-1}(\mathcal{N}_{X,y}(\lambda_2)) \subset \Delta_j$. 
	\end{proof}
	
	For any $\eta>0$, we define a sequence $(a_\eta(m))_{ m \geq 0}$ inductively as follows:
	$$
	a_\eta(0):=0;\qquad a_\eta(m+1):=2C a_\eta(m)+\eta. 
	$$ 
	Note that $\lim_{\eta \to 0} a_\eta(N)=0$. In the following, 
	we fix $\eta_0>0$ small enough that 
	$$
	a_{\eta_0}(N)<(2C)^{-N} \lambda,\qquad \eta_0<\frac{\lambda}{2}. 
	$$

	For each $k \in \{0,\dots,2^d-1\}$ and $j\in J_k$, the set $\mathcal{P}_{p_j,n_0 k}^{X}(\Delta_j)$ and the point $X_{n_0 k}(x)$ satisfy the hypotheses of Proposition  \ref{prop.proposition8}. We obtain a vector field $\widetilde Y\in \mathfrak{X}^1(M)$ such that the support of $X-\widetilde Y$ is contained in $\overline{\mathcal{V}_j^X(n_0 k,n_0)}$. 
	Moreover, for distinct choices of $(k,j)$,  \eqref{disj supporttt} guarantees that the associated perturbations will be disjointly supported. 
	Hence, applying Proposition  \ref{prop.proposition8}
	over all pairs $(k,j)$ with $k \in \{0,\dots,2^d-1\}$ and $j\in J_k$, we obtain a vector field $Y\in \mathfrak{X}^1(M)$ with the following properties:
	\begin{itemize}
		\item the support of $X-Y$ is contained in 
		\begin{equation*}
		\bigcup_{k=0}^{2^d-1}\bigcup_{j \in J_k}
		\overline{\mathcal{V}_j^X(n_0k,n_0)}
		\subset \bigcup_{j=1}^\ell \mathcal{U}^X( p_j, U_j, N);
		\end{equation*}
		\item $d_{C^1}(X,Y) < \varepsilon$;
		\item $d_{C^1}(\mathcal{P}_{X_i(p_j),1}^X, \mathcal{P}_{X,X_i(p_j),1}^Y) < \delta(\varepsilon)$, for all $i\in \{0, \dots, N\}$ and $j \in \{1,\dots,\ell\}$;
		\item $d_{C^0}(\mathcal{P}_{X_i(p_j),1}^X, \mathcal{P}_{X,X_i(p_j),1}^Y) < \eta_0$, for all $i\in \{0, \dots, N\}$ and $j \in \{1,\dots,\ell\}$;
		\item for each $k \in \{0,\dots,2^d-1\}$ and for each $z \in \bigcup_{j\in J_k} \Delta_j$, there exists an integer $n \in \{1,\dots,n_0\}$ such that:
		$$
		\left|\log \det P_{X_{n_0 k}(x),n}^Y-\log \det P_{\mathcal{P}_{p_j,n_0 k}^{X}(z),n}^Y\right|>3K.
		$$
	\end{itemize}

	\begin{claim}
		For each $y \in \bigcup_{j=1}^\ell D_j$, there exist $k \in \{0,\dots,2^d-1\}$, $j \in J_k$, and $t \in [0,2]$, such that $y=Y_t(w)$, with $w \in \Delta_j$ and   $\mathcal{P}_{X,p_j,n_0 k}^{Y}(w) \in \mathcal{P}_{p_j,n_0 k}^{X}(\Delta_j)$. 
	\end{claim}
	
	\begin{proof}
		Let $y\in \bigcup_{j=1}^\ell D_j$. By Claim \ref{claim lebs}, there exist $j \in \{1,\dots,\ell\}$, $z\in \Delta_j$ and $u \in [0,\frac 32]$ such that $y=\mathcal{P}_{p_j,u}^{X}(z)$,  and $\mathcal{N}_{X,z}(2\lambda)\subset \Delta_j$. 
		We have $d_{C^0}((\mathcal{P}_{p_j,u}^{X})^{-1}, (\mathcal{P}_{X,p_j,u}^{Y})^{-1}) < \eta_0<\frac{\lambda}{2}$, hence $y=\mathcal{P}_{X,p_j,u}^{Y}(w)=Y_t(w)$, for some $t \in [0,2]$, and $w \in \Delta_j$ satisfying $\mathcal{N}_{X,w}(\lambda)\subset \Delta_j$. Moreover, $X$ is bounded by $C$, hence $
		\mathcal{N}_{X,\mathcal{P}_{X,p_j}^i(w)}((2C)^{-i}\lambda)\subset \mathcal{P}_{p_j,i}^X(\Delta_j)$, for all $i \in \{0,\dots,N-1\}$. 
		For any $i \in \{0,\dots,N-1\}$, by \eqref{controle distance p x}, and by the fact that $d_{C^0}( \mathcal{P}_{X_i(p_j),1}^X, \mathcal{P}_{X,X_i(p_j),1}^Y) <\eta_0$, we have the estimate
		\begin{align*}
		d(\mathcal{P}_{p_j,i+1}^X(w),\mathcal{P}_{X,p_j,i+1}^Y(w)) &\leq d(\mathcal{P}_{X_i(p_j),1}^X \circ \mathcal{P}^{X}_{p_j,i}(w),\mathcal{P}_{X_i(p_j),1}^X \circ \mathcal{P}^{Y}_{X,p_j,i}(w))\\
		& + d(\mathcal{P}_{X_i(p_j),1}^X \circ \mathcal{P}_{X,p_j,i}^Y(w),\mathcal{P}_{X,X_i(p_j),1}^Y \circ \mathcal{P}_{X,p_j,i}^Y(w))\\
		&\leq 2C d(\mathcal{P}_{p_j,i}^X(w), \mathcal{P}^{Y}_{X,p_j,i}(w))+ \eta_0.
		\end{align*}
		Thus, for any $i \in \{0,\dots,N-1\}$, we obtain 
		$$
		d(\mathcal{P}_{p_j,i}^X(w),\mathcal{P}^{Y}_{X,p_j,i}(w))\leq a_{\eta_0}(i) <(2C)^{-N}\lambda.
		$$
		Let $k \in \{0,\dots,2^d-1\}$ be such that $j \in J_k$. 
		We conclude that $\mathcal{P}^{Y}_{X,p_j,n_0 k}(w)\in \mathcal{N}_{X,\mathcal{P}_{X,p_j,n_0 k}^Y(w)}((2C)^{-n_0 k}\lambda)\subset \mathcal{P}_{p_j,n_0 k}^{X}(\Delta_j)$, where $Y_{t}(w)=y$. 
	\end{proof}
	We deduce that for each $y \in D \subset\cup_{j=1}^\ell D_j$, there exist $k \in \{0,\dots,2^d-1\}$, $j \in J_k$, $w \in \Delta_j$, $t \in [0,2]$, such that $y=Y_t(w)$, and there exists $n \in \{1,\dots,n_0\}$, such that 
	$$
	|\log \det P_{X_{n_0 k}(x),n}^Y-\log \det P^Y_{\mathcal{P}_{X,p_j,n_0 k}^Y(w),n}|>3K.
	$$
	Since the vector field $X-Y$ has support in $\bigcup_{j=1}^\ell \mathcal{U}^X( p_j, U_j, N)$, which is disjoint from the orbit $\mathrm{\orb}(x)$, we have  $X_{n_0 k}(x)=Y_{n_0 k}(x)$.  Moreover,  there exists $t' \in [n_0 k-2,n_0k+2]$  such that $\mathcal{P}_{X,p_j,n_0 k}^{Y}(w)=Y_{t'}(y)$. We thus have
	$$
	|\log \det P^Y_{Y_{n_0 k}(x),n}-\log \det P_{Y_{t'}(y),n}^Y|>3K.
	$$
	We have $P_{Y_{t'}(y),n}^Y=P_{Y_{n_0 k}(y),n}^Y\circ P^{Y}_{Y_{t'}(y),n_0k -t'}$, with $n_0k -t' \in [-2,2]$. Recall that $K>0$ was  chosen such that $K> 2d \log(2C)$. Since  $Y$ is  close to $X$, we can assume that $Y$ is bounded by $2C$. We thus get
	\begin{align*}
	&|\log \det P_{Y_{n_0 k}(x),n}^Y-\log \det P^Y_{Y_{n_0 k}(y),n}|\\
	\geq\ &|\log \det P^Y_{Y_{n_0 k}(x),n}-\log \det P^Y_{Y_{t'}(y),n}|-\max_{u' \in [-2,2]} \max_{y' \in M_X}|\log \det  P_{y',u'}^Y|\\
	>\  &3K- 2d \log (2C) >2K.
	\end{align*}
	Besides, $P^{Y}_{z,n+n_0 k}=P^{Y}_{Y_{n_0 k}(z),n}\circ P^{Y}_{z,n_0 k}$, hence of the following two cases holds:
	\begin{itemize}
		\item $|\log \det P^{Y}_{x,n_0 k}-\log \det P^{Y}_{y,n_0 k}|>K$;
		\item $|\log \det P^{Y}_{x,n+n_0 k}-\log \det P^{Y}_{y,n+n_0 k}|>K$.
	\end{itemize}
	In either case, 
	$|\log \det P^{Y}_{x,n'}-\log \det P^{Y}_{y,n'}|>K$, for some $n'\in \{1,\dots,N\}$, as required. 
	
	By construction, the support of $X-Y$ is contained in at most $N$ iterates of $O - X_1(\overline{O})$ for some trapping region $O$, and thus, the iterates of $O - X_1(\overline{O})$ for $X$ and $Y$ coincide. This implies that the vector fields $X$ and $Y$ have the same chain recurrent set, and they coincide on this set, which concludes the proof. 
\end{proof}

\subsubsection{Proof of Theorem \ref{t.unboundednormalgeneric}}
Let $\mathcal{F}$ be a countable and dense subset of $M$, and let $\mathcal{K} = \{D_n\}_{n\in \N}$ be a countable collection of compact sets $D_n$, that verifies the following conditions:
\begin{itemize}
	\item[--] $\mathrm{diam} D_n \to 0$, as $n\to +\infty$;
	\item[--] for any $n_0 \geq 1$, it holds $\displaystyle \bigcup_{n\geq n_0} D_n = M$.
\end{itemize} 

For each $D \in \mathcal{K}$ we define the following set
\[
\mathcal{O}_{D} := \{ X\in \mathfrak{X}^1(M): \exists \textrm{ open set $U$, } X_1(\overline{U}) \subset U \textrm{ and } D \subset (U-X_1(\overline{U}))\}.
\]
It is easy to see that $\mathcal{O}_{D}$ is open. For any point $x\in \mathcal{F}$ we define
\[
\mathcal{U}_{x,D} = \{ X\in \mathcal{O}_{D} : x\notin \mathrm{Zero}(X) \textrm{ and } \orb(x) \cap D = \emptyset\}.
\]
For a vector field $X\in \mathcal{O}_D$ such that $x\in M_X$, it could happen that the orbit $\orb(x)$ is non compact and it could accumulate on the compact set $D$ without intersecting it. In this case, by a small pertubation we could make the orbit of $x$ intersect $D$. For instance if $D$ is a sink and $x$ is a point in its basin. In particular, the set $\mathcal{U}_{x,D}$ is not open. 
	\begin{figure}[H]
		\begin{center}
			\includegraphics [width=4cm]{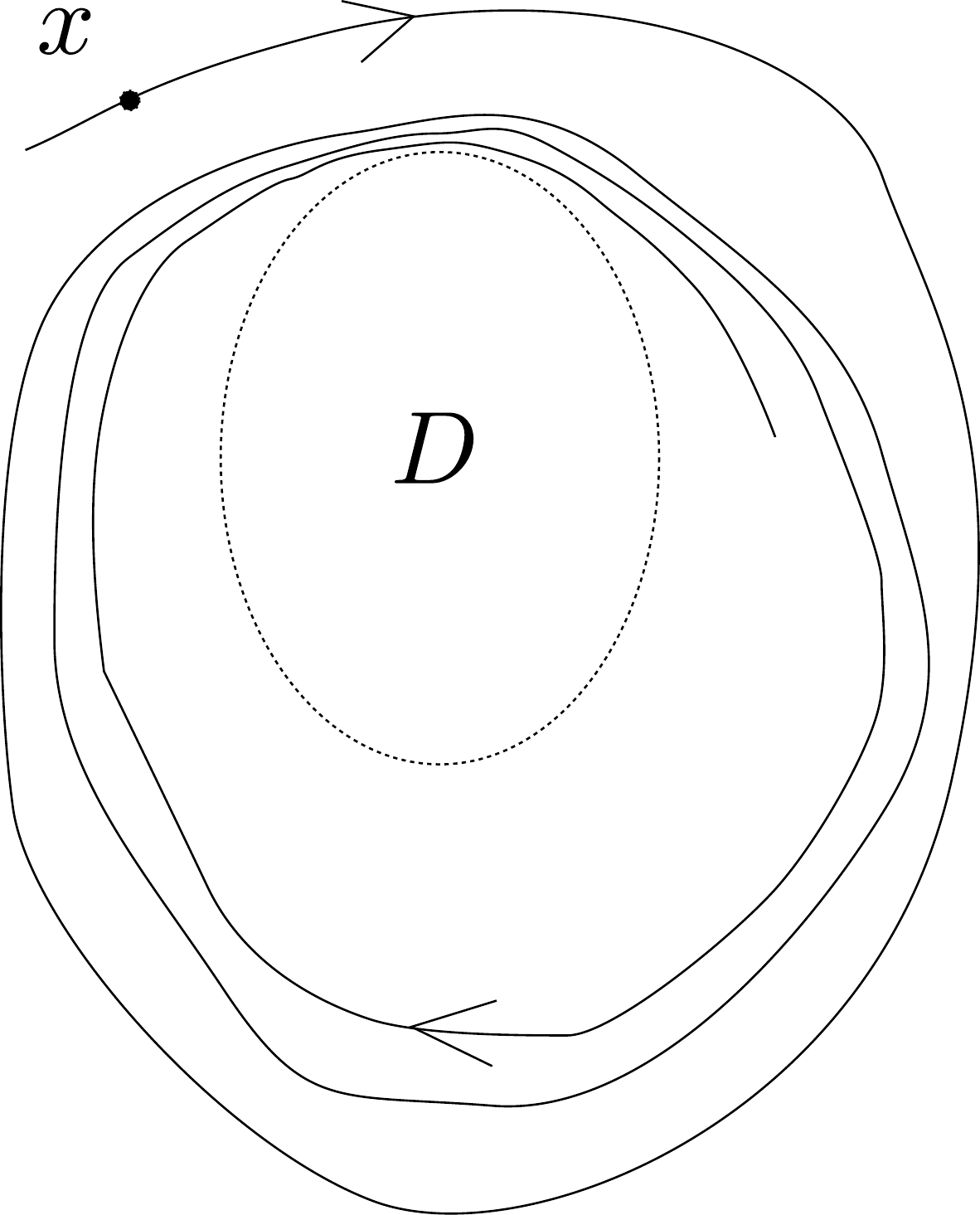}
		\end{center}
		\caption{The orbit $\orb(x)$   accumulates the compact set $D$.}\label{spirale}
	\end{figure}
	
	The next lemma gives a criterion to know that a vector field $X$ belongs to the interior of $\mathcal{U}_{x,D}$.

\begin{lemma}
	\label{lem.criterioninterion}
	Let $X\in \mathcal{U}_{x,D}$ and let $U\subset M$ be an open subset such that $X_1(\overline{U}) \subset U$ and $D \subset (U-X_1(\overline{U}))$. Assume that $\orb(x) \cap (U - X_1(\overline{U})) \neq \emptyset$. Then $X$ belongs to the interior of $\mathcal{U}_{x,D}$, in particular, for any $Y\in \mathfrak{X}^1(M)$ sufficiently close to $X$ it holds that $\mathrm{orb}^Y(x) \cap D = \emptyset$.
\end{lemma}
\begin{proof}
	Observe that the conditions $X_1(\overline{U}) \subset U$ and $D \subset (U - X_1( \overline{U}))$ are open. If $\orb(x) \cap U \neq \emptyset$, we can fix $t_1, t_2 \in \mathbb{R}$ such that $\left(\orb(x) \cap U - X_1( \overline{U})\right) \subsetneq X_{(t_1, t_2)}(x)$. We can also assume that this property is open, that is, for any $C^1$ vector field $Y$ sufficiently $C^1$-close to $X$, it holds 
	\[
	\mathrm{orb}^Y(x) \cap (U - Y_1(\overline{U})) \subsetneq Y_{[t_1,t_2]}(x).
	\]
	
	Since $D$ and $X_{[t_1,t_2]}(x)$ are compact and disjoint, the distance between them is strictly positive. This implies that for any $Y$ sufficiently $C^1$-close to $X$ it holds that $Y_{[t_1, t_2]}(x)$ does not intersect $D$. Since $\mathrm{orb}^Y(x) \cap (U - Y_1(\overline{U})) \subset Y_{[t_1,t_2]}(x)$, we conclude that $\mathrm{orb}^Y(x) \cap D = \emptyset$. In particular, $X$ belongs to the interior of $\mathcal{U}_{x,D}$.
\end{proof}

The proof of the following lemma is the same as Lemma $15$ in \cite{BonattiCrovisierWilkinson}.

\begin{lemma}
	\label{lem.opendense}
	The set $\mathrm{Int}(\mathcal{U}_{x,D}) \cup \mathrm{Int}(\mathcal{O}_{D} - \mathcal{U}_{x,D})$ is open and dense inside $\mathcal{O}_{D}$.
\end{lemma}
First, observe that if $X\in \mathcal{U}_{x,D}$ then $ D \cup \{x\}$ does not have any singularity of $X$. In particular, the linear Poincar\'e flow is well defined for any point $y\in D \cup \{x\}$. For $x\in \mathcal{F}$, $D \in \mathcal{K}$ and any $K \in \N$, we define:
\[
\mathcal{V}_{x,D, K} := \left\{X\in \mathrm{Int}(\mathcal{U}_{x,D}): \forall\, y \in D, \textrm{ } \exists\, n\geq 1,\ |\log \det P^X_{y,n} - \log\det P^X_{x,n}| > K\right\}.
\]
Using the fact that $D$ is compact, it is easy to see that $\mathcal{V}_{x,D,K}$ is open inside $\mathrm{Int}(\mathcal{U}_{x,D})$. Proposition \ref{prop sept} implies that $\mathcal{V}_{x,D,K}$ is dense in $\mathrm{Int}( \mathcal{U}_{x,D})$. Therefore, the set
\[
\mathcal{W}_{x,D, K} = \mathcal{V}_{x,D,K} \cup \mathrm{Int}(\mathcal{O}_{D} - \mathcal{U}_{x,D}) \cup \mathrm{Int}(\mathfrak{X}^1(M) - \mathcal{O}_{D})
\]
is open and dense in $\mathfrak{X}^1(M)$. Define the set
\[
\mathcal{R}_0 := \displaystyle \bigcap_{x\in \mathcal{F}, D \in \mathcal{K}, K \in \N } \mathcal{W}_{x,D, K}.
\]
By Baire's theorem, this set is residual in $\mathfrak{X}^1(M)$. Let $\mathcal{R} = \mathcal{R}_0 \cap \mathcal{R}_*$, where $\mathcal{R}_*$ is the residual subset given by Theorem \ref{thm.genericbackground}.

Let $X\in \mathcal{R}$. Consider $x\in \mathcal{F}-\mathrm{Zero}(X)$ and $y\in M- \mathrm{CR}(X)$ such that $y\notin \mathrm{orb}^X(x)$. Since $y \notin \mathrm{CR}(X)$, by Conley's theory there exists an open set $U\subset M$ such that $X_1(\overline{U}) \subset U$ and $y\in (U-X_1(\overline{U}))$. This is a direct consequence of the existence of \textit{complete Lyapunov functions} for $X$, which is given by the so called \textit{Fundamental Theorem of dynamical systems} (see Definition $4.7.1$, for the definition of complete Lyapunov function, and Theorem $4.8.1$ in \cite{AlongiNelson}). Observe also that $\mathrm{orb}^X(x) \cap (U - X_1(\overline{U}))$ is either empty or a compact orbit segment. Take $D \in \mathcal{K}$ a compact set that contains $y$. If its diameter is sufficiently small, we have that $D \subset (U - X_1(\overline{U})$ and $\mathrm{orb}^X(x) \cap D = \emptyset $.

Hence $X\in \mathcal{U}_{x,D}$. Since $X\in \mathcal{R}_0$ and by the definition of $\mathcal{R}_0$, for every $K\in \N$, it holds that $X\in \mathcal{W}_{x,D,K}$. By the definition of $\mathcal{W}_{x,D, K}$ and since $X\in \mathcal{U}_{x,D}$, we have that $X\in \mathcal{V}_{x,D, K}$. Therefore, for any $K \in \N$, there exists $n\geq 1$ such that 
\[
|\log \det P^X_{x,n} - \log \det P^X_{y,n} | > K.
\]  
We conclude that $X$ verifies the unbounded normal distortion property.
\qed

%
%

\begin{appendix}
\section{The separating property is not generic}\label{append sep}

In this section we prove that the separating property is not generic. Indeed we will see that it is not even $C^1$-dense. Let $M$ be a compact, connected Riemannian manifold. Take any Morse function  $f\in C^2(M,\mathbb{R})$  and let $X := \nabla f$ be the gradient vector field which is $C^1$. It holds that $X$ has two hyperbolic singularities, $\sigma_s$ and $\sigma_u$ with the following properties:
\begin{itemize}
\item $\sigma_s$ is a hyperbolic sink and $\sigma_u$ is a hyperbolic source;
\item $W^s(\sigma_s) \cap W^u(\sigma_u) \neq \emptyset$;
\item for any $C^1$ vector field $Y$ which is sufficiently $C^1$-close to $X$, then $W^s(\sigma_s(Y)) \cap W^u(\sigma_u(Y)) \neq \emptyset$, where $\sigma_*(Y)$ is the continuation of $\sigma_*$ for the vector field $Y$, for $*=s,u$.
\end{itemize}

We claim that $X$ is $C^1$-robustly not separating. Let $U$ be a compact ball inside $\left(W^s(\sigma_s) \cap W^u(\sigma_u)\right)-\{\sigma_s , \sigma_u \} $. Since compact parts of stable and unstable manifolds vary continuously with the vector field, it holds for any $Y$ sufficiently $C^1$-close to $X$ it holds that $U \subset \left( W^s(\sigma_s(Y)) \cap W^u(\sigma_u(Y)) \right) - \{\sigma_s, B_u \}$.  

Take any $\varepsilon>0$ and consider the balls $B(\sigma_s , \frac{\varepsilon}{2})$ and $B(\sigma_u, \frac{\varepsilon}{2})$. Since $U$ is compact, there exists $T_X= T(\varepsilon)>0$ such that any point $x\in U$ verifies
\begin{equation}
\label{eq.nonseparating}
X_{-t}(x) \in B\left(\sigma_u, \frac{\varepsilon}{2}\right) \textrm{ and } X_t(x) \in B\left(\sigma_s, \frac{\varepsilon}{2}\right), \textrm{ for all $t\geq T$.}
\end{equation}

Notice that for any two points $x,y\in B(\sigma_s, \frac{\varepsilon}{2})$ it holds that $d(X_t(x), X_t(y)) < \varepsilon$, for all $t\geq 0$. Similar statement is true for points in $B(\sigma_u, \frac{\varepsilon}{2})$ and the backward orbit.

Since $T$ that verifies (\ref{eq.nonseparating}) is fixed, there exists $\delta>0$ such that for any $x\in U$ and any $y\in B(x, \delta)\subset U$, it holds that 
\[
d(X_t(x), X_t(y))< \varepsilon, \textrm{ for any $t\in \mathbb{R}$.}
\]
In particular $X$ is not separating. Also, observe that this holds for any $Y$ sufficiently $C^1$-close to $X$. Thus we conclude that $X$ is $C^1$-robustly not separating.

\begin{remark}
It is easy to see that the same type of example proves that the hypothesis of Proposition \ref{thm.expcollinear} is not even $C^1$-dense. We conclude that the hypotheses of Propositions \ref{premier theorem} and \ref{thm.expcollinear} are not $C^1$-dense as well. 
\end{remark}

\section{Periodic orbits and chain-recurrent classes for flows: sketch of the proof of Item \eqref{iitem 3} in Theorem \ref{thm.genericbackground}}\label{append.croflow}

In this appendix, we briefly explain the structure of the proof of Item \eqref{iitem 3} of Theorem \ref{thm.genericbackground} for diffeomorphisms, which was proved by Crovisier in \cite{Crovisier}. Then we explain why the same proof works for vector fields as well.

 As we will see, the proof of this result follows easily from Theorem \ref{thm.thm3} below, which is an easy consequence of Proposition \ref{prop.propcro} below. We emphasize that the proof of Proposition \ref{prop.propcro} has two parts: a perturbation part, which only uses Hayashi's connecting lemma, and a combinatorial part, which has no perturbation at all. Usually in these pertubation lemmas such as the connecting lemma, closing lemma and others, the combinatorial part is the hardest part in the proof. The idea of the combinatorial part is to find the right pieces of orbits that you will connect by some elementary $C^1$-perturbation. The same happens in the proof of Item \eqref{iitem 3} of Theorem \ref{thm.genericbackground}. There is the combinatorial argument that will give which pieces of orbits are good to connect, and instead of some elementary perturbation you use Hayashi's connecting lemma to connect these pieces of orbits. We remark that the combinatorial argument will be the same for diffeomorphisms and flows, and that Hayashi's connecting lemma is available for flows, so the perturbative tool is also available for flows.

Let us give some of the dynamical background used in the proof. Below, we follow the notation in \cite{Crovisier}. 

Let $M$ be a compact, connected, Riemannian manifold, and let $\mathrm{Diff}^1(M)$ be the set of $C^1$-diffeomorphisms of $M$. A diffeomorphism $f\in \mathrm{Diff}^1(M)$ verifies \textit{Condition (A)} if for any $n\in \N$, every periodic orbit of period $n$ is isolated  in $M$. Observe that by Kupka-Smale's Theorem (see Theorem $3.1$ in \cite{PalisdeMelo}), Condition (A) is a $C^r$-generic condition.

Let $f\in \mathrm{Diff}^1(M)$. We say that a compact and invariant set $\mathcal{X}$ is \textit{weakly transitive} if for any non-empty open sets $U$ and $V$ that intersect $\mathcal{X}$, and any neighborhood $W$ of $\mathcal{X}$, there exists a segment of orbit $\{x, f(x), \cdots, f^n(x)\}$ contained in $W$ and such that $x$ belongs to $U$ and $f^n(x)$ belongs to $V$, and $n\geq 1$.

The main perturbation technique we want to describe is given by Theorem $3$ from \cite{Crovisier} which states the following:

\begin{theorem}[Theorem $3$ in \cite{Crovisier}]\label{thm.thm3}
Let $f$ be a  $C^1$ diffeomorphism that satisfies Condition (A), let $\mathcal{U}$ be a $C^1$-neighborhood of $f$ in $\mathrm{Diff}^1(M)$ and let $\mathcal{X}$ be a weakly transitive set of $f$. Then, for any $\eta>0$, there exists $g\in \mathcal{U}$ and a periodic orbit $\mathcal{O}$ of $g$ such that $\mathcal{O}$ is $\eta$-close to $\mathcal{X}$ for the Hausdorff distance.
\end{theorem}

Theorem \ref{thm.thm3} is the main perturbation technique used to prove Item \eqref{iitem 3} in Theorem \ref{thm.genericbackground} (which is given by Theorem $4$ in \cite{Crovisier}). The key ingredient in the proof of this theorem  is the following proposition:

\begin{proposition}[Proposition $8$ in \cite{Crovisier}] \label{prop.propcro}
Let $f$ be a $C^1$ diffeomorphism and $\mathcal{U}$ a $C^1$-neighborhood of $f$. Then, there exists $N\geq 1$ with the following property: \\ \indent if $W\subset M$ is an open set and $\mathcal{X}$ a finite set of points inside $W$ such that:
\begin{enumerate}
\item the points $f^j(x)$ for $j\in \{1, \cdots, N\}$ are two-by-two distinct and contained in $W$, for $x\in \mathcal{X}$;
\item for any two points $x,x' \in \mathcal{X}$, for any neighborhood $U$ and $V$ of $x$ and $x'$, respectively, there is a point $z\in U$ and $f^n(z) \in V$, with $n>1$, such that $\{z, \cdots, f^n(z)\} \subset W$;
\end{enumerate}
then, for any $\eta>0$ there exists a perturbation $g \in \mathcal{U}$ of $f$ with support in the union of the open sets $f^j(B(x,\eta))$, for $x\in \mathcal{X}$ and $j\in \{0,\cdots, N-1\}$, and a periodic orbit $\mathcal{O}$ of $g$ contained in $W$, which crosses all the balls $B(x,\eta)$, for $x\in \mathcal{X}$.  
\end{proposition} 

The proof of Theorem \ref{thm.thm3} using Proposition \ref{prop.propcro} follows from a short argument. Let us explain the main steps of the argument. Suppose that $f$ is a diffeomorphism verifying Condition (A) and $\mathcal{X}$ is a weakly transitive set. We may suppose also that $\mathcal{X}$ is not a periodic orbit, otherwise there would be nothing to prove. We use Condition (A) to find a finite set $\hat{X} \subset \mathcal{X}$ such that any point $x \in \hat{X}$ is not a periodic point, any two different points in $\hat{X}$ have disjoint orbits, and any point $z\in \mathcal{X}$ belongs to $B(f^k(x),\eta_0)$, for some $k\in \Z$ and $x\in \hat{X}$. We then can fix some $\eta\in (0,\eta_0)$ to be small enough such that any compact invariant set $K$ which is contained in a $\eta_0$-neighborhood of $\mathcal{X}$ and intersecting all the balls $B(x,\eta)$ for $x\in \hat{X}$, is $\eta_0$-close to $\hat{X}$ in the Hausdorff topology. One can then apply Proposition \ref{prop.propcro} and obtain a periodic orbit that verifies the conclusion of the theorem (see Section $2.4$ in \cite{Crovisier}).

Let us now explain the structure of the proof of Proposition \ref{prop.propcro}. We also refer the reader to Section $4.0$ in \cite{Crovisier}, where the structure and difficulties  of the proof of Proposition \ref{prop.propcro} are explained very clearly. The proof has two parts: the actual perturbation part, which only uses Hayashi's connecting lemma; and a combinatorial part, which is the most delicate part.

 Let us recall Hayashi's connecting lemma. The original proof was given by Hayashi in \cite{Hayashi}. Some other references are given in \cite{Arnaud, WenXia, BonattiCrovisier}.

\begin{theorem}[Hayashi's connecting lemma, \cite{Hayashi, Arnaud, WenXia, BonattiCrovisier}]
Let $f_0$ be a diffeomorphism of a compact manifold $M$, and $\mathcal{U}$ a $C^1$-neighborhood of $f_0$. Then there exists $N\geq 1$ such that for any $z\in M$ which is not a periodic point of period less than or equal to $N$, any two open neighborhoods $V$ and $U$ of $z$ such that $V\subset U$ has the following property. 

For any diffeomorphism $f$ that coincides with $f_0$ in $U\cup \cdots \cup f^{N-1}_0(U)$, for any two points $p,q\in M - \left(U \cup \cdots \cup f^{N}_0(U)\right)$ and any integers $n_p, n_q \geq 1$ such that $f^{n_p}(p)$ belongs to $V$ and $f^{-n_q}(q) \in V$ there is a diffeomorphism $g$ in $\mathcal{U}$ such that:
\begin{itemize}
\item $g$ coincides with $f$ on $M - \left(U \cup \cdots \cup f^{N}_0(U)\right)$;
\item there exists $m\geq 1$ such that $g^m(p) = q$;
\item the piece of orbit $\{p, \cdots, g^m(p)\}$ can be cut into three parts:
\begin{itemize}
\item the beginning $\{p, \cdots, g^{m'}(p)\}$, for some $m' \in \N$, is contained in 
\[
\{p, \cdots, f^{n_p}(p)\} \cup U \cup \cdots \cup f^N_0(U);
\]
\item the central part $\{g^{m'}(p) , \cdots, g^{m'+N}(p)\}$ is contained in 
\[
U \cup \cdots \cup f_0^N(U);
\]
\item the end $\{g^{m'+N}(p), \cdots, g^m(p)\}$ is contained in 
\[
U \cup \cdots \cup f_0^N(U) \cup \{f^{-n_q}(q), \cdots, q\}.
\]
\end{itemize}
\end{itemize} 
\end{theorem}

Let $\mathcal{X}$ be a finite set which is weakly transitive, and let $W$ be a neighborhood of $\mathcal{X}$, and fix some $\eta>0$. Fix some order in $\mathcal{X}= \{p_1, \cdots, p_k\}$ and for each $p_i$ we can associate two neighborhoods $V_i\subset U_i$ contained in $W$, where we will apply the connecting lemma. Observe that we may consider these neighborhoods to be arbitrarily small. The weak transitivity implies that for each $i$, there is a point $z_i$ in $V_i$ whose future orbit intersects $V_{i+1}$ and it is contained in $W$. The first naive approach then would be to simply use the connecting lemma, and connect the orbit of each $z_i$ with $z_{i+1}$, where one would connect the future orbit of $z_k$ with $z_0$. Applying the connecting lemma $k$ times, we would hope to have created a periodic orbit for a diffeomorphism $g\in \mathcal{U}$, and this periodic orbit contained in $W$. This is a perturbation of $f$ whose support is contained in the union of the first $N-1$ iterates of $U_i$, for every $i=1, \cdots, k$. This approach would only work if we had the following ``ideal'' picture (see Figure \ref{figure.below} below).

\begin{figure}[H]
\begin{center}
    \includegraphics [width=9cm]{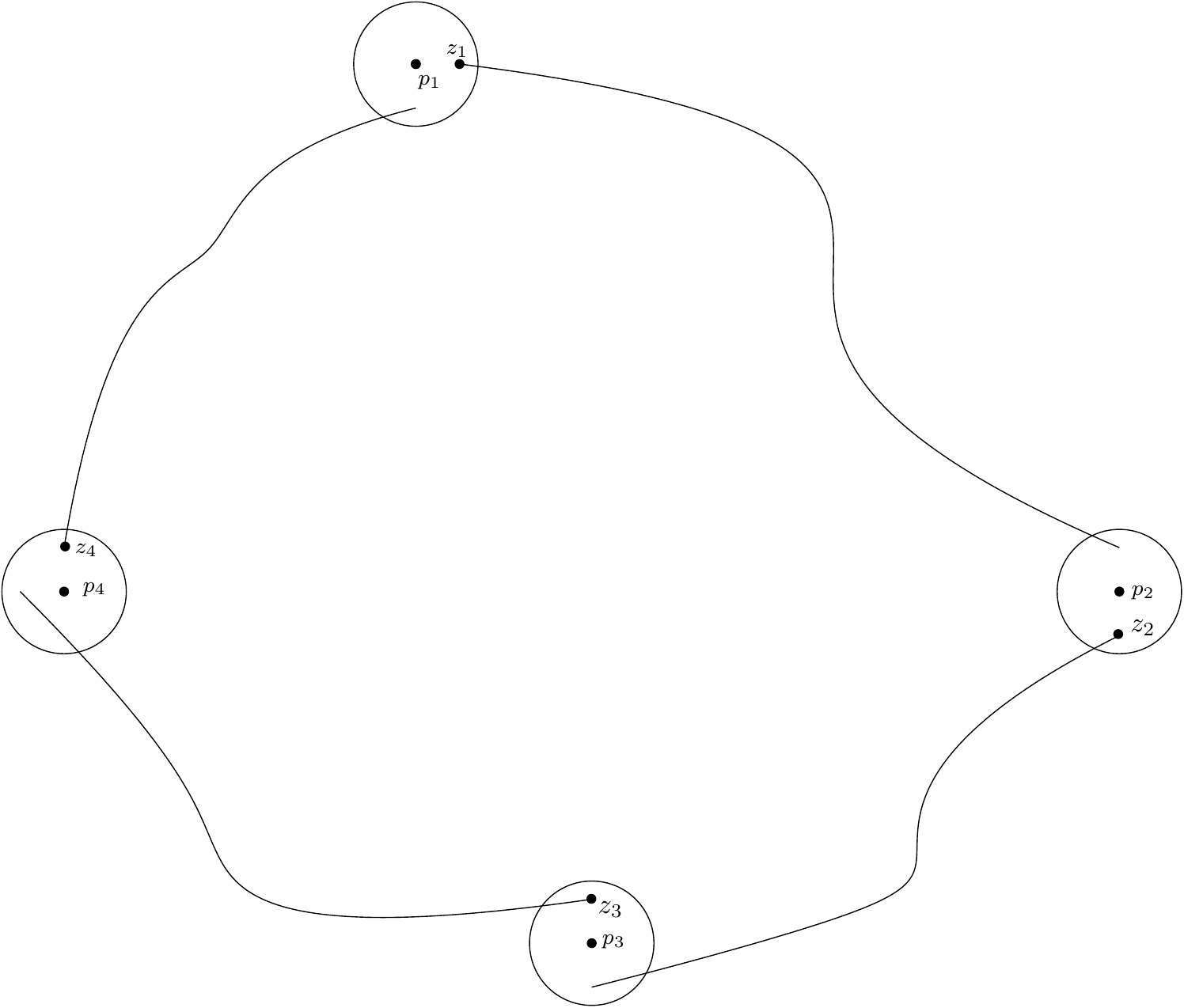}
\end{center}
\caption{``Ideal'' picture.}\label{figure.below}
\end{figure}

The problem is that with this approach one cannot guarantee that the piece of future orbit of $z_i$ connecting $V_i$ and $V_{i+1}$ does not intersect any other $U_j$, for $j\neq i,i+1$ (see  Figure \ref{figure.below 2} below).
\begin{figure}[H]
	\begin{center}
		\includegraphics [width=7cm]{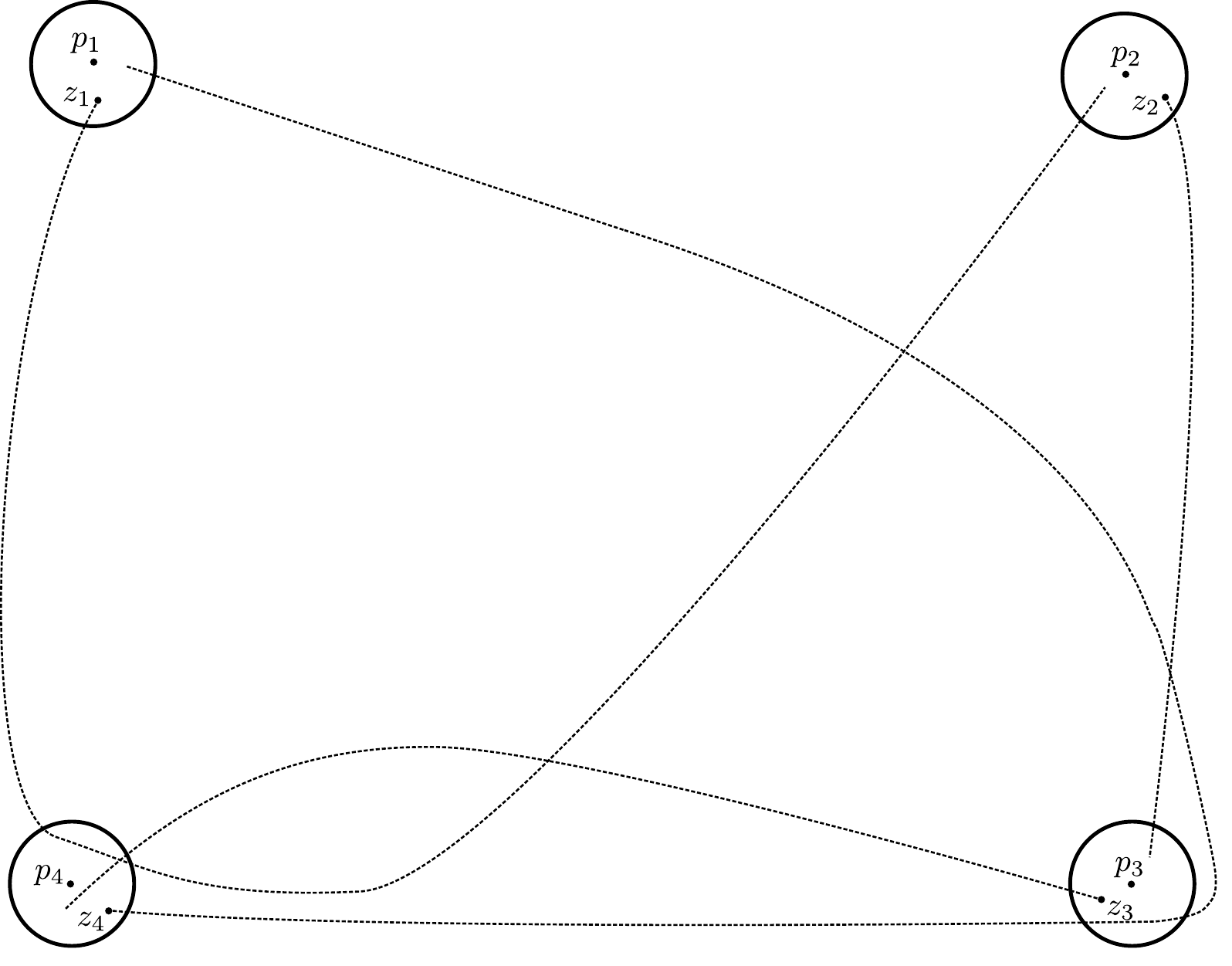}
	\end{center}
	\caption{``General'' picture.}\label{figure.below 2}
\end{figure} 

There is a delicate inductive and ``combinatorial'' argument that allows one to find a smaller set $\mathcal{X}' \subset \mathcal{X}$, with $\mathcal{X}' = \{p_1', \cdots, p_s'\}$, neighborhoods for the connecting lemma $V_i' \subset U_i'$, and points $\{z_1, \cdots, z_s\}$ that verify  the following: each $z_i$ connects $V'_i$ to $V'_{i+1}$, the piece of future orbit of $z_i$ connecting these two neighborhoods does not intersect any other $U'_j$, for $j\neq i, i+1$, also the union of these pieces of future orbits of the $z_i$'s intersects every $B(p, \eta)$, for every $p\in \mathcal{X}$. Hence, one can apply Hayashi's connecting lemma around the points of the set $\mathcal{X}'$, connecting the orbits of the points $z_i$ mentioned above, and obtain a perturbation $g$ of $f$ with a periodic orbit that verifies the conclusion of Proposition \ref{prop.propcro}. We refer the reader to Section $4.0$ in \cite{Crovisier} for more details.  

There is no perturbation in the combinatorial part, even though it is the most delicate part. Observe that the points in $\mathcal{X}$ are not fixed points, and if they are periodic they must have the period larger than $N$. 

Since Hayashi's connecting lemma is also available for flows \cite{Hayashi}, and the points where we are using the connecting lemma are far from the singularities (it is a finite fixed set of non-singular points), one can obtain the following result for flows:
\begin{proposition}
Let $X$ be a $C^1$ vector field on $M$ and let $\mathcal{U}$ be a $C^1$-neighborhood of $X$. There exists $T\geq 1$ with the following property: if $W\subset M$ is an open set and $\mathcal{X}$ is a finite set of points inside $W$ such that:
\begin{enumerate}
\item for each $x\in \mathcal{X}$ the map $t\mapsto X_t(x)$ is injective for $t\in [0,T]$ and $X_{[0,T]}(x)$ is contained in $W$;
\item for any two points $x,x' \in \mathcal{X}$, for any neighborhood $U$ and $V$ of $x$ and $x'$, respectively, there is a point $z\in U$ and $T'>1$ such that $X_{T'}(z) \in V$ and $X_{0,T'}(z) $ is contained in $W$;
\end{enumerate}
then, for any $\eta>0$ there exists a perturbation $Y$ of $X$ in $\mathcal{U}$ with support in the union of the open sets $\cup_{t\in [0,T-1]}X_t(B(x,\eta))$, with $x\in \mathcal{X}$, and a periodic orbit $\mathcal{O}$ for $Y$ contained in $W$ and intersects all the balls $B(x,\eta)$, for $x\in \mathcal{X}$.
\end{proposition}

Obseve that for flows we can also define an analogous of Condition (A), which we will call (Af). We say that a vector field verifies Condition (Af) if for any $T > 0$ every periodic orbit of period less than or equal to $T$ is isolated in $M$. As explained above, this proposition implies the following theorem:

\begin{theorem}
Let $X$ be a $C^1$ vector field verifying Condition (Af), let $\mathcal{U}$ be a $C^1$-neighborhood of $X$ and let $\mathcal{X}$ be a weakly transitive set for $X$. Then, for any $\eta>0$, there exists $Y$ in $\mathcal{U}$ and a periodic orbit $\mathcal{O}$ for $Y$ which is $\eta$-close to $\mathcal{X}$ for the Hausdorff topology.
\end{theorem}

Below, we explain how Crovisier uses Theorem \ref{thm.thm3} to conclude Item \eqref{iitem 3} in Theorem \ref{thm.genericbackground}. The same proof works for flows.

First, given a metric space $H$ we denote by $\mathcal{K}(H)$ the set of all compact sets of $H$ with the Hausdorff distance.  For each diffeomorphism $f$ we let $K_{per}(f)$ be the closure in $\mathcal{K}(M)$ of the set of periodic points of $f$. This is a set in $\mathcal{K}(\mathcal{K}(M))$. It is known that the the points of continuity of the function $f\mapsto K_{per}(f)$ contains a $C^1$-residual subset of $\mathrm{Diff}^1(M)$, see for instance \cite{Takens}. Let us denote this residual set by $\mathcal{R}_1$. 

For each diffeomorphism $f$, consider also $K_{WT}(f)$ to be the closure in $\mathcal{K}(M)$ of the set of weak transitive sets of $f$. Also, let $K_{CT}(f)$ to be the closure in $\mathcal{K}(M)$ of the set of \textit{chain transitive} sets of $f$. Both of these sets belong to $\mathcal{K}(\mathcal{K}(M))$. It is known that there exists a $C^1$-residual set $\mathcal{R}_2$ such that for any $f\in \mathcal{R}_2$ we have $K_{WT}(f) = K_{CT}(f)$. 

We claim that if $f\in \mathcal{R}_1 \cap \mathcal{R}_2$ then $K_{per}(f) = K_{WT}(f) = K_{CT}(f)$. Suppose that $K_{per}(f) \subsetneq K_{WT}(f)$. Then there exists a weak transitive set $\mathcal{X}$ and some open neighborhood $W$ of $\mathcal{X}$ in $\mathcal{K}(M)$, such that $K_{per}(f)$ does not intersect $W$. By the continuity of the function $f\mapsto K_{per}(f)$, at the point $f$, we know that this property holds in a neighborhood of $f$. However, Theorem \ref{thm.thm3} implies that there is a diffeomorphism $g$ which is $C^1$-close to $f$ such that $g$ has a periodic orbit $\mathcal{O}$ which is close in the Hausdorff distance to $\mathcal{X}$. This is a contradiction.

This implies that $C^1$-generically any chain transitive set is the Hausdorff limit of periodic orbits. Since a chain recurrent class is a chain transitive set, we conclude Item \eqref{iitem 3} of Theorem \ref{thm.genericbackground} for diffeomorphisms. We remark that the same proof also holds for flows. As we mentioned before, the main perturbative tool used is given by Theorem \ref{thm.thm3}.

\end{appendix}

\bibliographystyle{alpha}

\information

\end{document}